\documentclass[11pt]{amsart}

\usepackage{amssymb,latexsym}
\usepackage{amsthm}
\usepackage{amsmath}
\usepackage{amsfonts}
\usepackage{mathrsfs}
\usepackage{mathdots}
\usepackage{graphicx}
\usepackage[all]{xy}
\usepackage{fancyhdr}
\usepackage[top=1in,bottom=1in,left=1.25in,right=1.25in, marginparwidth=1in,]{geometry}

\usepackage[mathscr]{eucal}

\usepackage[title]{appendix}

\usepackage{tikz-cd}
\usepackage{stmaryrd}

\numberwithin{equation}{subsection}
\newtheorem{theorem}{Theorem}[section]
\newtheorem*{itheorem}{Theorem}

\newtheorem{definition}[theorem]{Definition}
\newtheorem{remark}[theorem]{Remark}

\newtheorem{proposition}[theorem]{Proposition}
\newtheorem{corollary}[theorem]{Corollary}
\newtheorem{lemma}[theorem]{Lemma}
\newtheorem{example}[theorem]{Example}

\def\Q{\mathbb{Q}}

\def\R{\mathbb{R}}
\def\Z{\mathbb{Z}}

\def\E{\mathcal{E}}

\def\cO{\mathcal{O}}

\def\Fil{\mathrm{Fil}}

\def\Spa{\mathrm{Spa}}

\def\Spec{\mathrm{Spec}}

\def\deg{\mathrm{deg}}

\def\ra{\rightarrow}

\def\Qp{\mathbb{Q}_p}

\def\Fbar{\bar{F}}

\def\BdR{B_{\mathrm{dR}}}
\def\BpdR{B_{\mathrm{dR}}^+}

\newcommand{\mar}[1]{\marginpar{\tiny #1}}

\renewcommand\appendix{\par \setcounter{section}{0} \setcounter{subsection}{0} \gdef\thesection{ \Alph{section}}}

\begin{document}

\title{On the weak Harder-Narasimhan stratification on $\BpdR$-affine Grassmannian}
\author{Miaofen Chen, Jilong Tong}
\date{}

\address{Shanghai Center for Mathematical Sciences\\
Fudan University\\ 2005 Songhu Road\\Shanghai 20438, China}\email{miaofenchen@fudan.edu.cn}

\address{School of Mathematical Sciences \\ Capital Normal University \\ 105, Xi San Huan Bei Lu \\ Beijing 100048, China}\email{jilong.tong@cnu.edu.cn}

\dedicatory{to the memory of Professor Linsheng Yin}

\begin{abstract}
We consider the Harder-Narasimhan formalism on the category of normed isocrystals and show that the Harder-Narasimhan filtration is compatible with tensor products which generalizes a result of Cornut. As an application of this result, we are able to define a (weak) Harder-Narasimhan stratification on the $\BpdR$-affine Grassmannian for arbitrary $(G, b, \mu)$. 
When $\mu$ is minuscule, it corresponds to the Harder-Narasimhan stratification on the flag varieties defined by Dat-Orlik-Rapoport. And when $b$ is basic, it's studied by Nguyen-Viehmann and Shen.  We study the  basic geometric properties of the Harder-Narasimhan stratification, such as non-emptiness, dimension and its relation with other stratifications.

\end{abstract}

\maketitle

\section*{introduction}
Let $F$ be a $p$-adic local field, with $\breve F$ the $p$-adic completion of the maximal unramified extension. Consider the $p$-adic flag variety $\mathcal{F}(G, \mu)$ attached to the pair $(G, \mu)$, where $G$ is a connected reductive group over $F$ and $\mu$ a geometric minuscule cocharacter of $G$. Let $E$ be the reflex field of the conjugacy class $\{\mu\}$ of $\mu$. Throughout this paper, we view the flag variety $\mathcal{F}(G, \mu)$ as an adic space over $\breve E$, the $p$-adic completion of the maximal unramified extension of $E$. Fix $b\in G(\breve F)$ which satisfies the Kottwitz condition $[b]\in B(G, \mu)$ (cf. Section \ref{section_Kottwitz}).  There are two open subspaces, both called \emph{$p$-adic period domains}, inside $\mathcal{F}(G, \mu)$
\[\mathcal{F}(G,\mu,b)^a\subset \mathcal{F}(G, \mu, b)^{wa}\subset \mathcal{F}(G, \mu),\]
where $\mathcal{F}(G, \mu, b)^{wa}$ is the \emph{weakly admissible locus} and $\mathcal{F}(G, \mu, b)^a$ is the \emph{admissible locus}. The weakly admissible locus $\mathcal{F}(G, \mu, b)^{wa}$ is of algebraic nature, and was constructed by Rapoport and Zink (\cite{RZ}) as the complement of a profinite union of explicit closed subspaces of $\mathcal F(G,\mu)$. Rapoport and Zink also conjectured the existence of the admissible locus $\mathcal{F}(G, \mu, b)^{a}$ which is of geometric nature, with the property that there exists a $p$-adic local system with additional structures over this space which interpolates the crystalline representations attached to all classical points of $\mathcal F(G,\mu)^{a}$. Since then, the admissible locus was first constructed only for several PEL type triples $(G, \mu, b)$ by Hartl (\cite{Har}, \cite{Har1}) and by Faltings (\cite{Fal1}). Thanks to the recent revolutionary progress in $p$-adic Hodge theory, we are now able to define the admissible locus $\mathcal{F}(G, \mu, b)^{a}$ for an arbitrary triple $(G, \mu, b)$ based on the work of Fargues-Fontaine (\cite{FF}), Fargues (\cite{Far20}), Kedlaya-Liu (\cite{KL}) and Scholze-Weinstein (\cite{SW}). The weakly admissible locus $\mathcal{F}(G, \mu, b)^{wa}$ can be considered as an algebraic approximation of the admissible locus $\mathcal{F}(G, \mu, b)^a$ in the sense that they have the same classical points due to a fundamental result in $p$-adic Hodge theory by Colmez-Fontaine (\cite{CF}). The study of the weakly admissible locus could reflect some information about the admissible locus whose geometry is in general still very mysterious (e.g. \cite{CFS}).

The admissible locus $\mathcal{F}(G, \mu, b)^a$ could be also understood as the image of an \'etale morphism of rigid analytic spaces over $\breve E$, i.e., the \emph{$p$-adic period mapping} (\cite{RZ}, \cite{SW}) 
\[\pi: (\mathcal{M}_{G, \mu, b, K})_{K\subset G(\Qp)}\longrightarrow \mathcal{F}(G, \mu),
\]
where $(\mathcal{M}_{G, \mu, b, K})_{K\subseteq G(\Qp)}$ is the local Shimura variety attached to the triple $(G, \mu, b)$. For a general (not necessarily minuscule) geometric cocharacter $\mu$,  the $p$-adic period mapping above could be upgraded to an \'etale morphism of diamonds over $\breve E$ (\cite{SW}): 
\[\pi^{\diamond}: (\mathrm{Sht}_{G, \mu, b, K})_{K\subset G(\Qp)}\longrightarrow \mathrm{Gr}_{G, \leq \mu},\] where $(\mathrm{Sht}_{G, \mu, b, K})_{K\subset G(\Qp)}$ is the moduli space of local Shtukas attached to $(G, \mu, b)$ and 
\[
\mathrm{Gr}_{G, \leq \mu}=\bigcup_{\lambda\leq \mu}\mathrm{Gr}_{G, \lambda}
\]
is the $\BpdR$-affine Grassmannian attached to $(G, \mu)$ (cf. Section \ref{section_BdR Grassmannian}). There exists a natural Bialynicki-Birula map
\[\mathrm{BB}: \mathrm{Gr}_{G, \mu}\longrightarrow \mathcal{F}(G, \mu)^{\diamond},
\]
where $\mathcal{F}(G, \mu)^{\diamond}$ denotes the diamond associated with  $\mathcal{F}(G, \mu)$. When $\mu$ is minuscule, the Bialynicki-Birula map is an isomorphism (\cite{CS}) and the upgraded $p$-adic period mapping $\pi^{\diamond}$ is the morphism of diamonds associated with the usual $p$-adic period mapping $\pi$ above. For non-minuscule $\mu$, the Bialynicki-Birula map is no longer an isomorphism. Moreover, when $\mu$ is non-minuscule, the points in the $\BpdR$-affine Grassmannian $\mathrm{Gr}_{G, \mu}$ but not the points in $\mathcal{F}(G,\mu)$ are naturally related to the modifications of $G$-bundles on the Fargues-Fontaine curve (cf. \cite{SW}). This suggests that when we work with general $\mu$, the flag variety should be replaced by the $\BpdR$-affine Grassmannian $\mathrm{Gr}_{G, \leq \mu}$ (or $\mathrm{Gr}_{G, \mu}$). As inside the flag varieties, for $[b]\in B(G, \mu)$, we still have subspaces
\[\mathrm{Gr}_{G, \leq \mu, b}^a\subset \mathrm{Gr}_{G, \leq \mu, b}^{wa} \subset\mathrm{Gr}_{G, \leq\mu}.
\]
Here the admissible locus $\mathrm{Gr}_{G, \leq \mu, b}^a:=\mathrm{Im}(\pi^{\diamond})$ could be also defined directly using modification of $G$-bundles on the Fargues-Fontaine curve for arbitrary triple $(G, \mu, b)$ with the Kottwitz condition. On the other hand,  the weakly admissible locus $\mathrm{Gr}_{G, \leq \mu, b}^{wa}$ is only defined for the triple $(G, \mu, b)$ with the condition that $\mu$ is minuscule (\cite{DOR}, \cite{CFS}), or $b=1$ (\cite{Vi}, \cite{Sh}), or $G=\mathrm{GL}_n$ (\cite{NV}). When  $\mu$ is minuscule,  the admissible locus $\mathrm{Gr}_{G, \leq\mu, b}^a=\mathrm{Gr}_{G, \mu, b}^a$ (resp. the weakly admissible locus $\mathrm{Gr}_{G, \leq\mu, b}^{wa}=\mathrm{Gr}_{G, \mu, b}^{wa}$) is the diamond associated to $\mathcal{F}(G, \mu, b)^a$ (resp. $\mathcal{F}(G, \mu, b)^{wa}$) through the identification $\mathrm{BB}: \mathrm{Gr}_{G, \mu}\stackrel{\sim}{\rightarrow} \mathcal{F}(G, \mu)^{\diamond}$.  


To better understand the admissible locus (resp. weakly admissible locus) on the $\BpdR$-affine Grassmannian $\mathrm{Gr}_{G,\mu}$, it's useful to have a natural stratification on  $\mathrm{Gr}_{G, \mu}$ having the admissible locus (resp. the weakly admissible locus) as its unique open stratum. The stratification on $\mathrm{Gr}_{G, \mu}$ with the required property for the admissible locus is called the \emph{Newton stratification}. Like the construction of the admissible locus, it's defined by the isomorphism classes of modifications of $G$-bundles on the Fargues-Fontaine curve for any triple $(G, \mu, b)$. Basic properties of the Newton stratification are studied by many people, such as Rapoport (\cite{Ra}), Caraiani-Scholze (\cite{CS}),  Chen-Fargues-Shen (\cite{CFS}), Hansen (\cite{Han}) and Viehmann (\cite{Vi}). 

On the weakly admissible side, the corresponding stratification on the $\BpdR$-affine Grassmannian $\mathrm{Gr}_{G, \mu}$ is called the \emph{weak Harder-Narasimhan stratification}. Like the situation of the weakly admissible locus, in the literature, it's not yet defined for all the triples $(G, \mu, b)$. When $\mu$ is minuscule, this is done by  Dat-Orlik-Rapoport (\cite{DOR}). Indeed, they construct and study the Harder-Narasimhan stratification on the flag variety $\mathcal{F}(G, \mu)$  for arbitrary $(G, \mu, b)$. When $\mu$ is minuscule, recall that the Bialynicki-Birula map $\mathrm{BB}$ is an isomorphism, their construction gives the Harder-Narasimhan stratification on $\mathrm{Gr}_{G, \mu}$. When $b$ is basic while $\mu$ is general, the corresponding Harder-Narasimhan stratification on $\mathrm{Gr}_{G,\mu}$ is constructed and studied by Nguyen-Viehmann (\cite{NV}) and Shen (\cite{Sh}) based on a crucial result of Cornut-Peche Irissarry (\cite{CPI}) about the compatibility of certain notion of semistability with tensor products that we will explain later. Moreover, Nguyen-Viehmann also completely settles the case when $G=\mathrm{GL}_n$. 

The reason we name it weak Harder-Narasimhan stratification here and also in the title is twofold. One is that it generalizes the weak admissibility condition. And the other is that historically the Newton stratification is also called Harder-Narasimhan stratification (e.g. \cite[5.3]{CFS}). In the following, for simplicity, we will omit the word ``weak" and call it simply the Harder-Narasimhan stratification if there is no confusion.

The purpose of this paper is to propose a construction of a Harder-Narasimhan stratification on $\mathrm{Gr}_{G, \mu}$ for arbitrary $(G, \mu, b)$ that coincides with the previous constructions. A key ingredient is to generalize the result of Cornut-Peche Irissarry in a more general setting. To explain the main idea of our construction, let us first recall how to define the Harder-Narasimhan stratification on the flag variety $\mathcal{F}(G, \mu)$. We know that the weakly admissible locus $\mathcal{F}(G, \mu, b)^{wa}$ parametrizes weakly admissible filtered isocrystals with $G$-structures, with the isocrystal determined by $b$ and the filtration of type $\mu$. Let $C$ be a field extension of $\breve F$. Consider the category 
\[
\mathbf{FilIsoc}_{\breve F|F}^C
\]
of filtered isocrystals over $C|\breve F$ whose objects consist of an isocrystal $D$ over $\breve F|F$ together with a separated exhaustive decreasing filtration by $C$-subspaces of $D\otimes_{\breve F}C$. We can define a Harder-Narasimhan formalism on this category, so that the semi-stable objects of degree $0$ are precisely the weakly admissible filtered isocrystals introduced by Fontaine. By classifying the filtered isocrystals according to the type of their Harder-Narasimhan filtration, we obtain the Harder-Narasimhan stratification on $\mathcal{F}(\mathrm{GL}_n, \mu)$ for any triple of the form $(\mathrm{GL}_n, \mu, b)$. Moreover, Faltings (\cite{Fa}) and Totaro (\cite{To1}) show that the Harder-Narasimhan filtration is compatible with tensor products. Based on this result, using Tannakian formalism, Dat-Orlik-Rapoport are able to define the Harder-Narasimhan stratification on the flag variety $\mathcal{F}(G, \mu)$ for arbitrary $G$. 

From now on, let $C$ be a complete algebraically closed non-archimedean field over $\breve F$. Let $\BpdR$ and $\BdR$ be Fontaine's de Rham period rings corresponding to $C$.
Recall that for an arbitrary cocharacter $\mu$ (possibly non-minuscule), we replace the flag variety $\mathcal{F}(G, \mu)$ (which parametrizes filtrations of type $\mu$) by the $\BpdR$-affine Grassmannian $\mathrm{Gr}_{G, \mu}$ (which parametrizes lattices having relative position $\mu$ with the standard lattice). Accordingly, we replace the category $\mathbf{FilIsoc}_{\breve F|F}^C$ by the category 
\[
\mathbf{BunIsoc}_{\breve F|F}^{\BdR}
\]
whose objects consist of isocrystals $D$ over $\breve F|F$ equipped with a $\BpdR$-lattice inside $D\otimes_{\breve F}\BdR$. There exists a natural Harder-Narasimhan formalism on this category, which leads to the construction of the Harder-Narasimhan stratification on $\mathrm{Gr}_{G, \mu}$ for $G=\mathrm{GL}_n$ by classifying the objects in $\mathbf{BunIso}_{\breve F|F}^{\BdR}$ according to the type of their Harder-Narasimhan filtration. In order to work with arbitrary $G$, we then need a result parallel to the result of Faltings (\cite{Fa}) and Totaro (\cite{To1}) mentioned above. More precisely, let 
\[
\mathcal{F}_{\rm HN}: \mathbf{BunIsoc}_{\breve F|F}^{\BdR}\longrightarrow \mathbf{F}(\mathbf{Isoc}_{\breve F|F}) 
\]
be the functor sending any object to its Harder-Narasimhan filtration, where $\mathbf{F}(\mathbf{Isoc}_{\breve F|F})$ denotes the category of $\R$-descending filtrations of isocrystals over $\breve F|F$ by subisocrystals. 

\begin{itheorem}[Theorem \ref{thm_compatible tensor}]The functor $\mathcal{F}_{\rm HN}: \mathbf{BunIsoc}_{\breve F|F}^{\BdR}\rightarrow \mathbf{F}(\mathbf{Isoc}_{\breve F|F})$ is compatible with tensor products. 
\end{itheorem}

When we restrict the functor $\mathcal F_{\rm HN}$ to the subcategory  consisting of objects whose underlying isocrystals are isoclinic of slope $0$, this result is proved by Cornut and Peche Irissarry (\cite{CPI}, compare also \cite{Cor}). To show the desired compatibility with tensor products without any restriction on the slopes, we follow the general approach developed by Cornut in \cite{Cor}. Indeed, generalizing the notion of normed vector spaces considered by Cornut in loc. cit., we introduce the notion of normed isocrystals, so that the isocrystals with lattices are precisely those normed isocrystals equipped with the gauge norm induced form a lattice (cf. Section \ref{Sec_variant}). We show that one can also develop a Harder-Narasimhan formalism for normed isocrystals, and the resulting Harder-Narasimhan filtration is compatible with tensor products (cf. Theorem \ref{thm_comp tensor}).

Based on this result, we are able to define the Harder-Narasimhan stratification on $\mathrm{Gr}_{G,\mu}$ for arbitrary triple $(G, \mu, b)$. Then we study the basic properties of the Harder-Narasimhan stratification and its relations to other stratifications. We show that the Harder-Narasimhan stratification, the Newton stratification on the $\BpdR$-affine Grassmannian, and the Harder-Narasimhan stratification on the flag variety (via Bialynicki-Birula map) coincide on classical points (Theorem \ref{thm_classical}) which generalizes the result of Viehmann (\cite{Vi}) for $b=1$. We give a combinatorial criterion for the non-emptiness of any Harder-Narasimhan stratum (Theorem \ref{thm_nonempty}) which is compatible with the result of Nguyen-Viehmann (\cite{NV}) for $b=1$ and is parallel to the result of Orlik (\cite{Orl}) for the Harder-Narasimhan stratum on flag varieties for $G=\mathrm{GL}_n$. When $\mu$ is minuscule, we also give a dimension formula for a Harder-Narasimhan stratum (Theorem \ref{thm_dimension}) which generalizes the result of Fargues (\cite{Far}) for $G=\mathrm{GL}_n$ and $b=1$.

We briefly describe the structure of this article. In Section 1, we review the preliminaries such as $G$-bundles on the Fargues-Fontaine curve and the $\BpdR$-affine Grassmannian. In Section 2, we show the fact that the Harder-Narasimhan filtration is compatible with tensor products in the category of normed isocrystals, which generalizes a result of Cornut for normed vector spaces. As an application, in Section 3, we define the Harder-Narasimhan stratification on the $\BpdR$-affine Grassmannian for arbitrary triple $(G, \mu, b)$. In Section 4, we compare the Harder-Narasimhan stratification with the Newton stratification and the Harder-Narasimhan stratification on the flag variety. In Section 5, we discuss some basic properties of a Harder-Narasimhan stratum such as non-emptiness and dimension formula.   

\textbf{Acknowledgments.} We would like to thank Christophe Cornut,  Laurent Fargues, David Hansen, Kieu Nieu Nguyen, Sian Nie, Xu Shen, Eva Viehmann for helpful discussions. We also thank Kieu Nieu Nguyen and Eva Viehmann for their comments on the previous version of this work. The first author is partially supported by NSFC grant No.12222104, No.12071135, and the second author is partially supported by NSFC grant No. 12231009.

\section{preliminaries on $G$-bundles on the Fargues-Fontaine curve}

We denote by $\mathbf{Perf}$ the category of perfectoid spaces over $\mathbb F_q$ which is the residue field of the $p$-adic local field $F$. Let $G$ be a connected reductive group over $F$. When $G$ is assumed to be quasi-split, we fix  $A\subset T\subset B\subset G$ defined over $F$, where $A$ is a maximal split torus, $T=Z_G(A)$ is the centralizer of $A$ in $G$, and $B$ is a Borel subgroup. Let $W=W_G=N_G(T)/T$ be the absolute Weyl group of $T$ in $G$ and let $w_0\in W$ be the longest element.

\subsection{$G$-bundles on the Fargues-Fontaine curve}
\subsubsection{Fargues-Fontaine curve}\label{sec_FF curve} We briefly recall the construction of the (relative) Fargues-Fontaine curve (cf. \cite[\S~1]{Far1}, or \cite[\S~II.1]{FS}).
Let $\pi_F$ be a uniformizer of $F$ and let $\mathbb{F}_q$ be the residue field of $F$. For  $S=\mathrm{Spa}(R, R^+)\in \mathbf{Perf}$ a perfectoid affinoid space over $\mathbb F_q$, fix a pseudo-uniformizer $\omega\in R^+\subset R$. Let 
\[\mathcal{Y}_S:=\mathrm{Spa}(\mathbb{A}, \mathbb{A})\backslash V(\pi_F[\omega]),
\]
where 
\[
\mathbb{A}=W_{\mathcal{O}_F}(R^+)=\left\{\sum_{n\geq 0}[x_n]\pi_F^n\mid x_n\in R^+\right\}
\]
is the ramified Witt ring with coefficients in $R^+$ equipped with the $(\pi_F,[\omega])$-adic topology. It's known that $\mathcal Y_S$ is an analytic adic space over $\mathcal O_F$ (\cite[Proposition II.1.1]{FS}). The $q$-th Frobenius automorphism on $R^+$ induces an automorphism $\phi_S$ on $\mathcal Y_S$ over $\mathcal O_F$, and this action is free and totally discontinuous (\cite[Proposition II.1.16]{FS}). The \emph{relative adic Fargues-Fontaine curve}  $\mathcal{X}_S$ is by definition the following quotient 
\[\mathcal{X}_S:=\mathcal{Y}_S/\phi_S^{\mathbb{Z}}.
\]
We also have the \emph{relative schematic Fargues-Fontaine curve} $X_S$ defined as 
\[
X_S:=\mathrm{Proj}\left( \bigoplus_{d=0}^{\infty} H^0(\mathcal{Y}_S, \mathcal{O}_{\mathcal{Y}_S})^{\phi_S=\pi_F^d}\right).
\]
When $S=\mathrm{Spa}(K,K^+)$ with $K$ a perfectiod field, $X_S$ (sometimes also denoted by $X$ or $X_K$ since it does not depend on the choice of $K^+$) is a noetherian scheme of dimension $1$.

\begin{remark} Let $S\in \mathbf{Perf}$ which is not necessarily affinoid. \begin{enumerate}
\item By gluing, the local construction above defines the relative Fargues-Fontaine curves $\mathcal X_S$ and $X_S$ (\cite[Proposition II.1.3 and II.2.7]{FS}). 
\item 
There is a morphism of locally ringed spaces:
\[
(\mathcal{X}_S,\mathcal{O}_{\mathcal X_S})\longrightarrow X_S.
\]
The GAGA functor associated to the morphism of ringed spaces 
above gives an equivalence of categories (\cite{KL}): 
\[\mathrm{Bun}_{X_S}\longrightarrow \mathrm{Bun}_{(\mathcal{X}_S,\mathcal O_{\mathcal X_S})},\]
where $\mathrm{Bun}_{X_S}$ (resp. $\mathrm{Bun}_{(\mathcal{X}_S,\mathcal O_{\mathcal X_S})}$) denotes the category of vector bundles on $X_S$ (resp. on the locally ringed space  $(\mathcal{X}_S,\mathcal O_{\mathcal X_S})$). See \cite[Proposition II.2.7]{FS} for more details. 
\end{enumerate}
\end{remark}

Let  $S=\mathrm{Spa}(R, R^+)\in \mathbf{Perf}/\mathrm{Spd}(F)$ be an object of $\mathbf{Perf}$ lying over the diamond $\mathrm{Spd}(F)$, which corresponds to an untilt $S^\sharp=\mathrm{Spa}(R^{\sharp}, R^{\sharp +})$ over $F$ of $S$. The untilt $S^{\sharp}$ gives rise to a natural closed immersion of adic spaces 
\[
S^{\sharp} \hookrightarrow \mathcal{Y}_S
\]
that presents $S^{\sharp}$ as a closed Cartier divisor in $\mathcal Y_S$ (\cite[Proposition II.3.1]{FS}), with image $V(\xi)\subset \mathcal Y_S$ for some element $\xi\in \mathbb{A}$. Furthermore, locally on $S$ the closed Cartier divisor above can be also  defined by an element in $H^0(\mathcal{X}_S,\cO_{\mathcal{X}_S}(1))=H^0(\mathcal Y_S,\mathcal O_{\mathcal Y_S})^{\phi_S=\pi_F}$ (\cite[Proposition II.2.3]{FS}). Consequently, it  defines a Carter divisor of the relative schematic Fargues-Fontaine curve 
\[
\Spec(R^{\sharp})\hookrightarrow X_{S}.
\]
The Fontaine's de Rham period ring $\BpdR(R^{\sharp})$ is defined to be the completion of $\mathcal{O}_{X_S}$ along this Cartier divisor(cf. \cite[1.33]{Far1}) which is also the $\xi$-adic completion of $\mathbb{A}[\frac{1}{\pi_F}]$. We write    $\BdR(R^{\sharp})= \BpdR(R^{\sharp})[\frac{1}{\xi}]$.

\subsubsection{The Kottwitz set}\label{section_Kottwitz}

Let $B(G)$ be the set of $\sigma$-conjugacy classes of $G(\breve F)$, where $\sigma$ is the Frobenius with respective to $\breve F|F$. For $b\in G(\breve F)$, let $[b]\in B(G)$ be the $\sigma$-conjugacy class of $b$. The Kottwitz set is characterized by two invariants. One is called the Newton map (cf. \cite[\S~4]{Kot1}). For any $b\in G(\breve F)$, we can define a slope morphism:
\[\nu_b: \mathbb{D}_{\breve F}\longrightarrow G_{\breve F}\] whose conjugacy class is defined over $F$, where $\mathbb{D}$ is the pro-torus with characters $X^*(\mathbb{D})=\mathbb{Q}$. This defines the \emph{Newton map}:\begin{eqnarray*}\nu: B(G)&\longrightarrow &\mathcal{N}(G),\\   {[b]}&\longmapsto &[\nu_b]\end{eqnarray*}
where \[\mathcal{N}(G):=[\mathrm{Hom}(\mathbb{D}_{\bar F}, G_{\bar F})/G(\bar F)-\text{conjugacy}]^{\Gamma}.
\]
with $\Gamma=\mathrm{Gal}(\bar F|F)$. An element $[b]$ in $B(G)$ is called \emph{basic} if $\nu_b$ is central. On $\mathcal{N}(G)$, there exists a partial order, called \emph{dominance order}, denoted by $\leq$. When $G$ is quasi-split, there is a natural identification
\[\mathcal{N}(G)\simeq X_*(A)^+_{\Q}.\]

The other invariant is the \emph{Kottwitz map} (cf. \cite[4.9, 7.5]{Kot}):  
\begin{eqnarray*}\kappa=\kappa_G: B(G)&\longrightarrow &\pi_1(G)_{\Gamma},\\   {[b]}&\longmapsto &\kappa([b])\end{eqnarray*}
where $\pi_1(G)$ is the algebraic fundamental group of $G$ and $\pi_1(G)_{\Gamma}$ is its Galois coinvariants.

By \cite[4.13]{Kot}, the induced map:
\begin{eqnarray}\label{eqn_RR} (\nu, \kappa): B(G)\longrightarrow \mathcal{N}(G)\times \pi_1(G)_{\Gamma}\end{eqnarray} is injective.

Let $[b]\in B(G)$ and $\mu\in X_*(T)^+$, we define the \emph{Kottwitz set}  (\cite{Kot})
\[B(G, \mu):=\{[b]\in B(G)\mid [\nu_b]\leq \mu^{\diamond}, \kappa_G(b)=\mu^{\sharp}\},\]
where 
 $\mu^{\diamond}\in \mathcal{N}(G)$ is the Galois average of $\mu$, $\mu^{\sharp}\in \pi_1(G)_{\Gamma}$ is the image of $\mu$ via the natural quotient map $X_{*}(T)\rightarrow \pi_1(G)_{\Gamma}$, and the order $\leq$ on $\mathcal{N}(G)$ is the usual order: $\nu_1\leq \nu_2$ if and only if $\nu_1$ lies in the convex hull of the Weyl orbit of $\nu_2$.

We also need the following \emph{generalized Kottwitz set} (cf. \cite{CFS}).
For $\epsilon\in \pi_1 (G)_\Gamma$ and $\delta\in \mathcal{N}(G)$ we set
\[
B(G,\epsilon,\delta) = \{ [b]\in B(G)\ |\ \kappa_G (b)=\epsilon \text{ and } [\nu_b]\leq \delta\}.
\]

We will need the following well-known fact, which can be deduced easily from a general result of Kottwitz (\cite[3.6]{Kot}).
\begin{lemma}\label{lemma_B(P)} Let $P$ be a parabolic subgroup of $G$ over $F$ with $M$ the Levi component. Then $B(M)\simeq B(P)$.
\end{lemma}


\subsubsection{$G$-bundles}\label{subsubsection:G-bundles}

Recall that a \emph{$G$-bundle} on an $F$-scheme $Z$ is a (left) $G$-torsor on $Z$ locally trivial for the \'etale topology. It could also be viewed as an exact tensor functor 
\[
\mathrm{Rep} G \longrightarrow \mathrm{Bun}_{Z},
\]
where $\mathrm{Rep } G$ is the category of rational algebraic representations of $G$ and $\mathrm{Bun}_Z$ denotes the category of vector bundles (of finite rank) over $Z$.

Let $S$ be a  perfectoid space over the residue field $\bar{\mathbb F}_q$ of $\breve F$. For each $b\in G(\breve F)$, we can associate a $G$-bundle 
\[
\mathcal{E}_b
\]
on the relative Fargues-Fontaine curve $X_S$ as follows. As the construction is  functorial, it suffices to consider the case where $S=\mathrm{Spa}(R,R^+)$ is an affinoid perfectoid space over $\bar{\mathbb F}_q$. In particular $X_S=\mathrm{Proj}(\oplus_d B^{\phi_S=\pi_F^d})$, with $B=H^0(\mathcal Y_S,\mathcal O_{\mathcal Y_S})$. 
As an exact tensor functor the $G$-bundle $\mathcal E_b$ is given by 
\[
\mathcal E_b:\mathrm{Rep}G\longrightarrow \mathrm{Bun}_{X_S}, \quad (V,\rho)\mapsto \mathcal E_b(V,\rho),
\]
with $\mathcal E_b(V,\rho)$ the vector bundle on $X_S$ corresponding to the following graded module 
\[
\bigoplus_{d\in \mathbb Z} \left(V\otimes_F B\right)^{\rho(b)\phi_S=\pi_F^d}
\]
over the graded ring $\oplus_d B^ {\phi_S=\pi_F^d}$.
Here we use the same notation $\rho(b)$ to denote the $B$-linear automorphism on $V\otimes_FB=(V\otimes_F\breve F)\otimes_{\breve F}B$ induced by $\rho(b)\in \mathrm{GL}(V\otimes_F \breve F)$. 
The isomorphism class of $\mathcal E_b$ only depends on the $\sigma$-conjugacy class of $b$.

If moreover $S=\mathrm{Spa}(C,C^+)$ with $C$ an algebraically closed perfectoid field containing $\bar{\mathbb F}_q$ and write $X=X_S$, this construction gives a bijection of pointed sets by a fundamental theorem of Fargues (\cite{Far20}):
\begin{eqnarray*}B(G)&\stackrel{\sim}{\longrightarrow}& H^1_{\acute{e}t}(X, G).\\
{[b]}&\longmapsto& \mathcal{E}_b\end{eqnarray*}

\subsubsection{Harder-Narasimhan reduction in the quasi-split case}\label{sec: HN reduction}
In this subsection, suppose $G$ is quasi-split.  

A $G$-bundle $\E$ on $X$ is called \emph{semi-stable} if for any standard parabolic subgroup $P$ of $G$, any reduction $\E_P$ of $\E$ to $P$ and any $\chi\in X^*(P/Z_G)^+$, we have $\mathrm{deg}\chi_*(\E_P)\leq 0$. 

For any $G$-bundle $\E$ on $X$, by \cite[Theorem 1.7]{CFS}, there exist a unique standard parabolic subgroup $P$ of $G$ and a unique reduction $\E_P$ of $\E$ to $P$, such that
the associated $M$-bundle $\E_P\times^P M$ is semi-stable and $P$ is maximal among all the standard parabolic with the previous condition. The reduction $\E_P$ is called the \emph{Harder-Narasimhan reduction} or \emph{canonical reduction} of $\E$. This defines the Harder-Narasimhan polygon of $\E$:
\[\nu_{\E}\in X_*(A)^+_{\Q},\] such that for any $\chi\in X^*(P)$, we have $\langle \chi, \nu_{\E}\rangle= \mathrm{deg}\chi_*(\E_P)$.
By \cite[Proposition 6.6]{Far20}, we have $\nu_{\E_{b}}=-w_0\nu_b$. 

The Harder-Narasimhan reduction has the following characterization.

\begin{proposition}[{\cite[Theorem 4.5.1]{Sch}}]\label{prop_canonical reduction characterization} For $\E$ a $G$-bundle on $X$ equipped with a reduction to the standard parabolic subgroup $Q$ of $G$, consider the vector
\[\begin{split}v: X^*(Q)&\longrightarrow\Z \\ \chi&\longmapsto \mathrm{deg}\chi_*\E_Q  \end{split}\] seen as an element of $X_*(A)_{\Q}$, then \begin{enumerate}
\item One has $v\leq \nu_\E$,
\item if this inequality is an equality, then $Q\subset P$ and the canonical reduction $\E_P\simeq \E_Q\times^Q P$.
\end{enumerate}
\end{proposition}

\subsection{$\BpdR$-affine Grassmannian} \label{section_BdR Grassmannian} We collect some basic properties of the $\BpdR$-affine Grassmannian that we will need in the sequel. The main references for this section are \cite{SW} and \cite{FS}.  We consider the $\BpdR$-affine Grassmannian
\[
\mathrm{Gr}_G=\mathrm{Gr}_{G}^{\BpdR}
\]
as a functor on the category $\mathbf{Perf}/\mathrm{Spd}(F)$ of perfectoid spaces over $\mathbb F_q$ lying above $\mathrm{Spd}(F)$. More precisely, $\mathrm{Gr}_G$ is the \'etale sheafification of the functor taking $S=\Spa(R,R^+)\in \mathbf{Perf}/\mathrm{Spd}(F)$, which corresponds to an untilt $(R^{\sharp},R^{\sharp +})$ of $(R,R^+)$ over $F$,  to the coset space
\[
G(\BdR(R^{\sharp}))/G(\BpdR(R^{\sharp})).
\]
According to \cite[Proposition 19.1.2]{SW}, 
\[
\mathrm{Gr}_G(S)\simeq \left\{(\mathcal E,\iota)\Bigg| { \mathcal E \textrm{ is a }G\textrm{-bundle on }\Spec(\BpdR(R^{\sharp})), \textrm{ and} \atop \iota\textrm{ is a trivialization of }\mathcal E|_{\Spec(\BdR(R^{\sharp}))}} \right\}/\simeq,
\]
and $\mathrm{Gr}_G$ is a small v-sheaf.

\subsubsection{Schubert varieties}\label{subsec:Schubert} Let $\bar F$ be an algebraic closure of $F$. We choose a maximal torus $T\subset G_{\bar F}$ and a Borel subgroup containing $T$. Write $\mathrm{Gr}_{G,\bar F}$ the restriction of $\mathrm{Gr}_G$ to $\mathbf{Perf}/\mathrm{Spd}(\bar F)$. Let $S=\Spa(C,C^+)\ra \mathrm{Spd}(\bar F)$ be a geometric point of $\mathrm{Spd}(\bar F)$, given by an untilt $\Spa(C^{\sharp},C^{\sharp+})$ of $S$ over $\bar F$.

\begin{remark}The period ring $\BpdR(C^{\sharp})$ is a complete discrete valuation ring with residue field $C^{\sharp}$. Moreover as $\bar F\subset C^{\sharp}$, the $F$-algebra $\BpdR(C^{\sharp})$ is also naturally an $\bar F$-algebra. Non-canonically we have $\BpdR(C^{\sharp})\simeq C^{\sharp}[\![\xi]\!]$ with $\xi$ a uniformizer of $\BpdR(C^{\sharp})$.
\end{remark}

By the Cartan decomposition,
\[
G(\BdR(C^{\sharp}))=\coprod_{\mu\in X_*(T)^+}G(\BpdR(C^{\sharp}))\mu(\xi)^{-1}G(\BpdR(C^{\sharp})).
\]
In particular,
\[
\begin{array}{rcl}
\mathrm{Gr}_G(S)&=&G(\BdR(C^{\sharp}))/G(\BpdR(C^{\sharp}))\\ &=&\coprod_{\mu\in X_*(T)^+} G(\BpdR(C^{\sharp}))\mu(\xi)^{-1}G(\BpdR(C^{\sharp}))/G(\BpdR(C^{\sharp})).
\end{array}
\]
Here we have the first equality since $C$ is algebraically closed and thus every $G$-bundle over $\mathrm{Spec}(\BpdR(C^{\sharp}))$ is trivial. 
By abuse of notation, for any geometric conjugacy class $\{\mu\}\in X_*(G)/G(\bar F)$, we denote by
\[
\mathrm{Gr}_{G,\mu}\subset \mathrm{Gr}_{G,\leq \mu}\subset \mathrm{Gr}_{G,\bar F}
\]
the subfunctors of $\mathrm{Gr}_{G,\bar F}$ defined by the condition that a map
\[
S'\longrightarrow \mathrm{Gr}_{G,\bar F}
\]
with $S'\in \mathbf{Perf}/\mathrm{Spd}(\bar F)$ factors through $\mathrm{Gr}_{G,\mu}$, or respectively $\mathrm{Gr}_{G,\leq \mu}$, if and only if for all geometric point $S=\Spa(C,{C}^{+})\ra S'$, the corresponding $S$-valued point of $\mathrm{Gr}_{G,\bar F}$ lies in
\[
G(\BpdR(C^{\sharp}))\mu(\xi)^{-1}G(\BpdR(C^{\sharp}))/G(\BpdR(C^{\sharp})),
\]
or respectively lies in
\[
\coprod_{\mu'\leq \mu }G(\BpdR(C^{\sharp}))\mu'(\xi)^{-1}G(\BpdR(C^{\sharp}))/G(\BpdR(C^{\sharp})).
\]
According to \cite[Proposition 19.2.3]{SW}, $\mathrm{Gr}_{G,\leq \mu}\subset \mathrm{Gr}_{G,\bar F}$ is a closed subfunctor proper over $\mathrm{Spd}(\bar F)$, and $\mathrm{Gr}_{G,\mu}\subset \mathrm{Gr}_{G,\leq \mu}$ is an open subfunctor. Furthermore, the small v-sheaf $\mathrm{Gr}_{G,\leq \mu}$ is a spatial diamond (\cite[Theorem 19.2.4]{SW}).

\subsubsection{Generalized semi-infinite orbits}\label{sec_semi-infinite orbits}
We keep the notations in \S~\ref{subsec:Schubert}. Let $P\subset G_{\bar F}$ be a standard parabolic subgroup. Let $M$ be the Levi quotient of $P$, $M^{ab}$ the maximal torus quotient of $M$ and $N\subset P$ the unipotent radical. We have a natural identification of v-sheaves on $\mathbf{Perf}/\mathrm{Spd}(\bar F)$
\[
\mathrm{Gr}_{M^{ab}}\stackrel{\sim}{\longrightarrow} X_*(M^{ab}), \]
which sends $v(\xi)^{-1}M^{ab}(\BpdR)\in \mathrm{Gr}_{M^{ab}}(C)$ to $v$ for every cocharacter $v\in X_*(M^{ab})$.

\begin{remark}To be consistent with the definition of Schubert varieties, we insert a minus sign here compared with \cite{FS}.
\end{remark}

On the other hand, consider the $\BpdR$-affine Grassmannian $\mathrm{Gr}_P$ associated with the parabolic subgroup $P$.
 For $v\in X_*(M^{ab})$, write
 \[
\mathrm{Gr}_{P}^v\subset \mathrm{Gr}_{P,\bar F}
\]
the open and closed subfunctor of $\mathrm{Gr}_{P,\bar F}$ (over $\mathbf{Perf}/\mathrm{Spd}(\bar F)$) obtained as the preimage of $v\in X_*(M^{ab})$  under the following composed map
\[
\mathrm{Gr}_{P,\bar F}\longrightarrow \mathrm{Gr}_{M,\bar F}\longrightarrow \mathrm{Gr}_{M^{ab},\bar F}\stackrel{\sim}{\longrightarrow} X_*(M^{ab}).
\]

\begin{proposition}[{\cite[Proposition VI 3.1]{FS}} ]\label{prop_GrPv} The map
\[
\mathrm{Gr}_{P}=\coprod_{v\in X_*(M^{ab})}\mathrm{Gr}_{P}^v\longrightarrow \mathrm{Gr}_{G,\bar F}
\]
is bijective on geometric points, and it is a locally closed immersion on each $\mathrm{Gr}_{P}^{v}$. Moreover, the image of the union
\[
\bigcup_{v\leq v'}\mathrm{Gr}_P^{v'}
\]
is closed in $\mathrm{Gr}_{G,\bar F}$. Here $X_*(M^{ab})\subset X_*(T)_{\mathbb Q}$ is equipped with the dominance order, where $v\leq v' $ if $v'-v\in X_*(T)_{\mathbb Q}$ is a sum of positive coroots with $\mathbb Z_{\geq 0}$-coefficients.
\end{proposition}

We also have a finer version of $\mathrm{Gr}_P^{v}$ obtained as follows. Let $\lambda\in X_*(T)$ be an $M$-dominant cocharacter. We set
\[
\mathrm{Gr}_{P,\lambda}\subset \mathrm{Gr}_{P,\bar F}
\]
the locally closed subfunctor given by the preimage of $\mathrm{Gr}_{M,\lambda}\subset \mathrm{Gr}_M$ through the natural projection
\[
\mathrm{pr}_M:\mathrm{Gr}_P\longrightarrow \mathrm{Gr}_M.
\]
Clearly,
\[
\mathrm{Gr}_{P,\lambda}\subset \mathrm{Gr}_P^{\bar \lambda},
\]
where $\bar \lambda$ denotes the image of $\lambda\in X_*(M)$ in $X_{*}(M^{ab})$. So $\mathrm{Gr}_{P,\lambda}\subset \mathrm{Gr}_{G,\bar F}$ is locally closed. Moreover, for a geometric point $\Spa(C,C^+)\rightarrow \mathrm{Spd}(\bar F)$ of $\mathrm{Spd}(\bar F)$ corresponding to an untilt $(C^{\sharp},C^{\sharp+})$ over $\bar F$, we have
\[
\mathrm{Gr}_{P,\lambda}(C,C^+)=N(\BdR(C^{\sharp}))M(\BpdR(C^{\sharp}))\lambda(t)^{-1}G(\BpdR(C^{\sharp}))/G(\BpdR(C^{\sharp})).
\]
\begin{remark}When the parabolic subgroup $P$ is minimal (i.e. a Borel subgroup), then $\mathrm{Gr}_{P, \lambda}=\mathrm{Gr}_{P}^{\lambda}$ is studied by Viehmann (cf. \cite[\S 2.2]{Vi}, see also \cite[Example VI.3.4]{FS}), and is denoted by $S_{\lambda}$ in loc. cit. It is the analogue of the semi-infinite orbit for $\lambda\in X_*(T)$ of the usual affine Grassmannian.  
\end{remark}

Motivated by \cite{GHKR}, for a given $G$-dominant cocharacter $\mu$ we write $S_{M}(\mu;C,C^+)$ for the set of $M$-dominant cocharacters $\lambda$ for which the intersection
\[
N(\BdR(C^{\sharp}))\lambda^{-1}(t)\cap G(\BpdR(C^{\sharp}))\mu^{-1}(t)  G(\BpdR(C^{\sharp}))\neq\emptyset.
\]
We shall see in the next lemma that the set $S_M(\mu;C,C^+)$ does not depend on the choice of the geometric point $\Spa(C,C^+)\ra \mathrm{Spd}(\bar F)$ and hence there will be no confusion to write $S_M(\mu):=S_M(\mu; C, C^+)$.

\begin{lemma}\label{lemma_reduction type} We keep the notations above. 
\begin{enumerate}
\item The set $S_M(\mu)=S_M(\mu;C,C^+)$ does not depend on the choice of the geometric point of $\mathrm{Spd}(\bar F)$.

\item Let $\lambda\in X_*(T)$ be $M$-dominant and $\mu\in X_*(T)$ be $G$-dominant. Then the following assertions are equivalent:
\begin{enumerate}
\item $
\mathrm{Gr}_{G,\mu}\cap \mathrm{Gr}_{P,\lambda}\neq \emptyset$,
\item $
\mathrm{Gr}_{G,\mu}(C,C^+)\cap \mathrm{Gr}_{P,\lambda}(C,C^+)\neq \emptyset$ for some geometric point $\mathrm{Spa}(C,C^+)\ra \mathrm{Spd}(\bar F)$,
\item $
\mathrm{Gr}_{G,\mu}(C,C^+)\cap \mathrm{Gr}_{P,\lambda}(C,C^+)\neq \emptyset$ for every geometric point $\mathrm{Spa}(C,C^+)\ra \mathrm{Spd}(\bar F)$,
\item $\lambda\in S_M(\mu)$.
\end{enumerate}
\end{enumerate}
\end{lemma}

\begin{proof} It is enough to prove (1).
Recall that non-canonically $\BpdR(C^{\sharp})\simeq C^{\sharp}[\![\xi]\!]$. Write $\mathscr{G}_G$ for the classical affine Grassmannian, that is, the \'etale sheafification of the functor on the category $\mathbf{Aff}_{\bar F}$ of affine $\bar F$-schemes sending $\mathrm{Spec}(R)\in \mathbf{Aff}_{\bar F}$ to the coset
\[
G(R(\!(\xi)\!))/G(R[\![\xi]\!]).
\]
Let $\mathscr{G}_{G,\mu}\subset \mathscr G_G$ be the corresponding Schubert variety attached to $\mu$ (with our convention on the minus sign), and $\mathscr{S}_{\lambda}\subset \mathscr G_G$ be the orbit of the $\bar F$-point 
\[\lambda(\xi)^{-1}G(\bar F[\![\xi]\!])\in \mathscr G_G(\bar F) 
\]
under the action of the ind-group scheme 
\[
\mathscr{N}:\mathbf{Aff}_{\bar F}\longrightarrow \mathbf{Grps}, \quad \mathrm{Spec}(R)\longmapsto N(R(\!(\xi)\!))
\]
by left-multiplication on $\mathscr G_G$. Using the (non-canonical) identification $\BpdR\simeq C^{\sharp}[\![\xi]\!]$, we have $\lambda\in S_{M}(\mu;C,C^+)$ if and only if 
\[
N(C^{\sharp}[\![\xi]\!])\lambda^{-1}(\xi)\cap G(C^{\sharp}[\![\xi]\!])\mu^{-1}(\xi)  G(C^{\sharp}[\![\xi]\!])\neq\emptyset,\]
or equivalently $(\mathscr G_{G,\mu}\cap \mathscr S_{\lambda})(C^{\sharp})\neq \emptyset$. But  $\mathscr G_{G,\mu}\cap \mathscr S_{\lambda}$ is an $\bar F$-scheme of finite type, so it has a point with values in the algebraically closed field $C^{\sharp}$ if and only if $\mathscr{G}_{G,\mu}\cap \mathscr S_{\lambda}\neq \emptyset$. This shows that the set $S_{M}(\mu;C,C^+)$ is actually independent on the choice of the geometric point $\mathrm{Spa}(C,C^+)\ra \mathrm{Spd}(\bar F)$.  
\end{proof}


For the description of $S_M(\mu)$, we have an analogous result as \cite[Lemma 5.4.1]{GHKR}.

\begin{lemma} Keep the notations above. 
\begin{enumerate}
\item There are inclusions of finite sets 
\[\Sigma(\mu)_{M-max}\subseteq S_{M}(\mu)\subseteq \Sigma (\mu)_{M-dom},
\]
where $\Sigma(\mu)_{M-dom}$ denotes the set of $M$-dominant cocharacters $\lambda\in X_*(T)$  such that $(\lambda)_{dom}\leq \mu$, and  $\Sigma(\mu)_{M-max}$ denotes the set of elements in  $\Sigma(\mu)_{M-dom}$ that are maximal with respect to the partial order $\leq_M$. In particular, when $\mu$ is minuscule or $M=T$ is the maximal torus, the inclusions become equalities.

\item  For any $\lambda\in S_M(\mu)$, $\kappa(\lambda)=\kappa(\mu)$ in $\pi_1(G)_{\Gamma}$.
\end{enumerate}
\end{lemma}

\begin{proof} The assertion (2) is obvious.

For (1), let us start with the case where $M=T$ is the maximal torus (and thus $N=U$ is the unipotent radical of the Borel subgroup $B\subset G$), $\Sigma(\mu)_{M-max}=S_M(\mu)=\Sigma(\mu)$. In this case, our lemma says that for $\lambda\in X_*(T)$
\[
U(\BdR(C^{\sharp}))\lambda^{-1}(\xi)\cap G(\BpdR(C^{\sharp}))\mu^{-1}(\xi)G(\BpdR(C^{\sharp}))\neq \emptyset 
\]
if and only if $(\lambda)_{dom}\leq \mu$. Using a non-canonical identification $\BpdR(C^{\sharp})\simeq C^{\sharp}[\![\xi]\!]$, the desired assertion follows from 
the geometry of semi-infinite orbits of the usual affine Grassmannian for the group $G$ (cf. \cite[3.2 Theorem]{MV}, or  \cite[Lemma 5.3.7 and Theorem 5.3.9]{Zhu}). The proof for the general standard Levi $M\subset G$ is then the same as that of \cite[Lemma 5.4.1]{GHKR}.
\end{proof}

\begin{remark}\begin{enumerate}\item In \cite[Proposition 9.4]{Ngu}, Nguyen also studies $S_M(\mu)$ when $G=\mathrm{GL}_n$ using categorical local Langlands correspondences. 

\item \label{remark_GrPv finite}Fix $\mu$, there are only finitely many $v\in X_*(M^{ab})$ such that $\mathrm{Gr}_P^v\cap \mathrm{Gr}_{G, \mu}\neq \emptyset$. Indeed, \[\coprod_{\lambda\in X_*(T)\atop \lambda=v\text{ in }\pi_1(M^{ab})}\mathrm{Gr}_{P, \lambda}\longrightarrow \mathrm{Gr}_P^v\] induces a bijection on geometric points. By Lemma \ref{lemma_reduction type},  $\mathrm{Gr}_P^v\cap \mathrm{Gr}_{G, \mu}\neq \emptyset$ if and only if $v$ is contained in the image of $S_M(\mu)$ via $\theta: X_*(T)\rightarrow X_*(M^{ab})$, if and only if $v\in \theta(W\mu)$ by the previous lemma.
\end{enumerate}
\end{remark}

\subsubsection{$\BpdR$-affine Grassmannian and modifications of $G$-bundles} Let $S$ be a perfectoid space over $\mathbb F_q$ lying above $\mathrm{Spd}(\breve F)$, given by an untilt $S^{\sharp}$ of $S$ over $\breve F$. Let $b\in G(\breve F)$. Consider the associated $G$-bundle $\mathcal E_b$ on the relative Fargues-Fontaine curve $X_S$ (whose construction is recalled in \S~\ref{subsubsection:G-bundles}). An element $x\in \mathrm{Gr}_G(S)$ gives rise to a modification
\[
\mathcal E_{b,x}
\]
of $\mathcal E_b$ \`a la Beauville-Laszlo whose construction is recalled briefly as follows. Assume without loss of generality that $S=\Spa(R,R^+)$ is a perfectoid affinoid space, $S^{\sharp}=\Spa(R^{\sharp}, R^{\sharp +})$. 
Let $x\in \mathrm{Gr}_{G}(S)$ correspond to the $G$-bundle $\mathcal F$ on $\BpdR(R^{\sharp})$ together with a trivialization of $G$-bundles
\[
\iota: G|_{\Spec(\BdR(R^{\sharp}))}\stackrel{\sim}{\longrightarrow}\mathcal F|_{\Spec(\BdR(R^{\sharp}))}
.\]
Recall that the completion of $X_{S}$ along  the  Cartier divisor $\Spec(R^{\sharp})\hookrightarrow X_{S}$ is given by the de Rham period ring $\BpdR(R^{\sharp})$ (cf. \cite[1.33]{Far1}). Moreover, the $G$-bundle $\mathcal E_{b}$ is canonically trivialized on $X_{S}\setminus \Spec(R^{\sharp})$. So by a theorem of Beauville-Laszlo (cf. \cite{BL}), we can glue $\mathcal E_{b}$ and the $G$-bundle $\mathcal F$ along the trivialization $\iota$. In this way we obtain the  $G$-bundle $\mathcal E_{b,x}$ on $X_{S}$.

\begin{proposition}
Assume $S=\Spa(C,C^+)$. Let $x\in \mathrm{Gr}_{G,\mu}(S)$, and Write $\mathcal E_{b,x}=\mathcal{E}_{b'}$, with $[b']\in B(G)$. Then $\kappa(b')=\kappa(b)-\mu^{\#}\in \pi_{1}(G)_{\Gamma}$.
\end{proposition}

\begin{proof}The $b=1$ case is proved in \cite[Lemma 3.5.5]{CS}. For general $b$, the proof in loc. cit. also shows that we may reduce to the case $G=\mathbb{G}_m$ which can be done by explicit computation.
\end{proof}

Recall also the following useful fact.

\begin{proposition}[{\cite[Lemma 2.6]{CFS}, \cite[Lemma 3.10]{Vi}}] Assume $G$ quasi-split. Let $M\subset G$ be a standard Levi subgroup.
Suppose that $(b_M,g)$ is a reduction of $b$ to $M$: so $b_M\in M(\breve F)$ and $g\in G(\breve F)$ such that $b=gb_M\sigma(g)^{-1}$. For a geometric point $x\in \mathrm{Gr}_G(C,C^+)$, also viewed as an element of $\mathrm{Gr}_P(C,C^+)$ via the bijection
\[
\mathrm{Gr}_P(C,C^+)\stackrel{\sim}{\longrightarrow}\mathrm{Gr}_G(C,C^+)
\]
given by the Iwasawa decomposition, we have a natural isomorphism
\[
(\E_{b, g\cdot x})_P\times^P M\simeq (\E_{b_M,x})_P\times^PM \simeq \E_{b_M, \mathrm{pr}_M(x)},
\] where $\mathrm{pr}_M: \mathrm{Gr}_P\rightarrow \mathrm{Gr}_M$ is the natural projection. 
\end{proposition}

\begin{proof}Indeed, \cite[Lemma 2.6]{CFS}  and \cite[Lemma 3.10]{Vi}  deals with the case $b$ basic. Their proof could be adapted for general $b$ as follows. 
Since $b=gb_M\sigma(g)^{-1}$, the element $g$ gives an isomorphism of $G$-bundles 
\[
\mathcal E_{b_M,x}\stackrel{\sim}{\longrightarrow}\mathcal E_{b,g\cdot x}. 
\]
On the other hand, as $b_M\in M(\breve F)$ the $G$-bundle $\mathcal E_{b_M,x}$ has a natural reduction to the $P$-bundle $(\mathcal E_{b_M,x})_P=\mathcal E_{b_M,x}^P$ where we view $x\in \mathrm{Gr}_{G}(C,C^+)$ as an element of $\mathrm{Gr}_{P}(C,C^+)$ using the Iwasawa decomposition, yielding a reduction $(\mathcal E_{b,g\cdot x})_P$ to $P$ of $\mathcal E_{b,g\cdot x}$. By functoriality $\mathcal E_{b_M,x}^P\times^PM\simeq \mathcal E_{b_M,\mathrm{pr}_M(x)}$, so we get the required isomorphisms  
\[
(\E_{b, g\cdot x})_P\times^P M\simeq (\E_{b_M,x})_P\times^PM \simeq \E_{b_M, \mathrm{pr}_M(x)}.
\]
\end{proof}

\section{Harder-Narasimhan filtration of normed isocrystals}

Let $L$ be a field extension of $\breve F$. Recall that a filtered isocrystal over $L/\breve F$ consists of an isocrystal $D$ over $\breve F|F$ together with a separated exhaustive decreasing filtration by $L$-subspaces of $D\otimes_{\breve F}L$. When $L/\breve F$ is of finite degree, Fontaine introduced the notion of weak admissibility for filtered isocrystals, and together with Colmez, they showed in \cite{CF} that the latter can be used to characterize those filtered isocrystals coming from crystalline representations. Recently, building on their discovery of the fundamental curve (that is, the Fargues-Fontaine curve), Fargues and Fontaine gave an elegant geometric proof of this classical result in \cite[\S~10.5.3]{FF}: a key observation is that, as $L/\breve F$ is finite, a filtration on $D\otimes_{\breve F}L$ corresponds naturally to a $\mathrm{Gal}(\bar L/L)$-invariant $\BpdR$-lattice inside $D\otimes_{\breve F}\BdR$, hence one can apply the powerful theory of modifications of vector bundles on Fargues-Fontaine curves. Here $\BdR$ and $\BpdR$ are the de Rham period rings corresponding to $C$, the $p$-adic completion of $\bar L$. However, such a correspondence between general filtrations and lattices does not exist in general. Hence, this will cause some trouble once we are working with filtered isocrystals in the geometric setting.

To overcome the difficulty mentioned above, a natural idea is that, instead of considering filtrations on $D\otimes_{\breve F}C$, we should work directly with the $\BpdR$-lattices inside $D\otimes_{\breve F}\BdR$. When $D$ is semi-stable of slope $0$, this point of view has been already adopted in the recent work of Cornut-Peche Irissarry (\cite{CPI}), Viehmann (\cite{Vi}), Nguyen-Viehmann (\cite{NV}) and Shen (\cite{Sh}). In this work, we develop a theory of isocrystals equipped with a $\BpdR$-lattice without restriction on the slope of the isocrystals. In particular, there exists a natural Harder-Narasimhan formalism for such objects. As the main result of this section, we show that the corresponding Harder-Narasimhan filtration is compatible with tensor products (Theorem  \ref{thm_compatible tensor}). Later on we will apply this crucial result to construct the Harder-Narasimhan stratification on the $\BpdR$-affine Grassmannian for a general reductive group.

\subsection{The category $\mathbf{NIsoc}_{\breve F|F}^{K}$} \label{sec:NIsoc}

Let $F$ be a $p$-adic local field, $\breve{F}$ the completion of a maximal unramified extension of $F$, with $\sigma$ the Frobenius $F$-automorphism of $\breve{F}$. Let $K$ be a henselian discrete valuation field, such that its ring of integers $\cO_K$ contains $\breve F$. Let $\pi\in \cO_K$ be a uniformizer. Recall that a norm $\alpha$ on a finite dimensional $K$-vector space $\mathcal V$ is called \emph{splittable} if there exists a $K$-basis $\underline{e}=(e_1,\ldots,e_r)$ of $\mathcal V$ such that
\[
\alpha(v)=\max\{|\lambda_i|\alpha(e_i)\}, \quad \textrm{for all }v=\sum\lambda_ie_i\in \mathcal V.
\]
In this case we say that $\underline{e}$ is an \emph{orthogonal basis} of $(\mathcal V,\alpha)$.

\begin{remark}
In \cite[\S~5.2]{Cor}, Cornut considers normed vector spaces over a  general henselian non-archimedean field. Here for simplicity, we only consider those henselian non-archimedean fields which are discretely valued. In this setting, if $K$ is moreover complete, by \cite[Proposition 1.5 (i)]{BT}, every norm on $\mathcal V$ is splittable.
\end{remark}

\begin{definition}
\begin{enumerate}
\item A \emph{$K$-normed isocrystal} over $\breve F|F$ is a triple $
(D,\phi,\alpha)$
consisting of an isocrystal $D=(D,\phi)$ over $\breve F|F$ equipped with a splittable $K$-norm $\alpha$ on the $K$-vector space $D\otimes_{\breve F}K$.
\item A morphism
\begin{equation*}\label{eq:morphism-in-NIso}
f:(D_1,\phi_1,\alpha_1)\longrightarrow (D_2,\phi_2,\alpha_2)
\end{equation*}
of $K$-normed isocrystals is a morphism $f:D_1\ra D_2$ of $\breve F|F$-isocrystals such that
\[
\alpha_2((f\otimes 1)(v))\leq \alpha_1(v), \quad \textrm{for all }v\in D_{1,K}:=D_1\otimes_{\breve F}K.
\]
\end{enumerate}
\end{definition}

\begin{remark}
If there is no possible confusion, a normed isocrystal $(D,\phi,\alpha)$ is often simply denoted by $(D,\alpha)$,  that is, we omit the Frobenius $\phi$ from the notation.
\end{remark}
We denote by
\[
\mathbf{NIsoc}_{\breve F|F}^{K}
\]
the category of $K$-normed isocrystals over $\breve F|F$. Like the category $\mathbf{Norm}_K$ of normed $K$-vector spaces considered in \cite[5.2]{Cor}, $\mathbf{NIsoc}_{\breve F|F}^{K}$ is a quasi-abelian $\otimes$-category. For example, the tensor product of two objects in $\mathbf{NIsoc}_{\breve F|F}^K$ is given by the formula
\[
(D_1,\alpha_1)\otimes(D_2,\alpha_2):=(D_1\otimes_{\breve F}D_2,\alpha_1\otimes\alpha_2),
\]
where $D_1\otimes_{\breve F}D_2$ is the tensor product of isocrystals, and for $v\in (D_1\otimes D_2)_K\simeq D_{1,K}\otimes_K D_{2,K}$,
\[
(\alpha_1\otimes\alpha_2)(v):=\min\left\{\max_i\left\{\alpha_1(v_{1,i})\alpha_2(v_{2,i})\right\} \ \bigg| \  v=\sum_i v_{1,i}\otimes v_{2,i}\in D_{1,K}\otimes D_{2,K}\right\}.
\]
This formula indeed defines a splittable $K$-norm on $(D_1\otimes D_2)_K$.

\subsection{Harder-Narasimhan filtration for normed isocrystals} There exists a theory of Harder-Narasimhan filtration for normed isocrystals. To see this, recall the following fundamental result.
\begin{theorem} Let $\mathcal V$ be a finite dimensional $K$-vector space. Let $\alpha$ and $\beta$ be two splittable norms on $\mathcal V$.
\begin{enumerate}
\item There exists a $K$-basis $\underline e=(e_1,\ldots, e_r)$ of $\mathcal V$ which is orthogonal for both of the norms $\alpha$ and $\beta$.
\item Write $\alpha(e_i)=|\pi|^{\lambda_i}$ and $\beta(e_i)=|\pi|^{\mu_i}$ for $1\leq i\leq r$. The following quantity
\[
\nu(\alpha,\beta):=\sum_i \left(\mu_i-\lambda_i\right)
\]
does not depend on the choice of the basis $\underline e$.
\item For $\gamma$ a third splittable norm on $\mathcal V$, we have
\[
\nu(\alpha,\gamma)=\nu(\alpha,\beta)+\nu(\beta,\gamma).
\]
\end{enumerate}
\end{theorem}

\begin{proof}
In our setting, $K$ is supposed to be discretely valued, so the assertion (1) follows from \cite[Proposition 1.26]{BT}. Moreover, for $\alpha$ a splittable norm on $\mathcal V$ and for $\underline{e}=(e_1,\ldots, e_r)$ an orthogonal basis of $(\mathcal V,\alpha)$, by \cite[Proposition 1.9]{BT}, the quantity
\[
 \mathrm{vol}(\alpha):=\mathrm{vol}(\underline{e})-\sum_i\lambda_i
\]
does not depend on the choice of the basis $\underline e$. Here $\mathrm{vol}(\underline e)$ is the volume of the lattice generated by the basis $\underline e$ of $\mathcal V$ relative to some fixed lattice of $\mathcal V$. In particular, we have
\[
 \nu(\alpha,\beta)=\mathrm{vol}(\alpha)-\mathrm{vol}(\beta)
 \]
 from which we get immediately (2) and (3).
\end{proof}

\begin{remark} As indicated in \cite[\S~5.2.1]{Cor}, the previous theorem also holds for a general henselian non-archimedean base field. For example, (1) is confirmed in \cite[Appendice, page 300]{BT} , while (2) and (3) are proved in \cite{Cor20}.
\end{remark}

\begin{definition}\label{def:degree-of-normed-isocrystals}
Let $(D,\alpha)$ be a normed isocrystal.
\begin{enumerate}
\item The \emph{rank} and the \emph{degree} of $(D,\alpha)$ are defined respectively by 
\begin{eqnarray*}
\mathrm{rank}(D,\alpha)& =&  \dim_{\breve F}D \\ \deg(D,\alpha)& =& \nu(\mathbf{o},\alpha)-\dim(D),
\end{eqnarray*}
with $\mathbf{o}$ the splittable norm on $D_K$ given by 
\[
\mathbf{o}(v)=\inf\{|\lambda| : \lambda\in K, v\in \lambda (D\otimes_{\breve F}\cO_K)\subset D_K\},
\]
and $\dim(D)$ the dimension of the isocrystal $D=(D,\phi)$, that is, the integer $v_{\breve F}(\det \phi)$ with $v_{\breve F}$ the normalized additive valuation on $\breve F$. 
    \item When $D\neq 0$,  the \emph{slope} $\rho(D,\alpha)$ of $(D,\alpha)$ is defined by the following ratio
    \[
    \rho(D,\alpha):=\frac{\deg(D,\alpha)}{\mathrm{rank}(D,\alpha)}\in \mathbb R.
    \]
      \item Assume $D\neq 0$ and write $\rho=\rho(D,\alpha)$. We say that the normed isocrystal $(D,\alpha)$ is \emph{semi-stable of slope $\rho$} if for any non-zero subisocrystal $D'\subset D$, the normed isocrystal $(D',\alpha|_{D'})$ is of slope $\leq \rho$.
\end{enumerate}
\end{definition}

\begin{remark}\label{remark_deg_deg'}
Let $(D,\alpha)$ be a normed isocrystal over $\breve F|F$ with underlying $\breve F$-vector space $V$. Then $\nu(\mathbf{o},\alpha)$ is the degree (defined in \cite[\S~5.2.4]{Cor}) of the normed vector space $(V,\alpha)$ obtained from $(D,\alpha)$ by forgetting the Frobenius. 
\end{remark}

\begin{remark}
As in \cite[5.2.4]{Cor}, the functions $\mathrm{rank}$ and $\deg$ are additive on short exact sequences and respectively constant and non-decreasing on mono-epis, and for
\[
f:(D_1,\alpha_1)\longrightarrow (D_2,\alpha_2)
\]
a mono-epis, $f$ is an isomorphism of normed isocrystals if and only if \[
\deg(D_1,\alpha_1)=\deg(D_2,\alpha_2).
\]
\end{remark}

Consider the following faithful exact $F$-linear functor
\[
\omega: \mathbf{NIsoc}_{\breve F|F}^{K}\longrightarrow \mathbf{Isoc}_{\breve F|F}, \quad (D,\alpha)\mapsto D.
\]
Here $\mathbf{Isoc}_{\breve F|F}$ is the category of isocrystals over $\breve F|F$. Like the case considered by Cornut in \cite{Cor}, the functions $\mathrm{rank}$ and $\deg$ above give rise to a Harder-Narasimhan formalism on $\mathbf{NIsoc}_{\breve F|F}^{K}$. More precisely, for $(D,\alpha)$ a normed isocrystal over $\breve F|F$, there exists an exhaustive separated decreasing $\mathbb R$-filtration
\[
\mathcal F_{\rm HN}(D,\alpha)=(\Fil^sD)_{s\in \mathbb R}
\]
of $D$ by sub-isocrystals such that for each $s\in \mathbb{R}$, 
\[
\Fil^{s}D=\Fil^{s,-}D:=\bigcap_{s'<s}\Fil^{s'}D, 
\]
and that the isocrystal
\[
\mathrm{gr}^s_{\mathcal F_{\rm HN}(D,\alpha)}D:=\Fil^sD/\Fil^{s+}D
\]
equipped with the norm induced from $\alpha$, is either trivial or semi-stable of slope $s$. We call $\mathcal F_{\rm HN}(D,\alpha)$ the \emph{Harder-Narasimhan filtration} of the normed isocrystal $(D,\alpha)$. Since the degree function on $\mathbf{NIsoc}_{\breve F|F}^K$ is additive on short exact sequences, we have 
\[
\deg(\mathcal F_{\rm HN}(D,\alpha)):=\sum_{s\in \mathbb R} s\dim_{\breve F}\mathrm{gr}^s_{\mathcal F_{\rm HN}(D,\alpha)}D=\deg(D,\alpha).
\] 
The proof of the following properties of Harder-Narasimhan filtration is standard. We refer to Example \ref{ex:direct-sum-and-tensor-of-fil} below for the definition of the direct sum of two filtrations. 

\begin{proposition}\label{prop:HN-filtration-and-direct-sim} Let $(D,\alpha)$ and $(D',\alpha')$ be two normed isocrystals. Then 
\[
\mathcal F_{\rm HN}(D\oplus D',\alpha\oplus \alpha')=\mathcal F_{\rm HN}(D,\alpha)\oplus  \mathcal{F}_{\rm HN}(D',\alpha').
\]
In other words, the Harder-Narasimhan filtration for normed isocrystals is compatible with direct sums.    
\end{proposition}

\subsection{Compatibility of the Harder-Narasimhan filtration with tensor products} Let
\[
\mathbf{F}(\mathbf{Isoc}_{\breve F|F})
\]
be the category of pairs $(D,\mathcal F)$ with $D$ an isocrystal over $\breve F|F$ and $\mathcal F=(\Fil^iD)_{i\in \mathbb R}$ a finite exhaustive separated decreasing $\mathbb R$-filtration of $D$ by sub-isocrystals. The formalism of  Harder-Narasimhan filtration for normed isocrystals provides a functor
\begin{equation}\label{eq:F_HN}
\mathcal F_{\rm HN}:\mathbf{NIsoc}_{\breve F|F}^{K}\longrightarrow \mathbf{F}(\mathbf{Isoc}_{\breve F|F}),\quad (D,\alpha)\longmapsto \mathcal F_{\rm HN}(D,\alpha). \end{equation}
Note that the category $\mathbf F(\mathbf{Isoc}_{\breve F|F})$ is naturally a quasi-abelian $F$-linear rigid $\otimes$-category, so one may wonder if the functor $\mathcal F_{\rm HN}$ is compatible with tensor products. The main result of this section is the following theorem.
\begin{theorem}\label{thm_comp tensor}
The functor $\mathcal F_{\rm HN}$ in \eqref{eq:F_HN} is compatible with tensor products. In other words, for $(D,\alpha)$ and $(D',\alpha' )$ two normed isocrystals, we have 
\[
\mathcal F_{\rm HN}(D\otimes D,\alpha\otimes \alpha' )=\mathcal{F}_{\rm HN}(D,\alpha)\otimes \mathcal{F}_{\rm HN}(D' ,\alpha' ).
\]\end{theorem}
We refer to Example \ref{ex:direct-sum-and-tensor-of-fil} for  the definition of the tensor product of two filtrations.

\begin{remark}
It's easy to verify that the assertion of the previous theorem is equivalent to the following: the tensor product of two semi-stable normed isocrystals is still semi-stable.  
\end{remark}

The analogue of the theorem above for normed vector spaces is due to Cornut (cf. \cite[Theorem 5.8]{Cor}). Our proof of Theorem \ref{thm_comp tensor} given below is motivated by his proof of loc. cit.

\subsubsection{The vectorial Tits building for $\mathbf{GL}_n$} Let $L$ be a field, and $V$ a finite-dimensional $L$-vector spaces. 
Let $\mathbf F(V)$ be the set of separated exhaustive decreasing $\R$-filtrations $f=(V_f^a)_{a\in \mathbb R}$ of $V$ by sub-spaces so that 
\[
V_f^{a}=V_f^{a,-}:=\bigcap_{a'<a}V_{f}^{a'},\quad \textrm{for all } a\in \mathbb R.
\] 
Let $\underline{e}=(e_1,\ldots, e_n)$ be a $K$-basis of $V$. We say that an $\mathbb R$-filtration $f\in \mathbf{F}(V)$ is splitted by $\underline e$ if each subspace $V_f^a$ is spanned by a subset of $\{e_1,\ldots, e_n\}$. We have a natural injective map 
\begin{equation}\label{eq:appartement}
\mathbb R^n\longrightarrow \mathbf F(V), 
\end{equation}
sending $\alpha=(a_1,\ldots,a_n)\in \mathbb R^n$ to the $\mathbb R$-filtration $f_{\alpha}$ of $V$ given by setting 
\[
V_{f_{\alpha}}^a:=\mathrm{Span}\{e_i| a_i\geq a\}\subset V.
\]
The image of \eqref{eq:appartement} is an \emph{apartment} of $\mathbf{F}(V)$, denoted by $\mathbf{F}(\underline{e})$. Hence, $\mathbf{F}(V)$ is the union of all its apartments. It is known that any two $\mathbb R$-filtrations of $V$ can be splitted by a same $L$-basis, thus are contained in a common apartment.

\begin{example} Let $f,g\in \mathbf{F}(V)$, and $\lambda\in \mathbb R_{>0}$. We can define two new elements $f+g$ and $\lambda \cdot f$ of $\mathbf F(V)$ by 
\[
V_{f+g}^a:=\sum_{b+c=a}V_f^b\bigcap V_g^c \quad \textrm{and}\quad V_{\lambda\cdot f}^a:=V_{f}^{\lambda^{-1}a}.
\]
Note that the formula defining $f+g$ makes sense since the sets 
\[
\{V_f^b \ |\  b\in \mathbb R\}\quad \textrm{and} \quad \{V_g^c\ |\ c\in \mathbb R\}
\]
consist of only finitely many subspaces of $V$. Moreover, if $f,g$ are contained in the apartment $\mathbf{F}(\underline{e})$ for a certain $K$-basis, and write $f=f_{\alpha}$ and $g=f_{\beta}$ for some $\alpha,\beta\in \mathbb R^n$, then $f+g=f_{\alpha+\beta}$ and $\lambda \cdot f=f_{\lambda \alpha}$. 
\end{example}


\begin{example}\label{ex:direct-sum-and-tensor-of-fil} Let $V,V'$ be two $L$-vector spaces. Let $f,f'$ be two $\mathbb R$-filtrations of $V$ and $V'$ respectively. 

\begin{enumerate}
\item The direct sum of $f$ and $g$, denoted by $f\oplus g$, is the $\mathbb R$-filtration of $V\oplus V'$ given by setting 
    \[
    V_{f\oplus f'}^a:=V_f^a\oplus V_{f'}^{'a}\subset V\oplus V', \quad \forall a\in \mathbb R. 
    \]
    Clearly, 
    \[
    \mathrm{gr}_{f\oplus f'}^a(V\oplus V')=\mathrm{gr}_{f}^aV\oplus \mathrm{gr}_{f'}^aV'.
    \]  
\item The tensor product of $f$ and $g$, denoted by $f\otimes g$, is the $\mathbb R$-filtration on $V\otimes V'$ given by setting (one checks easily that the formula below indeed makes sense)
\[
V_{f\otimes f'}^a:=\sum_{s+t=a}V_f^s\otimes V_{f'}^{'t}\subset V\otimes V', \quad \forall a\in \mathbb R.
\]
As for the graded pieces of $f\otimes f'$, we have the following well-known description 
    \[
    \mathrm{gr}_{f\otimes f'}^a= \bigoplus_{s+t=a}\mathrm{gr}_f^sV\otimes \mathrm{gr}_{f'}^tV'.
    \]
\end{enumerate}
    
\end{example}

The set $\mathbf{F}(V)$ is naturally a metric space (\cite[\S~2.3]{Cor}): let $f,g\in \mathbf F(V)$, and write
\begin{eqnarray*}
\langle f,g\rangle&: = & \sum_{s,t\in \mathbb R}st\dim \mathrm{gr}_{f}^s\left(\mathrm{gr}^t_{g}V\right), \\ \|f\|& :=& \sqrt{\langle f,f\rangle}, \quad \textrm{and}
\\ 
d(f,g)& :=& \sqrt{\|f\|^2+\|g\|^2-2\langle f,g\rangle }.
\end{eqnarray*}
\begin{proposition}\label{prop:F(V)-complete}
  The function $d(-.-)$ defines a metric on $\mathbf F(V)$, and the metric space  $(\mathbf{F}(V),d)$ is complete.   
\end{proposition}

\begin{proof}
This can be viewed as the special case of some general results of Cornut. Indeed, as $V$ is a finite-dimensional $K$-vector space, the set $\mathscr X$ of all $L$-subspaces, equipped with the partial order $\leq $ given by the inclusion relation of subspaces, is a bounded modular lattice of finite length $\dim_KV$, with 
\[
W_1\vee W_2:=W_1+W_2 \quad \textrm{and}\quad W_1\wedge W_2:=W_1\cap W_2. 
\]
Moreover, the set $\mathbf{F}(V)$ is naturally identified with the $\mathbf{F}(\mathscr X)$ of all $\mathbb R$-filtration on $\mathscr X$ considered in \cite[\S~2.2]{Cor}, and the map 
\[
\mathscr X\longrightarrow \mathbb R, \quad W\longmapsto \dim_{L}W
\]
gives a rank function (in the sense of \cite[\S~2.1.3]{Cor}) on $\mathscr X$. So our proposition follows from the corresponding general statements for bounded modular lattices of finite length established in \cite[Proposition 2.6,2.7]{Cor}.  
\end{proof}

\begin{remark} Let $\underline{e}$ be an $L$-basis of $V$. Restricting to the apartment $\mathbf{F}(\underline e)$, identified with $\mathbb R^n$ via \eqref{eq:appartement}, the pairing above is just the standard scale product $\langle -,-\rangle$ in $\mathbb R^n$, i.e., 
\[
\langle f_{\alpha},f_{\beta}\rangle =\langle \alpha,\beta\rangle, \quad \forall \ \alpha,\beta\in \mathbb R^n. 
\]
In particular, for $f,g,h\in \mathbf{F}(\underline e)$, we have 
\[
\langle f,g+h\rangle=\langle f,g\rangle +\langle f,h\rangle. 
\]
In general the pairing $\langle -,-\rangle $ is only concave (\cite[Lemma 2.5]{Cor}): we have 
\begin{equation}\label{eq:concave}
\langle f,g+h\rangle \geq \langle f,g\rangle +\langle f,h\rangle, \quad \forall f,g,h\in \mathbf F(V). 
\end{equation}
\end{remark}

Let $C\subset \mathbf{F}(V)$ be a closed subset. Assume moreover that it is \emph{convex}, i.e., for any $f,g\in C$
\[
\lambda \cdot f+(1-\lambda )\cdot g\in C, \quad \forall \lambda\in [0,1]. 
\] 
Then, for every $f\in \mathbf{F}(V)$, there exists a unique $p_C(f)\in C$, called the \emph{convex projection} of $f$ to $C$, such that 
\[
d(f,p_C(f))=\min\{d(f,g)| g\in C\}. 
\]
\begin{lemma}\label{lem:property-for-convex-projection} Let $f\in \mathbf{F}(V)$, with convex projection $p_C(f)\in C$. We have \[
|p_C(f)|\leq |f|.
\]
Moreover, if $g\in \mathbf F(V)$ so that $p_C(f)+t\cdot g\in C$ for  sufficiently small positive real $t$, we have 
\[
\langle p_C(f),g\rangle\geq \langle f,g\rangle.
\]
\end{lemma}

\begin{proof}
The inequality $|p_C(f)|\leq |f|$ follows from the fact that the convex projection $p_C$ is non-expanding (\cite[II.2.4]{BH}). The second asssertion is also well-known and let us recall the proof for the sake of completeness. Note that, for $t\in \mathbb R_{>0}$ with $p_C(f)+tg\in C$,  
\begin{eqnarray*}
d(f,p_C(f))^2 & \leq & d(f,p_C(f)+tg)^2 \\ & =& |f|^2+|p_C(f)+tg|^2-2\langle f,p_C(f)+tg\rangle \\ & =& |f|^2+|p_C(f)|^2+|g|^2t^2+2\langle p_C(f),g\rangle t -2\langle f,p_C(f)+tg\rangle \\ & \leq & |f|^2+|p_C(f)|^2+|g|^2t^2+2\langle p_C(f),g\rangle t -2\langle f,p_C(f)\rangle -2\langle f,tg\rangle \\& = & |f|^2+|p_C(f)|^2+|g|^2t^2+2\langle p_C(f),g\rangle t -2\langle f,p_C(f)\rangle -2t\langle f,g\rangle \\ & =& |g|^2t^2+2\left(\langle p_C(f),g\rangle-\langle f,g\rangle\right) t + |f|^2+|p_C(f)|^2-2\langle f,p_C(f)\rangle  \\ & =& |g|^2t^2+2\left(\langle p_C(f),g\rangle-\langle f,g\rangle\right) t + d(f,p_C(f))^2.
\end{eqnarray*}
Here the second inequality follows from the fact that the pairing $\langle -,-\rangle $ is concave (cf. \eqref{eq:concave}). Since these inequalities hold for any sufficiently small positive real number $t$, it follows that the 
polynomial  
\[
|g|^2t^2+2\left(\langle p_C(f),g\rangle-\langle f,g\rangle\right) t 
\]
attains its minimum at some $t_0\leq 0$. As a result, we find $\langle p_C(f),g\rangle \geq \langle f,g\rangle$, as required. 
\end{proof}

\subsubsection{Vectorial Tits building and Harder-Narasimhan filtration} We use the notations fixed at the beginning of \S~\ref{sec:NIsoc}. 
Let $(D,\alpha)$ be a normed isocrystal over $K/F$, with $V$ the underlying $\breve F$-vector space. Let $\mathbf{F}(D)\subset \mathbf{F}(V)$ be the subset of $\mathbb R$-filtrations of $D$ by subisocrystals. As in the proof of Proposition \ref{prop:F(V)-complete}, the set of all subisocrystals of $D$, equipped with the partial order $\leq$ given by the inclusion relation, is naturally a bounded modular lattice of finite length $\leq \dim_{\breve F}D$. So by \cite[Proposition 2.7]{Cor}, the subset $\mathbf{F}(D)$, equipped with the distance function induced from that on $\mathbf{F}(V)$ is complete, thus is closed in $\mathbf{F}(D)$. The norm $\alpha$ on $D\otimes_{\breve F}K$ defines a map 
\[
\langle \alpha,-\rangle : \mathbf{F}(D)\longrightarrow \mathbb R,\quad f\longmapsto \langle \alpha,f\rangle =\sum_{s\in \mathbb R}s\cdot \mathrm{deg}(\mathrm{gr}_f^sD,\alpha).
\]
Using the vectorial Tits building, the Harder-Narasimhan filtration for the normed isocrystal $(D,\alpha)$ has the following characterization:

\begin{proposition}[{\cite[Proposition 2.12]{Cor}}] The Harder-Narasimhan filtration for $(D,\alpha)$ is the unique filtration $\mathcal F\in \mathbf{F}(D)$ such that 
\[
|\mathcal F|^2-2\langle \alpha,\mathcal F\rangle\leq |f|^2-2\langle \alpha,f\rangle 
\]
for any $f\in\mathbf{F}(D)$. 
    
\end{proposition}

On can slightly strengthen the above inequality:

\begin{corollary}\label{cor:inequality-for-HN} Let $(D,\alpha)$ be a normed isocrystal, with $\mathcal F=\mathcal F_{\rm HN}(D,\alpha)$ its Harder-Narasimhan filtration. Then for any $f\in \mathbf{F}(D)$, 
\[
|\mathcal F|^2-2\langle \alpha,\mathcal F\rangle\leq |f|^2-2\langle \alpha,f\rangle-\frac{(\deg(f)-\deg(D,\alpha))^2}{\dim_{\breve F}D}. 
\]   \end{corollary}

\begin{proof}
For $r\in \mathbb R$, let $f_r$ be the $\mathbb R$-filtration on $D$ with 
\[
D_{f_r}^a:=D_f^{a-r}, \quad \textrm{for all }a\in \mathbb R.  
\]
Then 
\begin{eqnarray*}
|\mathcal F|^2-2\langle \alpha,\mathcal F\rangle & \leq &  
|f_r|^2-2\langle \alpha, f_r\rangle \\ & =& \sum_{a\in \mathbb R}(a+r)^2\dim_{\breve{F}}\mathrm{gr}_f^aD-2\sum_{a\in \mathbb R}(a+r)\deg(\mathrm{gr}_f^a,\alpha). \\ &  =& r^2\cdot \dim_{\breve F}D+2r\cdot (\deg(f)-\mathrm{deg}(D,\alpha))+|f|^2-2\langle \alpha, f\rangle.   \end{eqnarray*}
Since this holds for any $r$, taking $r=-\frac{\deg(f)-\deg(D,\alpha)}{\dim_{\breve F}D}$, we obtain the desired inequality. 
\end{proof}

\subsubsection{Compatibility of Harder-Narasimhan filtration for normed isocrystals with tensor products}

Let $(D,\alpha)$ and $(D',\alpha')$ be two normed isocrystals on $K/F$, with underlying $\breve F$-vector spaces $V$ and $V'$ respectively. Consider the isocrystal
\[
E=D\oplus D'\oplus (D\otimes D')
\]
equipped with the norm $\beta$ on $E_K$ induced from $\alpha$ and $\alpha'$. We have the following map 
\[
\mathbf F(D)\times \mathbf{F}(D')\longrightarrow \mathbf{F}(E), \quad (f,f' )\mapsto f\oplus f' \oplus (f\otimes f'). 
\]

\begin{lemma} The above map is injective, and identifies $\mathbf{F}(D)\times \mathbf{F}(D')$ as a closed convex subset of $\mathbf{F}(E)$. 
\end{lemma}

\begin{proof} The injectivity of the above map  is clear. So we can view $\mathbf{F}(D)\times \mathbf{F}(D')$ as a subset of $\mathbf{F}(E)$. To check that $\mathbf F(D)\times \mathbf F(D')$ is convex in $\mathbf F(E)$, we claim that, for $(f,f')$ and $(g,g')$ in $\mathbf F(D)\times \mathbf F(D')$, \[
(f+g)\otimes (f'+g')=f\otimes f'+g\otimes g'\in \mathbf F(D\otimes D').
\]
For this, let $\underline{e}=(e_1,\ldots,e_n)$ (resp.  $\underline e'=(e_{1}',\ldots, e_m')$) be an $\breve F$-basis that split $f$ and $g$ (resp. $f'$ and $g'$) at the same time, with the corresponding $n$-tuples (resp. $m$-tuples) $\alpha=(a_1,\ldots, a_n)$ and $\beta=(b_1,\ldots,b_n)$ (resp. $\alpha'=(a_1',\ldots,a_m')$ and $\beta'=(b_1',\ldots,b_m')$). Then the $\breve F$-basis $(e_i\otimes e_{j}')_{i,j}$ splits both the filtrations $(f+g)\otimes(f'+g')$ and $f\otimes f'+g\otimes g'$, and the corresponding $mn$-tuples are the same, given by 
\[
\left((a_i+b_i)+(a_j'+b_j' )\right)_{i,j}=\left( (a_i+a_j')+(b_i+b_{j} ')\right)_{i,j}, 
\]
giving our claim. 

It remains to check that $\mathbf{F}(D)\times \mathbf{F}(D')\subset \mathbf{F}(E)$ is closed. Let $(f,f')$ and $(g,g')$ are two elements in $\mathbf{F}(D)\times \mathbf F(D')$. Viewed as two elements in $\mathbf F(E)$, their distance is 
\[
\sqrt{d(f,g)^2+d(f',g')^2+d(f\otimes f',g\otimes g')^2}.
\]
On the other hand, consider the $\breve F$-bases $\underline e$ and $\underline e '$, the tuples $\alpha, \alpha',\beta$ and $\beta'$ as above, then we have  
\begin{eqnarray*}
d(f\otimes f',g\otimes g')^2 & = & \sum_{i=1}^n\sum_{j=1}^m (a_i+a_j'-b_i-b_j')^2 \\ & \leq & \sum_{i=1}^n\sum_{j=1}^m 2\left((a_i-b_i)^2+(a_j'-b_j')^2\right) \\ & =& 2m\cdot d(f,g)^2+2n\cdot d(f',g'), 
\end{eqnarray*}
with $m=\dim_{\breve F}D'$ and $n=\dim_{\breve F}D$. Consequently, 
\begin{eqnarray*}
\sqrt{d(f,g)^2+d(f',g')^2} & \leq&  \sqrt{d(f,g)^2+d(f',g')^2+d(f\otimes f',g\otimes g')^2} \\ & \leq & \sqrt{(2m+1)d(f,g)^2+(2n+1)d(f',g')^2}.
\end{eqnarray*}
In particular, the distance function on $\mathbf{F}(D)\times \mathbf F(D')$ induced from the distance function on $\mathbf{F}(E)$ is equivalent to the product of the distance functions on $\mathbf{F}(D)$ and on $\mathbf{F}(D')$. Since $(\mathbf{F}(D),d)$ and $(\mathbf{F}(D'),d)$ are complete, it follows that as a subset of $\mathbf {F}(E)$, $\mathbf{F}(D)\times \mathbf{F}(D')$ is also complete, thus is closed, as required. 
\end{proof}

As a result, we can consider the convex projection of $\mathbf{F}(E)$ to the closed convex subset $\mathbf{F}(D)\times \mathbf {F}(D')$, denoted by 
\[
p: \mathbf{F}(E)\longrightarrow \mathbf{F}(D)\times \mathbf{F}(D'). 
\]
in the sequel. Recall that the normed isocrystals have underlying $\breve{F}$-vector spaces $V$ and $V'$ respectively. Let 
\[
W=V\oplus V'\oplus (V\otimes V')
\]
be the underlying $\breve{F}$-vector space of the isocrystal $E$. A similar argument shows that the natural map 
\begin{equation}\label{eq:V-tensor-Vprime-to-W}
\mathbf{F}(V)\times \mathbf{F}(V')\longrightarrow F(W), \quad (f,f' )\longmapsto f\oplus f'\oplus (f\otimes f') 
\end{equation}
identifies $\mathbf{F}(V)\times \mathbf{F}(V' )$ as a closed convex subset of $\mathbf{F}(W)$. Let us denote the corresponding convex projection by $q$. 

\begin{lemma}\label{lem:two-convex-projections} The square below is commutative:
\[
\xymatrix{\mathbf{F}(E)\ar[r]^<<<<<p\ar@{^(->}[d] & \mathbf{F}(D)\times \mathbf{F}(D') \ar@{^(->}[d] \\ \mathbf{F}(W)\ar[r]^<<<<<{q} & \mathbf{F}(V)\times \mathbf{F}(V')}
\]
\end{lemma}

\begin{proof} By the uniqueness of convex projection, it suffices to check that, for $f\in \mathbf{F}(E)$ viewed as an element in $\mathbf{F}(W)$, its convex projection $q(f)\in \mathbf{F}(V)\times \mathbf F(W)$ is an $\mathbb R$-filtration by subisocrystals. To show this, by abuse of notation, we denote by $\phi$ the Frobenius on the related isocrystals in the proof. So for $\mathcal D\in \{D, D',E\} $, with underlying $\breve F$-vector space $\mathcal V$, we have an isomorphism of $\breve F$-spaces
\[
\phi\otimes 1: \mathcal V\otimes_{\breve F,\sigma}\breve F\longrightarrow \mathcal V.
\]
In particular, for $g\in \mathbf{F}(\mathcal V)$, we obtain a filtration $
\phi(g)$ of $\mathcal V$ by setting 
\[
\mathcal V_{\phi(g)}^a=(\phi\otimes 1)(\mathcal V_{g}^a)\subset \mathcal V. 
\]
As the Frobenius map $\phi$ is bijective, we obtain bijection
\[
\phi: \mathbf{F}(\mathcal V)\rightarrow \mathbf{F}(\mathcal V)
\]
and $\mathbf{F}(\mathcal D)\subset \mathbf{F}(\mathcal V)$ is precisely the subset of elements fixed by $\phi$. For $g'$ a second filtration of $\mathbf{F}(\mathcal V)$, we have 
\[
\langle g,g'\rangle=\langle \phi(g),\phi(g')\rangle
\]
and thus $d(\phi(g),\phi(g' ))=d(g,g' )$. Moreover, the map \eqref{eq:V-tensor-Vprime-to-W} is equivariant with respect to the bijection $\phi$ on both sides. 
It follows that, for $f\in \mathbf{F}(E)\subset \mathbf F(W)$, $q(f)$ and $\phi(q(f))$ are two elements in $\mathbf F(V)\times \mathbf{F}(V' )$ minimizing the function
\[
\mathbf{F}(V)\times \mathbf{F}(V')\ni g\longmapsto d( f,g).
\]
As a result,  $\phi(q(f))=q(f)\in \mathbf{F}(V)\times \mathbf{F}(V')$ and thus $q(f)\in \mathbf{F}(D)\times \mathbf{F}(D' )$, as required. 
\end{proof}

We are now ready to prove the main result of this section. 

\begin{proof}[Proof of Theorem \ref{thm_comp tensor}]
Let 
\[
\mathcal F=(\mathcal F_{\rm HN}(D,\alpha),\mathcal F_{\rm HN}(D' ,\alpha' ))\in \mathbf{F}(D)\times \mathbf{F}(D')
\]
and view it as the  element 
\[\mathcal F_{\rm HN}(D,\alpha)\oplus \mathcal F_{\rm HN}(D',\alpha')\oplus \left(\mathcal F_{\rm HN}(D,\alpha)\otimes \mathcal F_{\rm HN}(D',\alpha')\right)\]
in $\mathbf {F}(E)$. We want to show that this coincides with the Harder-Narasimhan filtration $\mathcal F_{\rm HN}(E,\beta)$ for $(E, \beta)$. For this, it suffices to check that $\mathcal F$ satisfies the inequality characterizing the Harder-Narasimhan filtration $\mathcal F_{\rm HN}(E,\beta)$:
\begin{equation}\label{eq:inequality-for-HN}
|\mathcal F|^2-2\langle \beta,\mathcal F\rangle \leq |f|^2-2\langle \beta,f\rangle, \quad \forall f\in \mathbf{F}(E). 
\end{equation}

For the moment, let us check the inequality \eqref{eq:inequality-for-HN} for $f=(g,g' )\in \mathbf{F}(D)\times \mathbf F(D')\subset \mathbf F(E)$. Firstly, 
\begin{eqnarray*}
|f|^2& =& \sum_{a\in \mathbb R} a^2\dim \mathrm{gr}_f^ a(D\oplus D' \oplus D\otimes D' )    \\ & =& \sum_{a\in \mathbb R} a^2\left(\dim \mathrm{gr}_g^ a(D)\oplus \mathrm{gr}_{g'}^a(D') \oplus \mathrm{gr}_{g\otimes g' }^ a(D\otimes D') )\right)\\ & =& |g|^2+|h|^2+\sum_{a}a^2\dim\left(\oplus_{r+s=a}\mathrm{gr}_g^r D\otimes \mathrm{gr}_{g'}^s D'  \right) \\ & =& |g|^2+|g' |^2+\sum_{r,s\in \mathbb R}(r+s)^2\dim(\mathrm{gr}_g^r D)\cdot \dim(\mathrm{gr}_{g'}^s D')  \\ & =& (\dim D'+1)|g|^2+(\dim D+1)|g'|^2+2\deg(g)\deg(g').  
\end{eqnarray*}
Secondly, 
\begin{eqnarray*}
\langle \alpha\otimes \alpha',g\otimes g' \rangle  & =& \sum_{a\in \mathbb R}a \cdot \deg(\mathrm{gr}_{g\otimes g' }^a,\alpha\otimes \alpha' ) \\ & =& \sum_{r,s\in \mathbb R}(r+s)\deg(\mathrm{gr}_g^rD\otimes \mathrm{gr}_{g' }^s D',\alpha\otimes \alpha' ) \\  & =& \sum_{r,s\in \mathbb R}(r+s)\left(\deg(\mathrm{gr}_g^rD,\alpha)\cdot \dim (\mathrm{gr}_{g' }^sD') + \deg(\mathrm{gr}_{g' }^sD',\alpha')\cdot \dim(\mathrm{gr}_{g }^r D) \right) \\ & =& \langle \alpha,g\rangle \dim D'+\langle \alpha',g' \rangle \dim D+\deg(D,\alpha)\deg(g')+\deg(D',\alpha' )\deg(g), 
\end{eqnarray*}
yielding 
\begin{eqnarray*}
\langle \beta, f\rangle &=& \langle \alpha,g\rangle+\langle \alpha',g' \rangle+\langle \alpha\otimes \alpha',g\otimes g' \rangle \\ & =& \langle \alpha,g\rangle (\dim D'+1)+\langle \alpha',g' \rangle (\dim D+1)+\deg(D,\alpha)\deg(g' )+\deg(D',\alpha' )\deg(g). 
\end{eqnarray*}
So 
\begin{eqnarray*}
|f|^2-2\langle \beta,f\rangle & =& (\dim D'+1)(|g|^2-2\langle \alpha,g\rangle) +(\dim D+1)(|g' |^2-2\langle \alpha' ,g'\rangle) \\ & & +2\deg(g)\deg(g')-2\deg(D,\alpha)\deg(g')-2\deg(D',\alpha' )\deg(g).    
\end{eqnarray*}
Applying this equality to $f=\mathcal F$, and noticing that 
\[
\deg(\mathcal F_{\rm HN}(D,\alpha))=\deg(D,\alpha) \quad \textrm{and}\quad \deg(\mathcal F_{\rm HN}(D' ,\alpha' ))=\deg(D',\alpha'),
\]
we find 
\begin{eqnarray*}
|\mathcal F|^2-2\langle \beta,\mathcal F\rangle & =& (\dim D'+1)(|\mathcal F_{\rm HN}(D,\alpha)|^2-2\langle \alpha,\mathcal F_{\rm HN}(D,\alpha)\rangle)\\ & &  +(\dim D+1)(|\mathcal F_{\rm HN}(D',\alpha' )|^2-2\langle \alpha' ,\mathcal F_{\rm HN}(D',\alpha' )\rangle) \\ & & -2\deg(D,\alpha)\deg(D',\alpha' ).    
\end{eqnarray*}
Using Corollary \ref{cor:inequality-for-HN}, we deduce 
\[
\begin{array}{rl}
& |\mathcal F|^2-2\langle \beta,\mathcal F\rangle -(|f|^2-2\langle \beta,f\rangle ) \\   \leq & -\left( \frac{\dim D '+1}{\dim D}(\deg(g)-\deg(D,\alpha))^2+\frac{\dim D +1}{\dim D' }(\deg(g')-\deg(D',\alpha'))^2\right) \\   & -2(\deg(g)-\deg(D,\alpha))(\deg(g')-\deg(D',\alpha' )) \\   \leq & -\left( \frac{\dim D '}{\dim D}(\deg(g)-\deg(D,\alpha))^2+\frac{\dim D }{\dim D' }(\deg(g')-\deg(D',\alpha'))^2\right) \\   & -2(\deg(g)-\deg(D,\alpha))(\deg(g')-\deg(D',\alpha' )) \\ \leq & 0
\end{array}
\]
Here for the last inequality, we use the classical inequality 
\[
a^2+b^2\geq 2ab, \quad \forall a,b\in \mathbb R. 
\]
As a result, we see that the inequality \eqref{eq:inequality-for-HN} holds once $f\in \mathbf{F}(D)\times \mathbf{F}(D')\subset \mathbf{F}(E)$. 

For general $f\in \mathbf{F}(E)$, consider its convex projection $p(f)$ to $\mathbf{F}(D)\times \mathbf{F}(D')$. According to Lemma \ref{lem:property-for-convex-projection}, we have 
\[
|p(f)|^2\leq |f|^2.
\]
On the other hand, viewing $f$ as an $\mathbb R$-filtration of the underlying $\breve{F}$-vector space $W$ of $E$, we have 
\[
\deg(\mathrm{gr}_f^ aE,\beta)=\deg(\mathrm{gr}_f^aW, \beta)-\dim(\mathrm{gr}_f^aE),
\]
where the first term on the right hand side is the degree of the normed vector space $\mathrm{gr}_f^a W$, equipped with the norm induced from $\beta$, and the second term $\dim(\mathrm{gr}_f^aE)$ denotes the dimension of the isocrystsal $\mathrm{gr}_f^aW$. Hence 
\begin{eqnarray*}
\langle \beta,f\rangle& = & \sum_{a\in \mathbb R}a\cdot\deg(\mathrm{gr}_f^{a}E,\beta) \\& = & \sum_{a\in \mathbb R}a\cdot\deg(\mathrm{gr}_f^{a}W,\beta)-\sum_a a\cdot \dim(\mathrm{gr}_f^ aE).
\end{eqnarray*}
Now the first sum
\[
\langle\beta,f\rangle ' :=\sum_a a\cdot \deg(\mathrm{gr}_f^aW,\beta)
\]
has been already considered by Cornut in \cite[\S~5.2]{Cor}. As $p(f)$ is also the convex projection of $f\in \mathbf{F}(E)\subset \mathbf F(W)$ to $\mathbf{F}(V)\times \mathbf{F}(V')\subset \mathbf F(W)$ according to Lemma \ref{lem:two-convex-projections}, by \cite[Proposition 5.6 and \S~5.2.13]{Cor}, we have 
\begin{equation}\label{eq:inequality-of-Cornut}
\langle \beta,f\rangle ' \leq \langle \beta,p(f)\rangle '. 
\end{equation}
Furthermore, if we write $\mathcal F_{\rm N}\in \mathbf{F}(E)$ the Newton filtration on the isocrystal $E$ and $\mathcal F_{\rm N}^*$ its opposite (defined from the slope decomposition), it follows that
\[
\sum_a a\cdot \dim(\mathrm{gr}_f^ aE)=\langle\mathcal F_{\rm N},f\rangle=-\langle \mathcal F_{\rm N}^*, f\rangle. 
\]
Clearly, $\mathcal F_{\rm N},\mathcal F_{\rm N}^*\in \mathbf{F}(D)\times \mathbf{F}(D' )\subset \mathbf F(E)$, and both $ t\cdot \mathcal F_N+p(f)$ and $t\cdot \mathcal F_{\rm N}^*+p(f)$ are contained in $\mathbf{F}(D)\times \mathbf{F}(D')$ for any $t\in \mathbb R_{>0}$ sufficiently closed to $0$. By Lemma \ref{lem:property-for-convex-projection} we have 
\[
\langle \mathcal F_{\rm N},f\rangle \leq \langle \mathcal F_{\rm N},p(f)\rangle \quad \textrm{and}\quad \langle \mathcal F_{\rm N}^*,f\rangle \leq \langle \mathcal F_{\rm N}^*,p(f)\rangle,
\]
and thus an equality 
\[
\langle \mathcal F_{\rm N},f\rangle =\langle \mathcal F_{\rm N},p(f)\rangle.
\]
Combining with the inequality  \eqref{eq:inequality-of-Cornut}, we find
\[
\langle \beta,f\rangle =\langle \beta,f\rangle'-\langle \mathcal F_{\rm N},f\rangle \leq \langle \beta, p(f)\rangle'-\langle \mathcal F_{\rm N},p(f)\rangle=\langle \beta, p(f)\rangle,
\]
and thus 
\begin{eqnarray*}
|\mathcal F|^2-2\langle \beta,\mathcal F\rangle &\leq & |p(f)|^2-2\langle \beta,p(f)\rangle \\ & \leq & |f|^2-2\langle \beta,f\rangle
\end{eqnarray*}
for any $f\in \mathbf{F}(E)$, where the first inequality follows from the previous step. Therefore we find 
\[
\mathcal F_{\rm HN}(D,\alpha)\oplus \mathcal F_{\rm HN}(D',\alpha')\oplus \left(\mathcal F_{\rm HN}(D,\alpha)\otimes \mathcal F_{\rm HN}(D',\alpha')\right)=\mathcal F_{\rm HN}(E,\beta),
\]
while 
\[
\mathcal{F}_{\rm HN}(E,\beta)=\mathcal F_{\rm HN}(D,\alpha)\oplus \mathcal F_{\rm HN}(D',\alpha')\oplus \mathcal F_{\rm HN}(D\otimes D',\alpha\otimes \alpha')
\]
by Proposition \ref{prop:HN-filtration-and-direct-sim}. 
In other words, 
\[
\mathcal F_{\rm HN}(D,\alpha)\otimes \mathcal F_{\rm HN}(D' ,\alpha' )=\mathcal F_{\rm HN}(D\otimes D',\alpha\otimes \alpha' ),
\]
and the functor \eqref{eq:F_HN} is compatible with tensor products. 
\end{proof}

\subsection{A variant}\label{Sec_variant} Let 
\[
\mathbf{BunIsoc}_{\breve F|F}^{K}
\]
be the quasi-abelian $F$-linear $\otimes$-category whose objects are pairs $(D,\Xi)$,
where $D$ is an isocrystal over $\breve F|F$ and $\Xi\subset D\otimes_{\breve F}K$ is a lattice. A morphism \[
f:(D_1,\Xi_1)\longrightarrow (D_2,\Xi_2)
\]
in $\mathbf{BunIsoc}_{\breve F|F}^{K}$ consists of a morphism $f:D_1\ra D_2$ of isocrystals with 
\[
(f\otimes 1)(\Xi_1)\subseteq \Xi_2.
\]
As in \cite[5.2.3]{Cor}, there is a natural $F$-linear exact $\otimes$-functor
\begin{equation}\label{eq:gauge}
\mathbf{BunIsoc}_{\breve F|F}^K\longrightarrow \mathbf{NIsoc}_{\breve F|F}^{K}, \quad (D,\Xi)\mapsto (D,\alpha_{\Xi}).
\end{equation}
Here $\alpha_{\Xi}$ denotes the \emph{gauge norm} of the lattice $\Xi\subset D\otimes_{\breve F}K$: for $v\in D\otimes_{\breve F}K$
\[
\alpha_{\Xi}(v):=\inf\{|\lambda|: v\in \lambda \Xi\}.
\]
For example, for the lattice $\Xi_0:=D\otimes_{\breve F}\cO_K$, the gauge norm $\alpha_{\Xi_0}$ is the norm $\mathbf{o}$ used in the definition of the degree of a normed isocrystal (Definition \ref{def:degree-of-normed-isocrystals}).
\begin{remark}\label{rem:HN-for-lattice} The functor \eqref{eq:gauge} above identifies  $\mathbf{BunIsoc}_{\breve F|F}^K$ with a full subcategory of $\mathbf{NIsoc}_{\breve F|F}^K$, made of those $(D,\alpha)$ such that $\alpha(D\otimes_{\breve F}K)\subset |K|$, which is stable under strict subobjects and quotients.
\end{remark}

For $(D,\Xi)$ an object in $\mathbf{BunIsoc}_{\breve F|F}^{K}$, define
\[
\mathrm{rank}(D,\Xi):=\dim_{\breve F} D, \quad \textrm{and}\quad \deg(D,\Xi):=\deg(D,\alpha_{\Xi}).
\]

\begin{remark}\label{rem:lattice to filtration} Let $(D,\Xi)$ be an object in $\mathbf{BunIsoc}_{\breve F|F}^{K}$. Let $\mathcal F_{\Xi}$ be the \emph{residue filtration} on $D\otimes_{\breve F}C$ defined by the lattice $\Xi$, with $C$ the residue field of $K$: for $i\in \mathbb Z$,
\[
\mathcal F_{\Xi}^{i}(D\otimes C)=\frac{(\pi^{i}\Xi)\cap \Xi_0+\pi\Xi_0}{\pi\Xi_0}\hookrightarrow D\otimes_{\breve F}C.
\]
Here $\Xi_0=D\otimes_{\breve F}\cO_K$. In particular, we obtain a filtered isocrystal $(D,\mathcal F_{\Xi})$, and this construction yields a functor 
\[
\mathbf{BunIsoc}_{\breve F|F}^K\longrightarrow \mathbf{FilIsoc}_{\breve F|F}^C, \quad (D,\Xi)\longmapsto (D,\mathcal F_{\Xi}).
\]
By the adapted basis theorem, there exists a basis $\underline{e}=(e_1,\ldots, e_r)$ of the free $\cO_K$-module $\Xi_0=D\otimes_{\breve F}\cO_K$, and elements $\lambda_1,\ldots,\lambda_r\in \mathbb Z$ such that $(\pi^{\lambda_1} e_1,\ldots, \pi^{\lambda_r}e_r)$ forms an $\cO_K$-basis of $\Xi$. Then $\underline{e}$ is an orthogonal basis for $\alpha_{\Xi}$ and for $\mathbf{o}=\alpha_{\Xi_0}$ at the same time. Moreover,  
\[
\alpha_{\Xi}(e_i)=|\pi|^{-\lambda_i},\quad \textrm{and}\quad  
\alpha_{\Xi}(e_i)=1, \quad i\in \mathbb Z.
\]
So
\[
\deg(D,\Xi)=-\sum_{i}\lambda_i-\dim(D)=\deg(\mathcal F_{\Xi})-\dim(D)=\deg(D,\mathcal F_{\Xi}).
\]
In other words, $\deg(D,\Xi)$ is the same as the degree of the filtered isocrystal $(D,\mathcal F_{\Xi})$. However, we caution the readers that, for $D'\subset D$ a subisocrystal, set $\mathcal F_{\Xi}|_{D_C'}$ the filtration on $D_C'=D' \otimes_{\breve F}C$ induced from $\mathcal F_{\Xi}$, then we have in general 
\[
\deg(D',\Xi\cap (D'\otimes_{\breve F}K))\neq \deg(D',\mathcal F_{\Xi}|_{D_C'}).
\]
See \cite[Example 2.1]{Sh} for an explicit example of Viehmann. As a result, the Harder-Narasimhan filtration on $D$ defined by the lattice $\Xi$ as we consider here differs from the one defined by the residue filtration $\mathcal F_{\Xi}$. We will come back to this point later in \S~\ref{sec:compactibility-with-flag}.
\end{remark}

Theorem \ref{thm_comp tensor} combined with Remark \ref{rem:HN-for-lattice} give the following result.

\begin{theorem}\label{thm_compatible tensor}
The functions $\mathrm{rank}$ and $\deg$ above on $\mathbf{BunIsoc}_{\breve{F}/F}^{K}$ induce a Harder-Narasimhan formalism on $\mathbf{BunIsoc}_{\breve F|F}^K$, whose Harder-Narasimhan filtration
\begin{eqnarray}\label{eqn_F HN}
\mathcal F_{\rm HN}: \mathbf{BunIsoc}_{\breve F|F}^{K}\longrightarrow \mathbf F(\mathrm{Isoc}_{\breve F|F})\end{eqnarray}
is induced from the Harder-Narasimhan filtration $\mathcal F_{\rm HN}$ on $\mathbf{NIsoc}_{\breve F|F}^K$ in \eqref{eq:F_HN}. Moreover, $\mathcal F_{\rm HN}$  is compatible with tensor products, or equivalently, the tensor product of two semi-stable objects in $\mathbf{BunIsoc}_{\breve F|F}^K$ is still semi-stable.   
\end{theorem}

\begin{remark}
The similar result of Theorem \ref{thm_compatible tensor} for filtered isocrystals was proved by Faltings in \cite{Fa} and by Totaro in \cite{To1}. Both proofs work by reducing the problem of $\sigma$-linear algebra to a different problem of pure linear algebra. Motivated by the idea of Ramanan and Ramanathan from geometric invariant theory, Totaro gave in \cite{To2} an different elementary proof which avoids the reduction from filtered isocrystals to filtered
vector spaces. The method of Cornut in \cite{Cor} can be viewed as an axiomatized version
of Totaro's overall strategy in \cite{To2} in which the GIT tools are replaced by tools from convex
metric geometry. 
\end{remark}

In the next section, we will apply Theorem \ref{thm_compatible tensor} in the case where $K=\BdR$ to define Harder-Narasimhan stratification for a general reductive group. Here $\BdR$ is the de Rham period ring associated with an algebraically closed perfectoid field of characteristic $0$.


\section{Harder-Narasimhan stratification on $\BpdR$-affine Grassmannian}
The goal of this section is to define the Harder-Narasimhan stratification on the $\BpdR$-affine Grassmannian of a reductive group $G$. For $G=\mathrm{GL}_n$, it would be enough to use the Harder-Narasimhan formalism, i.e., the existence of Harder-Narasimhan filtration for isocrystals with lattices. In order to pass to more general reductive groups, we use Tannakian duality, and hence the compatibility of Harder-Narasimhan filtration with tensor products (Theorem \ref{thm_compatible tensor}) becomes crucial. Note that the special case of the Harder-Narasimhan stratification for $b=1$ with general reductive groups has already been worked out by Nguyen-Viehmann (\cite{NV}) and by Shen (\cite{Sh}).

\subsection{HN-vectors}
Let $G$ be a reductive group over $F$. Let $S=\Spa(C,C^+)\ra \mathrm{Spd}(\breve F)$ be a geometric point of $\mathrm{Spd}(\breve F)$, given by an untilt $S^{\sharp}=\Spa(C^{\sharp},C^{\sharp+})$ over $\breve F$. Write 
\[
\BdR=\BdR(C^ {\sharp})\quad \textrm{and}\quad  \BpdR=\BpdR(C^{\sharp})
\]
the corresponding de Rham period rings (cf. Section \ref{sec_FF curve}). Let $b\in G(\breve F)$. Consider the following faithful exact $\otimes$-functor
\begin{equation}\label{eq:G-bundle-Eb}
\omega_{G,b}: \mathbf{Rep}(G)\longrightarrow \mathbf{Isoc}_{\breve F|F}, \quad (V,\rho)\longmapsto (V\otimes_F \breve F,\rho(b)\circ (1\otimes \sigma)).
\end{equation}
Let $x\in \mathrm{Gr}_G(S)=G(\BdR)/G(\BpdR)$. For an $F$-representation $\rho=(V,\rho)$ of $G$, it induces an element
\[
\rho(x)\in \mathrm{Gr}_{\mathrm{GL}_V}(S)=\mathrm{GL}(V\otimes_F \BdR)/\mathrm{GL}(V\otimes_{F}B^+_{\rm dR}),
\]
or equivalently, a $\BpdR$-lattice $\Xi_{x,\rho}$ in  $V\otimes_{F}\BdR$: 
\[
\Xi_{x,\rho}=\rho(x)(V\otimes_F \BpdR)\subseteq V\otimes_F \BdR.
\]
The pair $(b, x)$ then gives rise to a $\otimes$-functor
\begin{equation}\label{eq:omega-G-b-x}
\begin{split} \omega_{G,b,x}:\mathbf{Rep}(G)&\longrightarrow\mathbf{BunIsoc}_{\breve F|F}^{\BdR}\\
 (V, \rho)&\longmapsto (\omega_{G, b}(V, \rho), \Xi_{x, \rho}).\end{split}
\end{equation}
Composing it with the functor $\mathcal{F}_{\rm HN}$ in (\ref{eqn_F HN}), we obtain
\[
\mathcal F_{\rm HN}(b,x):\mathbf{Rep}(G)\longrightarrow \mathbf{BunIsoc}_{\breve F|F}^{\BdR}\stackrel{\mathcal{F}_{\rm HN}}{\longrightarrow}\mathbf F(\mathbf{Isoc}_{\breve F|F}),
\]
which is exact and compatible with tensor products by Corollary \ref{thm_compatible tensor}. Since $G$ is reductive, there exists some rational cocharacter
\begin{equation}\label{eq:cocharacter-v}
v: \mathbb D_{\breve F}\longrightarrow G_{\breve F}
\end{equation}
that splits $\mathcal F_{\rm HN}(b,x)$, i.e., the composed functor
\[
\mathbf{Rep}(G)\stackrel{\mathcal F_{\rm HN}(b,x)}{\longrightarrow}\mathbf{F}(\mathbf{Isoc}_{\breve F|F})\longrightarrow \mathbf{F}( \mathbf{Vect}_{\breve F})
\]
of $\mathcal F_{\rm HN}(b, x)$ with the forgetful functor $\mathbf{F}(\mathbf{Isoc}_{\breve F|F})\rightarrow \mathbf{F}( \mathbf{Vect}_{\breve F})$ admits a factorization
\[
\xymatrix{\mathbf{Rep}(G)\ar[r] \ar[d]&  \mathbf{F}(\mathbf{Vect}_{\breve F}) \\ \mathbf{Rep}(\mathbb D_{\breve F})\ar@{=}[r] & \mathbb Q\!-\!\mathbf{Vect}_{\breve F}\ar[u] }.
\]
Here the left-vertical functor is induced by the rational cocharacter $v$, while the right-vertical functor takes a $\mathbb Q$-graded $\breve F$-vector space $V=\oplus_{\alpha\in \mathbb Q}V_{\alpha}$ to the filtered $\breve F$-space $(V,\Fil^{\bullet}V)$ with
\[
\Fil^{\alpha}V=\bigoplus_{\beta\geq \alpha}V_{\beta}, \quad \forall ~\alpha\in \mathbb R.
\]
\begin{remark}\label{rem:parabolic-attached-to-HN}
Let $P\subset G_{\breve F}$ be the parabolic subgroup of elements of $G_{\breve F}$ that preserves the Harder-Narasimhan filtration. In particular
\[
P(\breve F)=\{g\in G(\breve F)~|~  \rho(g) \textrm{ preserves }\mathcal F_{\rm HN}(b,x)(V,\rho), \forall (V,\rho)\in \mathbf{Rep}(G)\}.
\]
Then $P$ is also the parabolic subgroup attached to the rational cocharacter $v$.
\end{remark}

Consider as in \eqref{eq:cocharacter-v} a rational cocharacter $v$ that splits the Harder-Narasimhan filtration $\mathcal F_{\rm HN}(b,x)$. The class of $v$ in $X_{*}(G)_{\mathbb Q}/G(\bar F)$ does not depend on the particular choice of the splitting $v$. Moreover, since the Tannakian category $\mathbf{Isoc}_{\breve F|F}$ has a fiber functor over $F^{\rm un}$, the Harder-Narasimhan filtration has also a splitting over $F^{\rm un}$, the maximal unramified extension (inside $\bar F$) of $F$. Let us denote resulting class by
\[
v_{b,x}\in X_*(G)_{\mathbb Q}/G(\bar F).
\]

\begin{lemma}\label{lem:vbx-is-Galois-invariant} We have $
v_{b,x}\in \mathcal N(G)=\left(X_*(G)_{\mathbb Q}/G(\bar F)\right)^{\Gamma}$.
\end{lemma}
In the following, we call the element $v_{b,x}\in \mathcal N(G)$ the \emph{HN-vector} of the pair $(b,x)$. When $G$ is quasi-split, we also consider $v_{b,x}$ as an element in $X_*(A)_{\Q}^+$ via the identification $\mathcal{N}(G)\simeq X_*(A)^+_\Q$.

\begin{proof}[Proof of Lemma \ref{lem:vbx-is-Galois-invariant}]
Suppose that the rational cocharacter $v$ in \eqref{eq:cocharacter-v} is defined over $F^{\rm un}$. So it suffices to check that the class $[v]\in X_{*}(G)_{\mathbb Q}/G(\bar F)$ is invariant under the action of the Frobenius $\sigma\in \mathrm{Gal}(F^{\rm un}/F)$, i.e., $[v^{\sigma}]=[v]$ with $v^{\sigma}$ be base change of $v:\mathbb D_{F^{\rm un}}\ra G_{F^{\rm un}}$ induced by the Frobenius $\sigma$.
This is a consequence of the fact that the Harder-Narasimhan filtration is a filtration by \emph{subisocrystals}. More precisely, recall that $v$ is a splitting of the composed functor
\[
\mathcal F_{\rm HN}(b,x)':\mathbf{Rep}(G)\stackrel{\mathcal F_{\rm HN}(b,x)}{\longrightarrow}\mathbf{F}(\mathbf{Isoc}_{\breve F|F})\stackrel{\textrm{forgetful}}{\longrightarrow}\mathbf{F}({\mathbf{Vect}_{\breve F}}).
\]
By functoriality, its twist $v^{\sigma}$ by Frobenius splits
\[
\mathcal F_{\rm HN}(b,x)'':\mathbf{Rep}(G)\stackrel{\mathcal F_{\rm HN}(b,x)}{\longrightarrow}\mathbf{F}(\mathbf{Vect}_{\breve F})\stackrel{-\otimes_{\breve F,\sigma}\breve F}{\longrightarrow}\mathbf{F}({\mathbf{Vect}_{\breve F}}).
\]
On the other hand, for $(V,\rho)$ a representation of $G$ and for $W\subset (V_{\breve F},\rho(b)\sigma)$ a subisocrystal, we have $(\rho(b)\sigma)(W)=W$. In particular, the isomorphism
\[
\rho(b)^{-1}:V_{\breve F}\longrightarrow V_{\breve F}
\]
maps $W$ onto $\sigma(W)$. Moreover, under the natural identification $V_{\breve F}\otimes_{\breve F,\sigma}\breve F\simeq V\otimes_{F}\breve F$, $W\otimes_{\breve F,\sigma}\breve F$ corresponds to $\sigma(W)\subset V_{\breve F}$. As a result, the bijection $\rho(b)^{-1}$ above sends the Harder-Narasimhan filtration of $(\omega_{G,b}(V,\rho),\Xi_{x,\rho})$ to its twist by Frobenius. In this way, we obtain a natural transformation of functors
\[
b^{-1}:\mathcal F_{\rm HN}(b,x)'\stackrel{\sim}{\longrightarrow}\mathcal F_{\rm HN}(b,x)''.
\]
So $\mathrm{Int}(b^{-1})\circ v$ also splits $\mathcal F_{\rm HN}(b,x)''$. Hence, the cocharacter $v^{\sigma}$ is $G(\breve F)$-conjugate to $\mathrm{Int}(b^{-1})\circ v$, and the class $v_{b,x}$ of $v$ in $X_*(G)_{\mathbb Q}/G(\bar F)$ is Galois invariant.
\end{proof}

\subsection{Compatibility with inner-twists}
\label{subsection_comp inner twists}

Let $H$ be an inner-twist of $G$, given by a basic element $b_0\in G(\breve F)$. In particular, $H=J_{b_0}$ and for any $\breve F$-algebra $R$, we have
\[
H(R)=\{g\in G(\breve F\otimes_F R)| gb_0\sigma(g)^{-1}=b_0\}\stackrel{\sim}{\longrightarrow} G(R).
\]
Here the morphism is induced by the natural map $\breve F\otimes_F R\ra R$. Therefore, we have an identification of algebraic groups over $\breve F$:
\[
\iota: H_{\breve F}\stackrel{\sim}{\longrightarrow} G_{\breve F}.
\]
Using the element $b_0$ we can deduce a well-defined bijection
\[
\tilde{\iota}: B(H)\longrightarrow B(G), \quad [h]\longmapsto [\iota(h)b_0].
\]
Indeed identify $H_{\breve F}$ with $G_{\breve F}$ and thus $H(\breve F)$ with $G(\breve F)$ through the isomorphism $\iota$, the Frobenius automorphism $\sigma_H$ on $H(\breve F)$ is given by the composed map $\mathrm{Int}(b_0)\circ \sigma_G$ on $H(\breve F)=G(\breve F)$. Therefore for any $h,h_1\in H(\breve F)=G(\breve F)$ 
\[
h_1 (hb_0)\sigma_G(h_1)^{-1}=h_1 hb_0\sigma_G(h_1)^{-1}b_0^{-1}b_0=(h_1h\sigma_H(h_1)^{-1})b_0
\]
and the map $h\mapsto hb_0$ sends a $\sigma$-conjugacy class in $H(\breve F)$ to a $\sigma$-conjugacy class in $G(\breve F)$. Thus the map $\tilde {\iota}$ above is well-defined and it is clear that $\tilde{\iota}$ is bijective. 
On the other hand, the isomorphism $\iota$ induces also an isomorphism between the $\BpdR$-affine Grassmannians, still denoted by $\iota$ in the sequel
\[
\mathrm{Gr}_{H,\breve F}\longrightarrow \mathrm{Gr}_{G,\breve F}.
\]

\begin{proposition}\label{prop_comp inner twists}
Keep the notation above. Let $h\in H(\breve F)$ and $x\in \mathrm{Gr}_{H,\breve F}$ be a geometric point. Then $v_{\iota(h)b_0,\iota(x)}=v_{h,x}-\nu_{b_0}\in \mathcal N(G)=\mathcal N(H)$.
\end{proposition}

\begin{proof} The proof is inspired by \cite[3.4]{Kot}, where can be found a similar result for the Newton vectors. Take a faithful representation $(V,\rho)$ of $G$. Together with the Frobenius induced by $b_0$, the base change $V\otimes_F \breve F$ becomes an isocrystal over $F$, denoted by $\breve{V}$. One can also use the Frobenius structure induced by $\iota(h)b_0$ on $V\otimes_F \breve F$, and we denote the resulting isocrystal by $\breve{V}'$. Let $U=\mathrm{End}_{\sigma,\breve F}(\breve V)$, and more generally, for any $F$-algebra $R$,
\[
U\otimes_FR=\mathrm{End}_{\sigma,\breve F\otimes R}(\breve{V}\otimes_F R).
\]
Then $H(R)\subset U(R)$, so $U$ gives a faithful representation of $H$ by left-multiplication. In this way, $\breve U:=U\otimes_F \breve F$ becomes an isocrystal whose Frobenius structure is given by $h$. If moreover $R$ is an $\breve F$-algebra, we have naturally an identification induced by the natural morphism $\breve F\otimes_F R\ra R$ of $F$-algebras (recall that $b_0\in G(\breve F)$ is basic)
\[
U\otimes_F R\stackrel{\sim}{\longrightarrow}\mathrm{End}_{R}(V\otimes_F R).
\]
Under this identification, the action of $H_{\breve F}$ on $\breve U:=U\otimes_F\breve F$ is the same as the action of $G_{\breve F}$ on $\mathrm{End}_{F}(V)\otimes_F \breve F$ by left multiplication, and hence the isocrystal $\breve U$ can be identified with the following isocrystal
\[
\mathrm{End}_F(V)\otimes_F \breve F\simeq \mathrm{Hom}_{\breve F}(V\otimes_F \breve F,V\otimes_F \breve F),
\]
where the Frobenius is induced by the left multiplication given by $\rho(\iota(h))\in \mathrm{GL}(V\otimes_F\breve F)$. Consequently, we get a morphism of isocrystals
\begin{equation}\label{eq:comparison-inner-twist}
\breve{U}\otimes \breve V\longrightarrow \breve V'.
\end{equation}
Furthermore, once we put the norm structure on $\breve U\otimes_{\breve F}\BdR$ given by $x\in \mathrm{Gr}_{H,\breve F}$, or equivalently, the norm structure on
\[
\mathrm{End}_{F}(V)\otimes_{F}\BdR =\mathrm{Hom}_{\BdR}(V\otimes_F \BdR, V\otimes_F \BdR)
\]
by the guage norm induced by the lattice
\[
\mathrm{Hom}_{\BpdR}(V\otimes_F \BpdR,\Xi_{\iota(x),\rho})\subset \mathrm{End}_{F}(V)\otimes_{F}\BdR,
\]
and the norm on $\breve{V}\otimes_{\breve F}\BdR$ (resp. on $\breve{V}'\otimes_{\breve F}\BdR=V\otimes_F \BdR$) by the trivial $\BpdR$-lattice $\breve{V}\otimes \BpdR$ (resp. $\Xi_{\iota(x),\rho}$), the map \eqref{eq:comparison-inner-twist} above is even a surjective map between the normed isocrystals. Now take $v\in X_*(H)_{\mathbb Q}$ (resp. $v_0\in X_*(G)_{\mathbb Q}$) a rational cocharacter that splits the Harder-Narasimhan filtration $\mathcal F_{\rm HN}(h,x)$
(resp. corresponding to the slope decomposition of the isocrystal $\breve V$). In particular, $-v_0\in X_*(G)_{\mathbb Q}$ gives a splitting of the Harder-Narasimhan filtration $\mathcal{F}_{\rm HN}(b_0,x_0)$, where $x_0\in \mathrm{Gr}_{G,\breve F}$ is given by the identity element of $G$. As the Harder-Narasimhan filtration is compatible with tensor products (Theorem \ref{thm_compatible tensor}), using the surjective map \eqref{eq:comparison-inner-twist}, we deduce that the rational cocharacter
\[
\rho\circ \iota\circ v-\rho\circ v_0
\]
splits the Harder-Narasimhan filtratiton of the normed isocrystal $\breve V'$. As $V$ is a faithful representation of $G$, it follows that the rational cocharacter $\iota\circ v-v_0$ splits $\mathcal F_{\rm HN}(\iota(h)b_0,\iota(x))$. Therefore, we obtain the required equality $v_{h,x}-\nu_{b_0}=v_{\iota(h)b_0,\iota(x)}\in \mathcal N(H)=\mathcal N(G)$.
\end{proof}

\subsection{HN-vectors in the quasi-split case}\label{subsection_HN-vector} We keep the notations in the previous subsection. In particular, $S=\Spa(C,C^+)\ra \mathrm{Spd}(\breve F)$ is a geometric point of $\mathrm{Spd}(\breve F)$, given by an untilt $S^{\sharp}=\Spa(C^{\sharp},C^{\sharp+})$ over $\breve F$, with $
\BdR$ and $  \BpdR$ the corresponding de Rham period rings. Suppose that $G$ is quasi-split.  Abusing the notation, we identify $v_{b,x}\in \mathcal N(G)$ with its dominant representative in $X_*(A)_{\mathbb Q}^+$ in the following via the natural identification $X_*(A)_{\mathbb Q}^{+}\simeq \mathcal N(G)$. Let $P\subset G_{\breve F}$ be the parabolic subgroup corresponding to the Harder-Narasimhan filtration of $\omega_{G,b,x}$ (Remark \ref{rem:parabolic-attached-to-HN}). Since the Harder-Narasimhan filtration on a normed isocrystal is a filtration by subisocrystals, the parabolic subgroup $P$ descends to a subgroup of $J_b$ defined over $F$. 

\begin{lemma}\label{lem:std-parabolic-Q}  There exist a standard parabolic $Q\subset G$ over $F$ with standard Levi $M$, and a reduction $(b_Q,g)$ of $b$ to $Q$ satisfying the following properties.  
\begin{itemize}
\item $Q$ is the parabolic subgroup of $G$ defined by $v_{b,x}\in X_*(A)_{\mathbb Q}^ +$, and $gQ_{\breve F}g^{-1}=P$. In particular $v_{b,x}$ is dominant $Q$-regular.
\item Write $b_M\in M(\breve F)$ the image of $b_Q\in Q(\breve F)$, and recall the natural projection $\mathrm{pr}_M:\mathrm{Gr}_Q\ra \mathrm{Gr}_M$. Then the HN-vector of the pair $(b_M,\mathrm{pr}_M(g^{-1}x ))$ is $v_{b,x}$, thus the functor $\omega_{M,b_M,\mathrm{pr}_M(g^{-1}x)}$ (defined in \ref{eq:omega-G-b-x}) is semi-stable in the sense that the HN-vector of $(b_M,\mathrm{pr}_M(g^{-1}x))$ is central in $M$. 
\end{itemize}
\end{lemma}

\begin{proof} Let $v\in X_*(G)_{\mathbb Q}$ be a rational cocharacter defined over $\breve F$ that splits the Harder-Narasimhan filtration of $\omega_{G,b,x}$: so $[v]=[v_{b,x}]\in \mathcal N(G)$. Since $G$ is quasi-split, $v$ is $G(\breve F)$-conjugate to $v_{b,x}\in X_*(A)_{\mathbb Q}^+$. Let $g\in G(\breve F)$ with $\mathrm{Int}(g)\circ v_{b,x}=v$. Consequently,
\[
gQ_{\breve F}g^{-1}=P\subset G_{\breve F}
\]
with $Q\subset G$ the standard parabolic subgroup defined by  $v_{b,x}\in X_*(A)_{\mathbb Q}^+$. In particular, $v_{b,x}$ is dominant $Q$-regular.

On the other hand, since the Harder-Narasimhan filtration of $\omega_{G,b,x}$ is a filtration by sub-isocrystals, $b\sigma(P)b^{-1}=P$. It follows that
\[
b\sigma(g)Q_{\breve F}\sigma(g)^{-1}b^{-1}=b\sigma(g)\sigma(Q_{\breve F})\sigma(g)^{-1}b^{-1}=gQ_{\breve F}g^{-1},
\]
and thus
\[
g^{-1}b\sigma(h)Q_{\breve F}\sigma(g)^{-1}b^{-1}g=Q_{\breve F}.
\]
As a result, $b_{Q}:=g^{-1}b\sigma(g)\in G(\breve F)$ normalizes $Q_{\breve F}$ and hence is contained in $Q(\breve F)$. So $b=gb_{Q}\sigma(g)^{-1}$ has a reduction $(b_{Q},g)$ to $Q$. Furthermore, $g\in G(\breve F)$ gives an isomorphism from $\omega_{G,b_Q,g^{-1}x}$ to $\omega_{G,b,x}$, thus an isomorphism of functors
\[
\mathcal F_{\rm HN}(b_Q,g^{-1}x)\stackrel{\sim}{\longrightarrow}\mathcal F_{\rm HN}(b,x).
\]
In particular, the Harder-Narasimhan filtration of $\omega_{G,b_Q,g^{-1}x}$ has a splitting given by $v_{b,x}=\mathrm{Int}(g^{-1})\circ v$. By considering a faithful representation of $G$ (and thus a faithful representation of $Q$ through the embedding $Q\hookrightarrow G$), we check that $v_{b,x}\in X_*(A)_{\mathbb Q}$ splits the Harder-Narasimhan filtration of $\omega_{Q,b_Q,g^{-1}x}$,  hence also splits the Harder-Narasimhan filtration of the functor  $\omega_{M,b_M,\mathrm{pr}_M(g^{-1}x)}$. Since $v_{b,x}$ is central in $M$, it coincides with the HN-vector of the pair $(b_M,\mathrm{pr}_M(g^{-1}x))$ in $X_*(A)_{\mathbb Q}^{M,+}$, and the functor $\omega_{M,b_M,\mathrm{pr}_M(g^{-1}x)}$ is semi-stable.   
\end{proof}

To go further, it is convenient to use the notion of (modifications of) $G$-bundles on the Fargues-Fontaine curve $X=X_{(C,C^ +)}$. We first introduce the notion of slope vector of a $G$-bundle for any $G$. Later on, we will use the slope vector of a $M$-bundle where $M$ is a Levi subgroup of $G$. The faithful $\otimes$-functor $\omega_{G,b}$ of \eqref{eq:G-bundle-Eb} corresponds naturally to the $G$-bundle $\mathcal E_b$ on $X$ recalled in \S~\ref{subsubsection:G-bundles}. The map
\[\begin{split}X^*(G)&\longrightarrow \mathbb{Z}\\  \chi&\longmapsto \mathrm{deg}(\chi_*\mathcal{E}_{b})
\end{split}\]
is Galois invariant and can thus be viewed as an element $v(\E_b)\in X_*(Z_G)_{\Q}^{\Gamma}$, where $Z_G$ denotes the center of $G$ and $\chi_*(\mathcal E_b)$ is the pushout of the $G$-bundle $\mathcal E_b$ by the character $\chi$, thus is a line bundle on $X_{\bar F}$, the base change to $\bar F$ of the Fargues-Fontaine curve $X/F$.\footnote{Over $X_{\bar F}$ every vector bundle is a direct sum of line bundles of rational slopes (\cite[Th\'eor\`eme 8.5.1]{FF}).}  In other words, 
\[
\langle v(\E_b), \chi\rangle= \mathrm{deg}(\chi_*\mathcal{E}_{b}), \quad \forall \chi\in X^*(G).\]
In fact, $v(\E_b)=-\mathrm{Av}_G (\nu_{b})$, where $\mathrm{Av}_G$ denote the average over the action of the Weyl group $W_G$ of $(G, T)$. The element $v(\E_b)\in X_*(Z_G)_{\mathbb Q}^{\Gamma}$ is called the \emph{slope vector} of $\E_b$. 

Let $(b_{Q'},g')$ be a reduction of $b$ to a standard parabolic subgroup $Q'\subseteq G$, which, together with the $G$-bundle $\mathcal E_{b,x}$ on $X$, gives rise to a reduction $(\mathcal E_{b,x})_{b_{Q'}}$ to $Q$ and thus an element, (by abuse of notations) still called the \emph{slope vector} for the pair $(b_{Q'}, g')$,
\[
v_{(b_{Q'},g')}\in X_*(Z_{M'})_{\mathbb Q}^{\Gamma}\subset X_*(T)_{\mathbb Q}^{\Gamma}= X_*(A)_{\mathbb Q}
\]
as the slope vector of the $M'$-bundle $(\mathcal E_{b,x})_{b_{Q'}}\times^{Q'}M'$, where $Z_{M'}$ is the center of the standard Levi $M'$ corresponding to $Q'$, and
\[
(\mathcal E_{b,x})_{b_{Q'}}\simeq \mathcal E_{b_{Q'},g^{'-1}x}^{Q'}
\]
the reduction to the parabolic subgroup $Q'$ of $\mathcal E_{b,x}$ corresponding to $(b_{Q'},g')$.

The following result is an analogue of \cite[Theorem 4.1]{Sch}.

\begin{proposition}\label{prop_HN vector maximum of slope vector} Let $Q_1,Q_2$ be two standard parabolic subgroups of $G$, with $M_1$ and $M_2$ their corresponding standard Levi subgroups. For $i=1$ or $2$, let $(b_{Q_i},g_i)$ be a reduction of $b$ to $Q_i$, and write $b_{M_i}$ (resp. $x_{M_i}$) the image of $b_{Q_i}\in Q_i(\breve F)$ (resp. of $g_i^{-1}x\in \mathrm{Gr}_G(C,C^+)=\mathrm{Gr}_{Q_i}(C,C^+)$) in $M_i(\breve F)$ (resp. $\mathrm{Gr}_{M_i}(C,C^+)$) via the natural projection $Q_i\rightarrow M_i$. Assume that 
\begin{enumerate}
\item[(i)] the $M_1$-bundle $\mathcal {E}_{b_{M_1},x_{M_1}}^{M_1}$ is HN-semi-stable, in the sense that the HN-vector of the pair $(b_{M_1},x_{M_1})$ is central in $M_1$; and that 
\item[(ii)] the slope vector $v_{(b_{Q_1},g_1)}$ is dominant $Q_1$-regular. 
\end{enumerate}
Then the following assertions hold:  
\begin{enumerate}
\item $v_{(b_{Q_2},g_2)}\leq  v_{(b_{Q_1},g_1)}$; 
\item we have $Q_1=Q$, the standard parabolic subgroup in Lemma \ref{lem:std-parabolic-Q}, and $v_{(b_{Q_1},g_1)}=v_{b,x}$; 
\item if the inequality in (1) is an equality, then $Q_2\subset Q_1$, and $(b_{Q_2},g_2)$, viewed as a reduction of $b$ to $Q_1$, is equivalent to $(b_{Q_1},g_1)$.
\end{enumerate}
\end{proposition}

\begin{proof}
To simplify notations, write $v_i=v_{(b_{Q_i},g_i)}$. Write also $v_1'\in X_*(A)_{\mathbb Q}$ the HN-vector of the pair $(b_{M_1},x_{M_1})$. We first check that, as $v_1'$ is central in $M_1$ by the  assumption (i) above, $v_1'$ is equal to the slope vector $v_1$ of the $M_1$-bundle $\mathcal E_{b_{M_1},x_{M_1}}$, or equivalently, of the $Q_1$-bundle $(\mathcal E_{b,x})_{b_{Q_1}}=\mathcal{E}_{b_{Q_1},g^{-1}x}^{Q_1}$. More precisely,  for  every  $\lambda\in X^*(Q_1)^{\Gamma}=X^*(M_1)^{\Gamma}$, we want to show
\begin{equation}\label{eq:characterization-for-vbx}
\langle v_{1}',\lambda \rangle =\deg (\lambda _*(\mathcal E_{b,x})_{b_{Q_1}}).
\end{equation}
By the definition of the HN-vector, $\lambda\circ v_{1}'$ splits the Harder-Narasimhan filtration of the one-dimensional isocrystal $(\breve F,\lambda(b_{Q}))$ equipped with the norm (or equivalently, the $\BpdR$-lattice) defined by $\lambda(g^{-1}x)\in \mathrm{Gr}_{\mathbb G_m}(C,C^+)$. Therefore, $\lambda\circ v_{1}'=\langle v_{1}',\lambda\rangle\in \mathbb Q$ is the degree of this one-dimensional normed isocrystal, i.e.,
\[
\deg(\lambda _*(\mathcal E_{b,x})_{b_{Q_1}})=\deg(\mathcal E^{\mathbb G_m}_{\lambda(b_{Q_1}),\lambda(g_1^{-1}x)})=\langle v_{1}',\lambda \rangle,
\]
as desired by \eqref{eq:characterization-for-vbx}. So $v_{1}=v_{1}'\in X_*(A)_{\mathbb Q}$ since $v_1'$ is already central in $M_1$. 

To check the inequality $v_{2}\leq v_{1}$, observe firstly that the difference
\[
v_{1}-v_2\in \mathbb Q\Phi_0^{\vee},
\]where $\Phi_0^\vee$ is the set of coroots for $(G, A)$.
Indeed, for any $\lambda\in X^*(G)\subset X^*(T)_{\mathbb Q}$, we have
\[
\langle v_{1},\lambda\rangle=\langle v_2,\lambda\rangle=\deg \lambda_{*}\mathcal E_{b,x}.
\]
Hence $v_{1}-v_2\in X_*(T)_{\mathbb Q}$ lies in the orthogonal complementary of $X^*(Z_{G})_{\mathbb Q}\subset X^*(T)$ under the natural pairing
\[
X_*(T)_{\mathbb Q}\times X^*(T)_{\mathbb Q}\longrightarrow \mathbb Q.
\]
Consequently $v_{1}-v_2\in \mathbb Q\Phi^{\vee}$, the $\mathbb Q$-subspace of $X_*(T)_{\mathbb Q}$ generated by the set $\Phi^{\vee}$ of coroots for $(G, T)$. As $v_{1}-v_2$ is Galois-invariant, so $v_{1}-v_2\in \mathbb Q\Phi_0^{\vee}$, as wanted. 

To complete the proof of (1), we need to check that, for any dominant $\lambda\in X^*(A)_{\mathbb Q}^{+}$,
\[
\langle v_{1},\lambda\rangle \geq \langle v_2,\lambda\rangle.
\]
Let $\rho:G\ra \mathbf{GL}_V$ be a finite-dimensional representation of highest weight $\lambda$, and let
\[
V=\bigoplus_{\mu\in X^*(A)}V[\mu]
\]
be its weight decomposition. For any $q\in \mathbb Q$, consider the linear subspaces
\[
V_{1,q}=\bigoplus_{\langle v_1,\mu\rangle \geq q}V[\mu]\subset V. 
\]
Then $V_{1,q'}\subset V_{1,q}$ once $q\leq q'$, and we obtain a $\mathbb Q$-filtration $V_{1,\bullet}=(V_{1,q})_{q\in \mathbb Q}$ of $V$. 
Since $v_1$ is $Q_1$-dominant, for any character $\mu$ occurring in the weight decomposition above we have $\mu\leq \lambda$ and thus 
\[
\langle v_1,\mu\rangle \leq \langle v_1,\lambda \rangle.
\]
Consequently, the maximal rational number $q_1$ occurring among the jumps of the $\mathbb Q$-filtration $V_{1,\bullet}$ is equal to $\langle v_1,\lambda\rangle$. Moreover, as in \cite[Lemma 5.1]{Sch}, the filtration $V_{1,\bullet}$ is a filtration of $V$ by $Q_1$-subrepresentations, and the $Q_1$-action on each quotient $\mathrm{gr}_qV_{1,\bullet}$ descends to an action of the Levi $M_1$. 

We claim that the normed isocrystal 
\begin{equation}\label{eq:ss-normed-isocrystal}
\left(\mathrm{gr}_{q}V_{1,\bullet} \otimes_F \breve{F} ,b_{M_1}\circ (1\otimes\sigma),\alpha_{x_{M_1}}\right)
\end{equation}
is semi-stable of slope $q$. Indeed, as a representation of $A$, we have weight decomposition 
\[
\mathrm{gr}_q V_{1,\bullet}\simeq \bigoplus_{\langle v_1,\mu\rangle =q} V[\mu].  
\]
Since the $M_1$-bundle $\mathcal E_{b_{M_1},x_{M_1}}^{M_1}$ is HN-semi-stable, $v_1=v_1'$ is the HN-vector of $(b_{M_1},x_{M_1})$ and thus splits the Harder-Narasimhan filtration of the normed isocrystal \eqref{eq:ss-normed-isocrystal} above. As $\langle v_1,\mu\rangle =q$ for all character $\mu$ occurring in the weight decomposition of $\mathrm{gr}_qV_{1,\bullet}$, it follows that the normed isocrystal \eqref{eq:ss-normed-isocrystal} has only one jump $q$ in its Harder-Narasimhan filtration, hence is semi-stable of slope $q$ as claimed. As a result, the $\mathbf{Q}$-filtration $V_{1,\bullet}\otimes \breve F$ is precisely the Harder-Narasimhan filtration of the normed isocrystal $(V\otimes_F \breve F,b_Q\circ (1\otimes \sigma),\alpha_{g^{-1}x})$. In particular, the maximal destabilized sub-normed isocrystal of 
\[
(V\otimes_F \breve F,b_Q\circ (1\otimes \sigma),\alpha_{g^{-1}x})\stackrel{\sim}{\longrightarrow}(V\otimes_F \breve F, b\circ(1\otimes \sigma),\alpha_x)
\]
is of slope $q_1=\langle v_1,\lambda\rangle$. 
On the other hand, for the standard parabolic $Q_2$, we can define a subspace
\[
V_{2}'=\bigoplus_{\mu\in \lambda+\mathbb Z\Phi_{M_2,0}}V[\mu]
\]
of $V$. As  in \cite[Lemma 3.5]{Sch} we know that $V_2'$ is a $Q_2$-subrepresentation of $V$, so it gives a sub-normed isocrystal $(V_2'\otimes_F\breve F, b_{Q_2}\circ (1\otimes \sigma),\alpha_{g^{-1}x})$ of slope $\langle v_2,\lambda\rangle $ of  
\[
(V\otimes_F\breve F,b_{Q_2}\circ(1\otimes \sigma),\alpha_{g^{-1}x})\stackrel{\sim}{\longrightarrow}(V\otimes_F \breve  F,b\circ(1\otimes \sigma).\alpha_x).
\]
So $\langle v_{1},\lambda\rangle\geq \langle v_2,\lambda\rangle$. As this inequality holds for every dominant weight, we deduce 
\[
v_2=v_{(b_{Q_2},g_2)}\leq  v_{1}=v_{(b_{Q_1},g_1)},
\]
as asserted by (1).

(2) We use the notation of Lemma \ref{lem:std-parabolic-Q}. As $v_{b,x}\in X_*(A)_{\mathbb Q}^+$ is the HN-vector of the pair $(b_M,\mathrm{pr}_M(g^{-1}x))$ and is central in $M$, it coincides with the slope vector of $\mathcal E_{b_M, \mathrm{pr}_{M}(g^{-1}x)}^{M}$. Thus $v_{(b_Q,g)}=v_{b,x}$. Moreover, $v_{b,x}\in X_*(A)_{\mathbb Q}^+$ is dominant $Q$-regular. Consequently, by (1) we must have 
\[
v_{(b_{Q_1},g_1)}=v_{(b_Q,g)}=v_{b,x}
\]     
and thus also $Q_1=Q$.

(3) Finally suppose $v_1=v_2$. Since $v_1$ is dominant $Q_1$-regular, necessarily $Q_2\subset Q_1=Q$. It remains to check that the pair $(b_{Q_2},g_2)$, viewed as a reduction of $b$ to $Q_1=Q$, is equivalent to $(b_{Q_1,}g_1)$. Replacing $Q_2$ by $Q_1=Q$ and $b_{Q_2}$ by its image in $Q$ if needed, we assume moreover that $Q_1=Q_2=Q$. Write $b_i=b_{Q_i}$ for $i=1,2$. So we have two reduction $(b_i,g_i)$ of $b$ to $Q$ and hence $b_2=(g_2^{-1}g_1)b_1\sigma(g_2^{-1}g_1)^{-1}$. To conclude the result we need to check that 
\[
g_2^{-1}g_1\in Q(\breve F).
\]
For each $i=1,2$, the element $g_i\in G(\breve F)$ induces an isomorphism of functors 
\begin{equation}\label{eq:two-omega}
\omega_{G,b_i,g_i^{-1}x}\stackrel{\sim}{\longrightarrow }\omega_{G,b,x}.
\end{equation}
We claim that the parabolic subgroup $Q_{\breve F}\subset G_{\breve F}$ preserves the Harder-Narasimhan filtration on $\omega_{G,b_i,g_i^{-1}x}$. To see this, it is enough to check the corresponding assertion after evaluating $\omega_{G,b_i,g_i^{-1}x}$ at an algebraic representation $(V,\rho)$ of $G$. Clearly, we can assume moreover that $(V,\rho)$ is a highest weight representation with highest weight $\lambda\in X_*(A)_{\mathbb Q}^ +$. Using the fact that $v_1$ is dominant $Q$-regular, the construction in (1) provides a filtration $V_{\bullet}=(V_q)_{q\in \mathbb Q}$ of $V$, which is stable under the action of $Q$. Moreover, as observed above, $V_{\bullet,\breve F}=(V_q\otimes \breve F)_{q\in \mathbb Q}$ gives the Harder-Narasimhan filtration of 
\[
\omega_{G,b_1,g_1^{-1}x}(V,\rho)=(V\otimes \breve F,\rho(b_1)(1\otimes \sigma),\Xi_{\rho,g_1^{-1}x}). 
\]
Hence this verifies our claim when $i=1$. For $i=2$, the same filtration $V_{\bullet,\breve F}$ of $V\otimes \breve F$ yields a filtration by subisocrystals of $\omega_{G,b_2,g_2^{-1}x}(V,\rho)$ whose slope vector is the same as that of its Harder-Narasimhan filtration. Therefore, according to the lemma below, $V_{\bullet,\breve F}$ coincides necessarily with the Harder-Narasimhan filtration of $\omega_{G,b_2,g_2^{-1}x}(V,\rho)$, showing our claim when $i=2$. In particular, it follows that $Q_{\breve F}\subset g_i^{-1}Pg_i$ and thus $Q_{\breve F}= g_i^{-1}Pg_i$ since $Q_{\breve F}$ is a $G(\breve F)$-conjugate of $P$. As a result, we get 
\[
g_1Q_{\breve F}g_1^{-1}=P=g_2Q_{\breve F}g_2^{-1}
\]
and hence $g_2^{-1}g_1\in Q(\breve F)$ as desired. 
\end{proof}

\begin{lemma} Let $(D,\alpha)$ be a normed isocrystal. Let $
D_{\bullet}=(D_q)_{q\in \mathbb R}$ and $D_{\bullet}'=(D_q')_{q\in \mathbb R}$ be two $\mathbb R$-filtrations of $D$ by subisocrystals. Assume the following two conditions:
\begin{itemize}
\item $D_{\bullet}$ is the Harder-Narasimhan filtration of $(D,\alpha)$; 
\item $D_{\bullet}'$ has the same numerical data as $D_{\bullet}$, i.e., 
\[
\deg(\mathrm{gr}_{D_{\bullet}}^qD,\alpha)=\deg(\mathrm{gr}_{D_{\bullet}'}^qD,\alpha), \quad \textrm{and} \quad \dim (\mathrm{gr}_{D_{\bullet}}^qD)=\dim (\mathrm{gr}_{D_{\bullet}'}^qD).
\]
\end{itemize}
Then $D_{\bullet}=D_{\bullet}'$. 
\end{lemma}
\begin{proof}
Write $r$ (resp. $r'$) be the maximal real number occurring among the jumps of the $\mathbb R$-filtration $D_{\bullet}$ (resp. of $D_{\bullet}'$). Then $(D_r,\alpha)$ is the maximal destabilized sub-normed isocrystal of $(D,\alpha)$ and it is of slope $r$. Since $D_{\bullet}'$ has the same numerical data as $D_{\bullet}$, it follows that $r=r'$ and that $(D_r',\alpha)$ is a sub-normed isocrystal of $(D,\alpha)$ of slope $r$. Therefore, $(D_r',\alpha)$ must be semi-stable, and hence coincides with $(D_r,\alpha)$ as $\dim_{\breve F} D_r=\dim_{\breve F}D_r'$. Next replacing $D$ by $D/D_r=D/D_{r'}$ we check by induction that $D_{\bullet}=D_{\bullet}'$. 
\end{proof}

Combining Lemma \ref{lem:std-parabolic-Q} and Proposition \ref{prop_HN vector maximum of slope vector} we obtain 

\begin{corollary}[Canonical reduction for the HN-filtration]\label{prop_can reduction} Assume $G$ quasi-split. Let $b\in G$ and $x\in \mathrm{Gr}_G(C,C^+)$ a geometric point. Then 
\[
v_{b,x}=\max\left\{v_{(b_{Q'},g')} \bigg| {Q'\textrm{ is a standard parabolic of } G, \textrm{ and}\atop (b_{Q'},g') \textrm{ is a reduction of }b \textrm{ to }Q' } \right\}.
\]
Moreover, there exist a unique standard parabolic subgroup $Q$ with Levi component $M$ and a unique class of reduction $(b_Q,g)$ of $b$ to $Q$ satisfying the following properties:

\begin{enumerate}
    \item write $b_M\in M(\breve F)$ the image of $b_Q$ and recall the projection $\mathrm{pr}_M:\mathrm{Gr}_P\rightarrow  \mathrm{Gr}_M$, then the $M$-bundle $\mathcal E_{b_M, \mathrm{pr}_M(g^{-1}x)}$ is \emph{HN-semi-stable} in the sense that the HN-vector of $(b_M,x_M)$ is central in $M$; and
    \item the slope vector of the $M$-bundle in (1) is dominant $Q$-regular.
\end{enumerate}
Finally, the slope vector of the $M$-bundle in (1) is $v_{b,x}$.
\end{corollary}

\begin{corollary}\label{coro_HN vector}  Let $G$ be a reductive group over $F$, which is not necessarily quasi-split. Let $b\in G(\breve F)$ and $x\in \mathrm{Gr}_{G}$ a geometric point, such that $\mathcal E_{b,x}=\mathcal E_{b'}$ with $[b']\in B(G)$. Then there exists a unique element $[b'']\in B(G)$ such that $\nu_{b''}=-w_0v_{b,x}$ and $\kappa(b'')=\kappa(b')=\kappa(b)-\mu^\sharp$. Moreover $\nu_{b''}\leq \nu_{b'}$.

In the following, we shall denote by $\mathrm{HN}(b,x)$ this unique element $[b'']\in B(G)$.
\end{corollary}

\begin{proof}The uniqueness is clear as any element in $B(G)$ is uniquely determined by Newton invariant and Kottwitz invariant (\ref{eqn_RR}). It suffices to prove the existence of $b''$.

Assume first that $G$ is quasi-split. Let $Q$ be the standard parabolic subgroup of $G$ as given in Lemma \ref{lem:std-parabolic-Q}, with $M$ its corresponding standard Levi which is  the centralizer of $v_{b,x}$. Let $b_Q=(b_{Q},g)$ be a reduction of $b$ to $Q$ such that $v_{(b_Q,g)}=v_{b,x}$. Let $b_{M}\in M(\breve F)$ be the image of $b_{Q}$ via the natural projection. So
\[
(\mathcal E_{b,x})_{b_Q}\times^QM\simeq \mathcal E_{b_M,x'}
\]
with $x'\in \mathrm{Gr}_{M, \lambda}$ a geometric point for some $\lambda\in S_M(\mu)$. Its $M$-equivariant first Chern class $c_M^1(\mathcal E_{b_M,x'})\in \pi_1(M)_{\Gamma}$ gives a basic element $[\tilde{b}_M'']\in B(M)$ with
\[
\kappa_M(\tilde{b}_M'')=c_M^{1}(\mathcal E_{b_M,x'})=\lambda^{\sharp}-\kappa_M(b_M) \text{ in } \pi_1(M)_{\Gamma}.
\] Then $\nu_{\tilde{b}_{M}''}=v_{b,x}$. Let $[\tilde{b}'']\in B(G)$ be its image via the natural map $B(M)\rightarrow B(G)$. 
Recall that by \cite[Corollary 2.9]{CT}, the map 
\[\begin{split} X_*(A)_{\mathbb Q}^+&\simeq X_*(A)_{\mathbb Q}^+\\
v&\mapsto v^*:=-w_0v\end{split}
\]
induces a bijection between generalized Kottwitz sets
\[
B(G, \epsilon, \delta)\simeq B(G, -\epsilon, -w_0\delta).
\]
Let $[b'']\in B(G)$ be the image of $[\tilde{b}'']$ under this bijection. Then it is clear that 
\[\kappa(b'')=-\kappa(\tilde{b}'')=-\mu^\sharp+\kappa(b)=\kappa(b')\] and
\[
\nu_{b''}=-w_0\nu_{\tilde{b}''}=-w_0v_{b,x}.
\]

The inequality $\nu_{b''}\leq \nu_{b'}$, or equivalently, $v_{b,x}\leq \nu_{\mathcal E_{b,x}}$ (where the HN polygone $\nu_{\E_{b, x}}$ of $\E_{b, x}$ is defined in Section \ref{sec: HN reduction}), follows from the previous result, as $v_{\mathcal E_{b,x}}$ can be characterized similarly as the maximum of the slope vectors of all possible reductions of $\mathcal E_{b,x}$ to a standard parabolic subgroups of $G$ (Proposition \ref{prop_canonical reduction characterization}).

Now consider the general case where $G$ is not necessarily quasi-split. Let $H$ be the unique quasi-split inner form over $F$ of $G$, and write $H=J_{b_0}$ for some basic $b_0\in G(\breve F)$. In particular, we have a natural identification $\iota: H_{\breve F}=J_{b_0,\breve F}\stackrel{\sim}{\ra}G_{\breve F}$. Write $b_H=\iota^{-1}(bb_{0}^{-1})\in H(\breve F)$, $x_H=\iota^{-1}(x)\in \mathrm{Gr}_H$ and $\mathcal E_{b_H,x_H}=\mathcal E_{b_H'}$. Then under the natural identification $\mathcal N(H)=\mathcal N(G)$, we have $v_{b,x}=v_{b_H,x_H}-\nu_{b_0}$ by Proposition \ref{prop_comp inner twists}. After applying our corollary to the quasi-split group $H$ for the triple $(b_H, x_H, b'_H)$, there exists an unique $[b''_H]\in B(H)$ such that 
\[\nu_{b''_H}=-w_0v_{b_H, x_H} \text{ and } \kappa(b''_H)=\kappa(b'_H)=\kappa(b_H)-\mu^{\sharp}.\]Let $[b'']\in B(G)$ be the image of $[b''_H]\in B(H)$ via the identification $B(H)\simeq B(G)$. In fact $[b'']=[\iota(b_H'')b_0]\in B(G)$. So
\[
-w_0v_{b,x}=-w_0v_{b_H,x_H}+\nu_{b_0}=\nu_{b_{H}''}+\nu_{b_0}=\nu_{b''}
\]
and
\[
\kappa_G(b'')=\kappa_H(b_{H}'')+\kappa_G(b_0)=\kappa_H(b_H')+\kappa_G(b_0)=\kappa_G(b').
\]
\end{proof}

\subsection{Harder-Narasimhan stratification on $\BpdR$-affine Grassmannian} Let $b\in G(\breve F)$. Then we have the following simple observation. 

\begin{lemma} Let $\mathrm{Spa}(C_2,C_2^+)\ra \mathrm{Spa}(C_1,C_1^+)$ be a morphism in $\mathbf{Perf}/\mathrm{Spd}(F)$ with $C_1,C_2$ two algebraically closed perfectoid fields with $C_1^+\subset C_1$ and $C_2^+\subset C_2$ bounded valuation subrings. For $x_1\in \mathrm{Gr}_G(C_1,C^+_1)$ with image $x_2\in \mathrm{Gr}_G(C_2,C_2^+)$, we have 
\[
\mathrm{HN}(b,x_1)=\mathrm{HN}(b,x_2).
\]  
    
\end{lemma}

\begin{proof} By Tannakian duality, we reduce to the case $G=\mathrm{GL}_V$ with $V$ a finite-dimensional $F$-vector space. Write $(C_i^{\sharp},C_i^{\sharp +})$ be the corresponding untilt of $(C_i,C_i^+)$, $B_i^+:=\BpdR(C_i^{\sharp})$ and $B_i=\BdR(C_i^{\sharp})$ for $i=1,2$. Then the natural map $B_1^+\ra B_2^+$ is an unramified extension of DVRs whose residue field extension is just $C_1^{\sharp}\subset C_2^{\sharp}$. Let $\Xi_1\subset V\otimes_{F}B_1$ be a $B_1^+$-lattice and $\Xi_2=\Xi_1\otimes_{B_1^+}B_2^+$. Then the HN-filtration for the normed isocrystal $(V\otimes \breve F, b,\alpha_{\Xi_1})$ coincides with that for $(V\otimes \breve F,b,\alpha_{\Xi_2})$, from where our lemma follows.
\end{proof}

Consequently, we dispose of a map
\begin{equation}\label{eq:HN_b}
\mathrm{HN}_b:|\mathrm{Gr}_G|\longrightarrow B(G), \quad x\longmapsto \mathrm{HN}(b,x).
\end{equation}
Let $[b'']\in B(G)$. We write 
\[
|\mathrm{Gr}_{G}|^{\mathrm{HN}_b\geq [b'']} \quad (\textrm{resp. } |\mathrm{Gr}_{G}|^{\mathrm{HN}_b= [b'']})
\]
the subset of $|\mathrm{Gr}_G|$ consisting of points $x$ such that 
\[
\mathrm{HN}(b,x)\geq [b'']\quad  (\textrm{resp. } \mathrm{HN}(b,x)=[b'']).
\]
Write also 
\[
\mathrm{Gr}_{G,b}^{\mathrm{HN}=[b'']}:=\mathrm{Gr}_{G}\times_{|\mathrm{Gr}_{G}|} |\mathrm{Gr}_G|^{\mathrm{HN}_b\geq [b'']} \quad (\textrm{resp.}\ \mathrm{Gr}_{G,b}^{\mathrm{HN}\geq [b'']}:=\mathrm{Gr}_{G}\times_{|\mathrm{Gr}_{G}|} |\mathrm{Gr}_G|^{\mathrm{HN}_b= [b'']}),
\]
that is, the v-subsheaf of $\mathrm{Gr}_{G}$, such that a morphism $S\ra \mathrm{Gr}_G$ from a perfectoid space $S\in \mathbf{Perf}/\mathrm{Spd}(F)$ factors through $\mathrm{Gr}_{G,b}^{\mathrm{HN}=[b'']}$ (resp. $\mathrm{Gr}_{G,b}^{\mathrm{HN}\geq [b'']}$) if and only if $\mathrm{HN}(b,x)=[b'']$ (resp. (resp. $\mathrm{HN}(b,x)\geq [b'']$) for every geometric point $x$ of $S$. 
On the other hand, by the local constancy of the Kottwitz invariant (\cite[Theorem III.2.7]{FS}),  the following composed map is locally constant
\[
\kappa_b: |\mathrm{Gr}_G|\stackrel{BL_b}{\longrightarrow} |\mathrm{Bun}_G|\stackrel{\kappa}{\longrightarrow}\pi_1(G)_{\Gamma}.
\]
Here the first map is the Beauville-Laszlo morphism defined by using the element $b\in G(\breve F)$. The preimage of an element $\alpha\in \pi_1(G)_{\Gamma}$ is then the underlying topological space of an open and closed v-subsheaf, denoted by
\[
\mathrm{Gr}_G^{\kappa_b=\alpha},
\]
of $\mathrm{Gr}_G$, and we deduce a decomposition in open and closed subsheaves
\[
\mathrm{Gr}_G=\coprod_{\alpha\in \pi_1(G)_{\Gamma}}\mathrm{Gr}_G^{\kappa_b=\alpha}.
\]
For a geometric conjugacy class $\{\mu\}$ of cocharacters of $G$, write also 
\[
\mathrm{Gr}_{G,\leq \mu}^{\kappa_b=\alpha}=\mathrm{Gr}_{G,\bar F}^{\kappa_b=\alpha}\cap \mathrm{Gr}_{G,\leq \mu}.
\]

\begin{proposition}\label{prop:HN-strata-are-locally-closed} Let $[b'']\in B(G)$.
\begin{enumerate}
\item  Suppose moreover $G$ quasi-split. Let $S\ra \mathrm{Spd}(\bar F)$ be a geometric point of $\mathrm{Spd}(\bar F)$. We have
\[
\mathrm{Gr}_{G,b}^{\mathrm{HN}\geq [b'']}(S)=\left(\bigcup_{\substack{\textrm{reduction }(b_{P},g) \textrm{ of }b \textrm{ to}\\ \textrm{a standard parabolic }P}} g\cdot \left(\bigcup_{\substack{u\in X_*(M^{ab}) \textrm{ such that} \\ -\mathrm{Av}_{M}(\nu_{b_M})+u^{\diamond}\geq -w_0\nu_{b''} }} \mathrm{Gr}_P^{u}(S)\right)\right)\bigcap \mathrm{Gr}_{G}^{\kappa_b=\kappa(b'')}(S),
\] 
where \begin{itemize}
\item $\mathrm{Gr}_P^u$ is defined in section \ref{sec_semi-infinite orbits};
\item $[b_M]$ is the image of $[b_P]$ via the natural map $B(P)\rightarrow B(M)$;
\item in the inequality, $u^{\diamond}$ is the Galois average of $u$ and is considered to be an element in $X_*(A)_\Q$ via the natural inclusion $X_*(M^{ab})_{\mathbb Q}^{\Gamma}\subset X_*(A)_{\Q}$.
\end{itemize}
\item For general reductive group $G$ and for a geometric conjugacy class $\{\mu\}$ of cocharacters of $G$, 
the subfunctor 
\[
\mathrm{Gr}_{G,b,\bar F}^{\mathrm{HN}\geq [b'']}\cap \mathrm{Gr}_{G,\leq \mu}\quad \quad (\textrm{resp. }\mathrm{Gr}_{G,b,\bar F}^{\mathrm{HN}= [b'']}\cap \mathrm{Gr}_{G,\leq \mu}) 
\]
is closed (resp. locally closed) in $\mathrm{Gr}_{G,\leq \mu}$, thus a spatial diamond (resp. a locally  spatial diamond). In particular, the restriction to $|\mathrm{Gr}_{G,\leq \mu}|$ of the map  $\mathrm{HN}_b$ in \eqref{eq:HN_b} is upper semi-continuous.
\end{enumerate}
\end{proposition}

\begin{proof}
(1) Take $x\in \mathrm{Gr}_{G,b}^{\mathrm{HN}\geq [b'']}(S)$. Then $v_{b,x}\geq -w_0\nu_{b''}$ and $\kappa_b(x)=\kappa(b'')$ by Corollary \ref{coro_HN vector}. Let $M\subset G$ be the centralizer of $v_{b,x}$, with $P\subset G$ the corresponding standard parabolic subgroup. In particular, $b$ has a reduction $(b_P,g)$ to $P$ such that
\[
\mathcal E_{b,x}\simeq \mathcal E_{b_P,g^{-1}x},
\] with $g^{-1}x\in \mathrm{Gr}_P^u$ for some $u\in \pi_1(M^{ab})$.
Moreover, the HN-vector $v_{b,x}$ is the slope vector of the $M$-bundle $\mathcal E_{b_M,\mathrm{pr}_M(g^{-1}x)}$, which is nothing but $-\mathrm{Av}_M(\nu_{b_M})+u^{\diamond}$. It follows that \[
-\mathrm{Av}_M(\nu_{b_M})+u^{\diamond}\geq -w_0\nu_{b''}.
\]
Consequently, $x\in \mathrm{Gr}_{G,b}^{\mathrm{HN}\geq [b'']}(S)$ is contained in the right-hand side of the equality in (1).

Conversely, suppose that $b$ has a reduction $(b_P,g)$ and $x\in g\mathrm{Gr}_{P}^u(S)$ for some $u\in X_*(M^{ab})$ with $
-\mathrm{Av}_M(\nu_{b_P})+u^{\diamond}\geq -w_0\nu_{b''}$. The slope vector of the $M$-bundle $\mathcal E_{b_M,\mathrm{pr}(g^{-1}x)}$ is precisely 
\[
-\mathrm{Av}_M(\nu_{b_M})+u^{\diamond}.
\]
Therefore 
\[
v_{b,x}\geq -\mathrm{Av}_M(\nu_{b_P})+u^{\diamond}\geq -w_0\nu_{b''}.
\]
If moreover $\kappa_b(x)=\kappa(b'')$, we then obtain $\mathrm{HN}(b,x)\geq [b'']$, as desired.

(2) Clearly we may assume $G$ quasi-split and it is enough to show that $\mathrm{Gr}_{G,b}^{\mathrm{HN}\geq [b'']}\cap \mathrm{Gr}_{G,\leq \mu}$ is closed in $\mathrm{Gr}_{G,\leq \mu}$, i.e., for every totally disconnected perfectoid space $S$ mapping to $\mathrm{Gr}_{G,\leq \mu}$, the pullback 
\[
(\mathrm{Gr}_{G,b}^{\mathrm{HN}\geq [b'']}\cap \mathrm{Gr}_{G,\leq \mu}) \times_{\mathrm{Gr}_{G,\leq \mu}} S
\]
is representable by a closed immersion to $S$ (cf. \cite[Definition 10.7  (ii)]{Sch2}). Since 
\[
\mathrm{Gr}_{G,b}^{\mathrm{HN}\geq [b'']}\cap \mathrm{Gr}_{G,\leq \mu}=\mathrm{Gr}_{G,\leq \mu}\times_{|\mathrm{Gr}_{G,\leq \mu }|}|\mathrm{Gr}_{G,b}^{\mathrm{HN}\geq [b'']}\cap \mathrm{Gr}_{G,\leq \mu}|,
\]
according to \cite[Lemma 7.6 and Definition 5.6]{Sch2}, 
we only need to show that the subset 
\[
|\mathrm{Gr}_{G,b}^{\mathrm{HN}\geq [b'']}\cap \mathrm{Gr}_{G,\leq \mu}|\subset |\mathrm{Gr}_{G,\leq \mu}|
\]
is closed and generalizing. 

For a standard parabolic $P$ and $\lambda:=-w_0\nu_{b''}$, consider the following v-subsheaves 
\[
\mathscr{S}_{b_P,\lambda}:=\left(\bigcup_{\substack{u\in X_*(M^{ab}) \textrm{ such that} \\ -\mathrm{Av}_{M}(\nu_{b_M})+u^{\diamond}\geq  \lambda }}\mathrm{Gr}_P^{u}\right)\cap \mathrm{Gr}_{G,\leq \mu}^{\kappa_b=\kappa(b'')}
\]
and 
\[
\mathscr{S}:=\bigcup_{\substack{\textrm{reduction }(b_P,g) \textrm{ of }b \textrm{ to } \\ \textrm{a standard parabolic }P}} g\cdot \mathscr{S}_{b_P,\lambda}.
\]
Clearly we have inclusion 
\[
\mathscr S\subset \mathrm{Gr}_{G,b}^{\mathrm{HN}\geq [b'']}\cap \mathrm{Gr}_{G,\leq \mu}
\]
of subsheaves of $\mathrm{Gr}_{G,\leq \mu}$ and they have the same geometric points by (1). So  
\[
|\mathscr S|=|\mathrm{Gr}_{G,b}^{\mathrm{HN}\geq [b'']}\cap \mathrm{Gr}_{G,\leq \mu}|\subset |\mathrm{Gr}_{G,\leq \mu}|. 
\]
Moreover, for $g\in P(\breve F)$, $g\mathrm{Gr}_{P}^u=\mathrm{Gr}_{P}^{u}$, so the subfunctor $\mathscr{S}_{b_P,\lambda}$ is preserved by the left-multiplication by an element of $P(\breve F)$. 

We first claim that $|\mathscr{S}_{b_P,\lambda}|$ is closed in $|\mathrm{Gr}_{G,\leq \mu}|$. Indeed, for $u,u'\in X_*(M^{ab})$ with $u'\geq u$, we have also ${u'}^{\diamond}\geq u^{\diamond}$. It follows that  
\[
\mathscr{S}_{b_P,\lambda}=\left(\bigcup_{\substack{u\in X_*(M^{ab}) \textrm{ such that} \\ -\mathrm{Av}_{M}(\nu_{b_M})+u^{\diamond}\geq  \lambda }}\mathrm{Gr}_P^{u}\right)\cap \mathrm{Gr}_{G,\leq \mu}^{\kappa_b=\kappa(b'')}=\bigcup_{\substack{u\in X_*(M^{ab}) \textrm{ such that} \\ -\mathrm{Av}_{M}(\nu_{b_M})+u^{\diamond}\geq  \lambda }} \left(\left(\bigcup_{\substack{u'\in X_*(M^{ab}) \\ \textrm{with }u'\geq u}}\mathrm{Gr}_P^{u'}\right)\cap \mathrm{Gr}_{G,\leq \mu}^{\kappa_b=\kappa(b'')}\right).
\]
On the other hand, the union 
\[
\bigcup_{\substack{u'\in X_*(M^{ab}) \\ \textrm{with }u'\geq u}}\mathrm{Gr}_P^{u'}
\]
is closed in $\mathrm{Gr}_{G}$ by Proposition \ref{prop_GrPv}, and there exist only finitely many $u\in X_*(M^{ab})$ with $\mathrm{Gr}_P^u\cap \mathrm{Gr}_{G,\leq \mu}\neq\emptyset$ by Remark \ref{remark_GrPv finite}. So $\mathscr{S}_{b_P,\lambda}$ is a union of finitely many closed subfunctors of $\mathrm{Gr}_{G,\leq \mu}$. Consequently  
\[
|\mathscr S_{b_P,\lambda}|\subset |\mathrm{Gr}_{G,\leq \mu}|
\]
is closed.

Next we show that $|\mathscr S|\subset |\mathrm{Gr}_{G,\leq \mu}|$ is also closed. By Lemma \ref{lem:BM-and-BG} below, $[b]\in B(G)$ has finite preimage via the natural map $B(M)\ra B(G)$. We denote by $[b_i]\in B(M)$ ($1\leq i\leq n$) all these preimages. Suppose $g_ib_i\sigma(g_i)^{-1}=b$ with $g_i\in G(\breve F)$, and write $J_i:=J_{b_i}$ where $b_i$ is viewed as an element of $G(\breve F)$. Let $(b_P,g)$ be a reduction of $b$ to $P$. By Lemma \ref{lemma_B(P)}, there exist some $p\in P(\breve F)$ and $1\leq i\leq n$ such that $b_P=pb_i\sigma(p)^{-1}$, and thus
\[
b=gb_P\sigma (g)^{-1}=(gp)b_i\sigma(gp)^{-1}=g_ib_i\sigma(g_i)^{-1}.
\]
Therefore $g_i^{-1}gp\in J_i(F)$. Hence all the reductions of $b$ to $P$ are of the form
\[
(pb_i\sigma(p)^{-1},g), \quad \textrm{with}\quad p\in P(\breve F) \textrm{ and }g\in g_iJ_i(F)p^{-1}
\]
for some $1\leq i\leq n$. 
Moreover, as $[b_M]=[b_i]\in B(M)$, it follows that $\mathscr{S}_{b_P,\lambda}=\mathscr{S}_{b_i,\lambda}$, and
\[
\mathscr S=\bigcup_{\textrm{std parabolic }P}\left(\bigcup_{\substack{\textrm{reduction }(b_{P},g) \\ \textrm{of }b \textrm{ to }P}} g\cdot \mathscr{S}_{b_P,\lambda}\right)=\bigcup_{\textrm{std parabolic }P}\left(\bigcup_{i=1}^ng_i\cdot \left(\bigcup_{g\in J_i(F)}g\cdot \mathscr{S}_{b_i,\lambda}\right)\right).
\]
Since for a fixed standard parabolic $P$, the closed subfunctor $\mathscr{S}_{b_i,\lambda}$ is preserved by the action of an element of $P(\breve F)$, the union 
\[
\bigcup_{g\in J_i(F)}g\cdot |\mathscr{S}_{b_i,\lambda}|
\]
is closed in $|\mathrm{Gr}_{G,\leq \mu}|$ if we can show that the quotient $
J_i(F)/J_i(F)\cap P(\breve F)$ is compact. But $J_i(F)\cap P(\breve F)$ is the $F$-points of a parabolic subgroup $P_i\subset J_i$, so the compactness of 
\[
J_i(F)/J_i(F)\cap P(\breve F)=J_i(F)/P_i(F)\simeq (J_i/P_i)(F)
\]
follows from the fact that $J_i/P_i$ is projective over locally compact field $F$ (see \cite[15.1.4 Corollary]{Spr} for the last identification above). As a result, being a finite union of closed subsets 
\[
|\mathscr S|=\bigcup_{\textrm{std parabolic }P}\left(\bigcup_{i=1}^ng_i\cdot \left(\bigcup_{g\in J_i(F)}g\cdot |\mathscr{S}_{b_i,\lambda}|\right)\right)
\]
is closed in $|\mathrm{Gr}_{G,\leq \mu}|$.

Finally we show that the closed subset $|\mathscr S|\subset |\mathrm{Gr}_{G,\leq \mu}|$ is generalizing. To check this, let \[
f:T\longrightarrow \mathrm{Gr}_{G,\leq \mu}
\]
be a surjective map from a perfectoid space (such a map exists since $\mathrm{Gr}_{G,\leq \mu}$ is diamond). Because the induced map $|f|:|T|\ra |\mathrm{Gr}_{G,\leq \mu}|$ is generalizing (\cite[Proposition 11.18]{Sch2}) and surjective. So it is enough to check that the locus of points $x\in T$ whose image by $|f|$ is contained in $|\mathscr S|$ is generalizing. Since $T$ is an analytic adic space (hence every specialization in $T$ is vertical), so we are reduced to showing that for any perfectoid field $C$ with an open and bounded valuation subring $C^+\subset C$, and every morphism $\mathrm{Spa}(C,C^+)\ra \mathrm{Gr}_{G,\leq \mu}$, the induced map 
\[
|\mathrm{Spa}(C,C^+)|\longrightarrow |\mathrm{Gr}_{G,\leq \mu}|
\]
factors through $|\mathscr S|$ if and only if $|\mathscr S|$ contains the image of the closed point of $\mathrm{Spa}(C,C^+)$. But this is clear since $\mathrm{Gr}_{G}(C,C^+)$, $\mathrm{Gr}_P^u(C,C^+)$ and also the decomposition 
\[
\mathrm{Gr}_{G}(C,C^+)=\coprod_{u\in X_*(M^{ab})} \mathrm{Gr}_{P}^u(C,C^+)
\]
do not depend on the choice of $C^+$. 
\end{proof}

\begin{lemma}\label{lem:BM-and-BG} Let $G$ be a quasi-split group, and $M\subset G$ a standard Levi. Then, every element in $B(G)$ has at most finite preimages through the map $B(M)\ra B(G)$.
\end{lemma}

\begin{proof} Recall that by \cite[4.13]{RR}, the combination of the Newton map and the Kottwitz map give an injective map
\[
(\nu,\kappa): B(G)\longrightarrow \mathcal N(G)\times \pi_1(G)_{\Gamma}.
\]
Moreover, up to torsion, the Kottwitz map $\kappa$ is determined by the Newton map $\nu$. As the subgroup of torsion elements in $\pi_1(G)_{\Gamma}$ is finite, to prove our lemma, it suffices to check that the natural map
\[
\mathcal N(M)\longrightarrow \mathcal N(G)
\]
is finite-to-one. But this is clear since the later can be embedded into
\[
X_*(A)_{\mathbb Q}^{M,+}\longrightarrow X_*(A)_{\mathbb Q}^+, \quad \lambda\longmapsto \lambda_{\rm dom}.
\] which is finite-to-one.
\end{proof}

\begin{definition}Fix a triple $(G, \mu, b)$ and $[b']\in B(G)$.
\begin{enumerate}\item Let
\[\begin{split}\mathrm{Gr}_{G, \mu, b}^{\mathrm{HN}=[b']}:=\mathrm{Gr}_{G, b}^{\mathrm{HN}=[b']}\cap \mathrm{Gr}_{G, \mu},\\
 \mathrm{Gr}_{G, \mu, b}^{\mathrm{HN}\geq[b']}:=\mathrm{Gr}_{G, b}^{\mathrm{HN}\geq [b']}\cap \mathrm{Gr}_{G, \mu}.\end{split}\]The subfunctor $\mathrm{Gr}_{G, \mu, b}^{\mathrm{HN}=[b']}$ (resp. $\mathrm{Gr}_{G, \mu, b}^{\mathrm{HN}\geq [b']}$) is locally closed (resp. closed) in $\mathrm{Gr}_{G, \mu}$. This defines the Harder-Narasimhan stratification on $\mathrm{Gr}_{G, \mu}$. We will also call it HN-stratification for simplicity.

 \item If $b'$ is basic, then the HN-stratum $\mathrm{Gr}_{G, \mu, b}^{\mathrm{HN}=[b']}$ is called \emph{semi-stable} or \emph{weakly admissible}. There is at most one non-empty semi-stable locus and is denoted by by $\mathrm{Gr}_{G, \mu, b}^{wa}$. 
 \end{enumerate}
\end{definition}

\begin{remark} \label{remark_different wa}(1)Harder-Narasimhan stratification on the $\BpdR$-affine Grassmannian has been studied by many people. When $\mu$ is minuscule, this is studied by Dat-Orlik-Rapoport \cite[IX.6]{DOR}. Indeed, they study the Harder-Narasimhan stratification on the flag variety $\mathcal{F}(G, \mu)$ for arbitrary $\mu$. When $\mu$ is minuscule, the flag variety $\mathcal{F}(G, \mu)$ is isomorphic to $\mathrm{Gr}_{G, \mu}$ via the Bialynicki-Birula map. When $G=\mathrm{GL}_n$ or $b=1$ (or more generally $b$ is basic), the Harder-Narasimhan stratification is studied by  Nguyen-Viehmann (\cite{NV}) and Shen (\cite{Sh}) based on a result about the compatibility of Harder-Narasimhan filtration with tensor product by Cornut and Peche Irissarry (\cite{CPI}) which is Theorem \ref{thm_comp tensor} for the case $b=1$.

(2)Here the definition of the weakly admissible locus  $\mathrm{Gr}_{G, \mu, b}^{wa}$ is the same as the definition in \cite{Vi, NV} and is more general than the notion in some literatures such as \cite{RZ, CFS, Sh} where the weakly admissible locus is referred to $\mathrm{Gr}_{G, \mu, b}^{\mathrm{HN}=[b']}$ with $[b']=[1]$. Therefore in these literatures, they usually impose the condition $[b]\in B(G, \mu)$. Under this condition, our definition of weakly admissible locus coincides with theirs. 
\end{remark}

\section{Compatibility of HN-stratification}
In the rest of the article, we change the notations. Let $C$ be a complete algebraically closed non-archimedean field over $\bar F$ (unlike in the previous sections, $C$ usually stands for a complete algebraically closed non-archimedean field of characteristic $p$).  Let $C^+\subset C$ be an open valuation subring. To simplify notation, we write $\mathrm{Gr}_{G}(C, C^+)$ (resp. $\mathrm{Gr}_{G, \mu}(C, C^+)$, resp. $\mathrm{Gr}_{G, \leq \mu}(C, C^+)$) for $\mathrm{Gr}_{G}(C^{\flat}, C^{\flat, +})$ (resp. $\mathrm{Gr}_{G, \mu}(C^{\flat}, C^{\flat, +})$, resp. $\mathrm{Gr}_{G, \leq \mu}(C^{\flat}, C^{\flat, +})$). As $\mathrm{Gr}_{G}(C, C^+)$ only depends on $C$ but not on $C^+$, we also write  $\mathrm{Gr}_{G}(C)$ for $\mathrm{Gr}_{G}(C, C^+)$. The notations $\mathrm{Gr}_{G, \mu}(C)$ and $\mathrm{Gr}_{G, \leq \mu}(C)$ are similar.

In this section, we want to compare the HN-stratification for different groups (inner twists and adjoint) and also compare it with the Newton stratification on $\mathrm{Gr}_{G, \mu}$ and the HN-stratification on flag varieties.

\subsection{Compatibility with inner twists}\label{subsection_comp inner twists HN strat} We take notations as in \S\ref{subsection_comp inner twists}. Let $H=J_{b_0}$ be an inner twist of $G$ with $b_0\in G(\breve F)$ basic. Then the inner twisting $\iota: H_{\breve F}\rightarrow G_{\breve F}$ induces an isomorphism $\mathrm{Gr}_{H, \breve{F}}\stackrel{\sim}{\rightarrow} \mathrm{Gr}_{G, \breve{F}}$ on which the Harder-Narasimhan stratifications on both sides are compatible by Proposition \ref{prop_comp inner twists}, i.e.
\[
\mathrm{Gr}_{H,\mu, b}^{\mathrm{HN}=[b']}\stackrel{\sim}{\longrightarrow} \mathrm{Gr}_{G, \iota(\mu), \iota(b)b_0}^{\mathrm{HN}=[\iota(b')b_0]}.
\]
Therefore, when we study the properties of HN-strata, we may always assume that $G$ is quasi-split.

\subsection{Compatibility with adjoint groups}\label{subsection_adjoint} Consider the natural homomorphism
\[
(-)^{ad}: G\longrightarrow G^{ad}.\]
This induces a morphism
\[
\pi: \mathrm{Gr}_{G}\longrightarrow \mathrm{Gr}_{G^{ad}}
\]
which is compatible with Harder-Narasimhan stratification on both sides in the following sense:

\begin{proposition}
Suppose $\mathrm{Gr}_{G^{ad}, \mu^{ad}, b^{ad}}^{\mathrm{HN}=[\tilde{b'}]}$ is non-empty, then
\[
\pi^{-1}\left(\mathrm{Gr}_{G^{ad}, \mu^{ad}, b^{ad}}^{\mathrm{HN}=[\tilde{b'}]}\right)=\mathrm{Gr}_{G,\mu, b}^{\mathrm{HN}=[b']},\] where $[b']\in B(G)$ is the unique element which is mapped to $[\tilde{b}']$ via $B(G)\rightarrow B(G^{ad})$ and to $\kappa_G(b')=\kappa_G(b)-\mu^{\sharp}$ in $\pi_1(G)_{\Gamma}$.
\end{proposition}
\begin{proof}Here the existence and uniqueness of $[b']$ is due to Kottwitz \cite[Corollary 4.11]{Kot}.

 By subsection \ref{subsection_comp inner twists HN strat}, we may assume that $G$ is quasi-split. Note that for any $x\in \mathrm{Gr}_G(C)$,
\[\E_{b^{ad}, \pi(x)}=\E_{b, x}\times^G G^{ad}.\] And there is a canonical bijection
\[\begin{split}\{\text{standard parabolic subgroups of } G\}&\longrightarrow \{\text{standard parabolic subgroups of } G^{ad}\}\\
Q&\longmapsto Q^{ad}\end{split}.\]  Moreover, the natural map
\[\begin{split}\{\text{reduction of } b \text{ to } Q\}&\longrightarrow \{\text{reduction of } b^{ad} \text{ to } Q^{ad}\}\\ (b_Q, g)&\longmapsto (b_Q^{ad}, g^{ad})\end{split}\] is surjective with slope vectors \[v_{(b_Q, g)}^{ad}=v_{(b_Q^{ad}, g^{ad})},\] where the slope vectors are defined in subsection \ref{subsection_HN-vector}. Then the result follows from Proposition \ref{prop_HN vector maximum of slope vector}.
\end{proof}

\subsection{Compatibility with Newton stratification} On the $\BpdR$-affine Grassmannian $\mathrm{Gr}_{G,\mu}$, we have another stratification called the Newton stratification. Let us first briefly recall its definition. For each geometric point $x\in \mathrm{Gr}_{G,\mu}(C,C^+)$, the isomorphism class of the modification $\mathcal E_{b,x}$ of the $G$-bundle $\mathcal E_b$ corresponds to a unique element $\mathrm{New}(b,x)\in B(G)$ by a theorem Fargues (\cite{Far20}, see also \S~\ref{subsubsection:G-bundles}). So we obtain a map 
\[
\mathrm{Gr}_{G,\mu}(C,C^+)\longrightarrow B(G).
\]
Letting $(C,C^+)$ vary, we deduce a map
\begin{eqnarray}\label{eqn_Newton vector}
\mathrm{Newt}_b: |\mathrm{Gr}_{G,\mu}|\longrightarrow B(G), \quad x\longmapsto  \mathrm{New}(b,x).
\end{eqnarray}
For $[b']\in B(G)$, the corresponding Newton strata $\mathrm{Gr}_{G,\mu,b}^{\mathrm{New}=[b']}$ is defined to be the preimage of $[b']\in B(G)$ via the map $\mathrm{New}_b$: so for any complete algebraically closed field extension $C$ of $\Fbar$ and for any valuation subring $C^+\subset C$,
\[
\mathrm{Gr}_{G,\mu,b}^{\mathrm{New}=[b']}(C,C^+)=\{x\in\mathrm{Gr}_{G,\mu}(C,C^+)| \mathcal E_{b,x}\simeq \mathcal E_{b'}\}.
\]
In particular, we have a decomposition
\begin{equation}\label{eq:Newton-stratification}
\mathrm{Gr}_{G,\mu}=\coprod_{[b']\in B(G)} \mathrm{Gr}_{G,\mu,b}^{\mathrm{New}=[b']}
\end{equation}
of the $\BpdR$-affine Grassmannian $\mathrm{Gr}_{G,\mu}$. Write also 
\[
\mathrm{Gr}_{G,\mu,b}^{\mathrm{New}\geq [b']}=\bigcup_{B(G)\ni[b'']\geq [b']}\mathrm{Gr}_{G,\mu,b}^{\mathrm{New}=[b'']}.
\]
 The non-empty Newton stratum $\mathrm{Gr}_{G, \mu, b}^{\mathrm{New}=[b']}$ with $[b']\in B(G)$ basic is called the admissible locus and is denoted by $\mathrm{Gr}_{G, \mu, b}^a$.

\begin{remark} The definition of the admissible locus is more general to the one in some literature which is parallel to the situation of the weakly admissible locus (cf. \ref{remark_different wa} (2)). It is also known that each Newton stratum $\mathrm{Gr}_{G,\mu,b}^{\mathrm{New}=[b'']}$ is locally closed and  $\mathrm{Gr}_{G,\mu,b}^{\mathrm{New}\geq  [b'']}$ is closed in $\mathrm{Gr}_{G,\mu}$ (\cite[Theorem 7.4.5]{KL}, \cite[Corollary 3.5.9]{CS}, \cite[\S~3.1]{Vi}). 
\end{remark}

\begin{proposition} For any $x\in \mathrm{Gr}_{G, \mu}(C)$, we have
\[\mathrm{HN}(b, x)\leq \mathrm{New}(b, x),\] which implies \[\mathrm{Gr}_{G, \mu, b}^{\mathrm{HN}=[b']}\subseteq \mathrm{Gr}_{G, \mu, b}^{\mathrm{New}\geq [b']}.\] In particular, when $[b]\in B(G, \mu)$, we have $\mathrm{Gr}_{G, \mu, b}^a\subseteq \mathrm{Gr}_{G, \mu, b}^{wa}$.

\end{proposition}
\begin{proof}
The inequality $\mathrm{HN}(b,x)\leq \mathrm{New}(b,x)$ following from Corollary \ref{coro_HN vector}, from which the second assertion follows. 
\end{proof}

\begin{remark}When $b=1$, this result is known by Nguyen-Viehmann \cite[Lemma 6.1]{NV} and Shen \cite[Proposition 3.4]{Sh}.
\end{remark}

\subsection{Compatibility with HN-stratification on flag varieties} \label{sec:compactibility-with-flag}
Dat-Orlik-Rapoport  defined in \cite[IX.6]{DOR} the Harder-Narasimhan formalism for the flag variety associated to the triple $(G, \mu, b)$.
We recall briefly its construction which is parallel to ours.

Let $K|\breve F$ be a field extension. Denote by 
\[
\mathbf{FilIsoc}_{\breve{F}|F}^K
\]
the category of filtered isocrystals consisting of triples $(N, \varphi, \mathrm{Fil}^{\bullet}N_K)$, where
\begin{itemize}
\item $(N, \varphi)$ is an isocrystal over $\breve{F}|F$,
\item $\mathrm{Fil}^{\bullet}N_K$ is a separated exhaustive decreasing $\Z$-filtration of $K$-subspaces on $N_K:=N\otimes_{\breve{F}}K$. 
\end{itemize}
For an object $(N, \varphi, \mathrm{Fil}^{\bullet}N_K)$ in $\mathbf{FilIsoc}_{K|\breve{F}}$, define the slope
\[
\mathrm{slope}(N, \varphi, \mathrm{Fil}^{\bullet}N_K):=\frac{\deg (N, \varphi, \mathrm{Fil}^{\bullet}N_K)}{\mathrm{rank} (N, \varphi, \mathrm{Fil}^{\bullet}N_K)},\]
where 
\[\begin{split}\mathrm{rank} (N, \varphi, \mathrm{Fil}^{\bullet}N_K)&:=\mathrm{rank}_{\breve{F}} N,\\  
\deg (N, \varphi, \mathrm{Fil}^{\bullet}N_K)&:=\deg(\mathrm{Fil}^{\bullet}N_K)-\deg(N, \varphi),\end{split}
\]
with $\deg(\mathrm{Fil}^{\bullet}N_K):=\sum_{i\in\Z}i \dim_K(\mathrm{Fil}^{i}N_K/ \mathrm{Fil}^{i+1}N_K)$ and $\deg (N, \varphi)=v_F(\det(\varphi))$ is the $p$-adic valuation of the determinant of $\varphi$ where $v_F$ denotes the normalized valuation on $F$. This slope function leads to a Harder-Narasimhan formalism on the category $\mathbf{FilIsoc}_{K|\breve{F}}$. In particular, we have the Harder-Narasimhan filtration for each objects which defines the functor
\[
\tilde{\mathcal{F}}_{\rm HN}: \mathbf{FilIsoc}_{\breve{F}|F}^K\longrightarrow \mathbf{F}(\mathrm{Isoc}_{\breve{F}|F}).
\] 
It is known that $\tilde{\mathcal F}_{\rm HN}$ is compatible with tensor products (\cite{Fa},  \cite{To1}).

Now we apply similar argument as in Section \ref{subsection_HN-vector} to construct HN-vector associated to a pair $(b, x)$ with $x\in \mathcal{F}(G, \mu)(K)$. We have the functor
\[\begin{array}{cccccc} \tilde{\mathcal{F}}_{HN}(b, x): &\mathrm{Rep}_F(G)&\longrightarrow &\mathbf{FilIsoc}_{\breve{F}|F}^K&\stackrel{\tilde{\mathcal{F}}_{HN}}{\longrightarrow} &\mathbf{F}(\mathrm{Isoc}_{\breve{F}|F}).\\
&(V, \rho)&\longmapsto &(V_{\breve{F}}, \rho(b)\sigma, \mathrm{Fil}^{\bullet}_{\rho(x)}V_K)& &\end{array}\]
There exists some rational cocharacter of $G$ which splits $\tilde{\mathcal{F}}_{HN}(b, x)$. Its conjugacy class does not depend on the choice of the splitting and is denoted by $\tilde{v}_{b, x}$. 
Similar arguments as Corollary \ref{coro_HN vector} shows that there exists a unique element $\mathrm{HN}(b, x)\in B(G)$ called the Harder-Narasimhan vector for the pair $(b, x)$, such that 
\[
\nu_{\mathrm{HN}(b, x)}=-\tilde{v}_{b,x} \quad \textrm{and}\quad  \kappa(\mathrm{HN}(b, x))=\kappa(b')=\kappa(b)-\mu^\sharp.
\]
The parametrization by the HN-vector gives the Harder-Narasimhan stratification on the flag variety $\mathcal{F}(G, \mu)$:
\[\mathcal{F}(G, \mu)=\coprod_{[b']\in B(G)}\mathcal{F}(G, \mu, b)^{\mathrm{HN}=[b']},\] 
where for each geometric point $x\in \mathcal{F}(G, \mu)(C,C^+)$, \[x\in \mathcal{F}(G, \mu, b)^{\mathrm{HN}=[b']}(C,C^+) \Longleftrightarrow\mathrm{HN}(b, x)=[b'].\] Define also 
\[ \mathcal{F}(G, \mu, b)^{\mathrm{HN}\geq [b']}:=\coprod_{[b'']\geq [b']} \mathcal{F}(G, \mu, b)^{\mathrm{HN}=[b'']}.\]

\begin{remark}In \cite{DOR}, the index for the each HN-stratum is given by $\tilde{v}_{b, x}$. Here we take the index to be $\mathrm{HN}(b, x)$ in order to be compatible with the index of the HN-strata and Newton strata on the $\BpdR$-Grassmannian.
\end{remark}

\begin{proposition}\label{prop_comparison via BB}Via the Bialynicki-Birula map $\mathrm{BB}: \mathrm{Gr}_{G, \mu}\rightarrow \mathcal{F}(G, \mu)$, we have 
\[\mathrm{HN}(b, x)\leq \mathrm{HN}(b, \mathrm{BB}(x)).\] In particular, \[\mathrm{BB}^{-1}(\mathcal{F}(G, \mu, b)^{\mathrm{HN}=[b']})\subseteq \coprod_{[b'']\leq [b']\in B(G), }\mathrm{Gr}_{G, \mu, b}^{\mathrm{HN}=[b'']},\] and $\mathrm{Gr}_{G, \mu, b}^{wa}\supseteq \mathrm{BB}^{-1}(\mathcal{F}(G, \mu, b)^{ss})$.
\end{proposition}
\begin{proof}By Tannakian formalism, we reduce to the $\mathrm{GL}_n$ case. In this case, the $\mathrm{BB}$ map gives a functor still denoted by $\mathrm{BB}$:
\[\begin{split}\mathrm{BB}: \mathbf{BunIsoc}_{\breve{F}|F}^{\BdR} &\longrightarrow \mathbf{FilIsoc}_{\breve{F}|F}^C\\
(D, \Xi)&\longmapsto (D, \mathcal{F}_{\Xi}),\end{split}\] where $\mathcal{F}_{\Xi}$ is defined in Remark \ref{rem:lattice to filtration}. 

We want to compare the HN vectors on both sides. First note that the $\mathrm{BB}$ map preserves the rank and degree functions, i.e., $\deg(D,\Xi)=\deg(D,\mathcal F_{\Xi})$ and $\mathrm{rank}(D,\Xi)=\mathrm{rank}(D,\mathcal F_{\Xi})$. Let $D_1$ be a sub-isocrystal of $D$. Let $(D_1, \Xi_1)$ (resp. $(D_1, \mathcal{F}_{\Xi, 1})$) be the corresponding strict sub-object of $(D, \Xi)$ (resp. $(D, \mathcal{F}_{\Xi})$) in $\mathrm{BunIsoc}_{\breve{F}|F}^{\BdR}$ (resp. $\mathrm{FilIsoc}_{\breve{F}|F}^C$): so 
\[
\Xi_1=(D_1\otimes_{\breve F}\BdR)\cap \Xi
\]
and $\mathcal{F}_{\Xi, 1}$ is the intersection of $\mathcal{F}_{\Xi}$ with $D_1\otimes_{\breve F}C$. In general, the $\mathrm{BB}$ map does NOT preserve strict sub-objects, i.e., $\mathcal{F}_{\Xi, 1}\neq \mathcal{F}_{\Xi_1}$. More precisely note that for $i\in\Z$, 
\[\begin{split}\mathcal{F}_{\Xi_1}^i&=\frac{\pi^i (D_1\otimes \BdR\cap \Xi)\cap \Xi_{1, 0}+\pi\Xi_{1, 0}}{\pi \Xi_{1, 0}}\subset D_1\otimes_{\breve F}C\\
\mathcal{F}_{\Xi, 1}^i&=\frac{\pi^i\Xi\cap \Xi_0+\pi \Xi_0}{\pi \Xi_0}\cap (D_1\otimes_{\breve F} C),\end{split}
\] where $\Xi_0=D\otimes \BpdR$ and $\Xi_{1, 0}=D_1\otimes \BpdR$. We can check that the identity map on $D_1\otimes_{\breve F}C$ induces a morphism of filtered $C$-vector spaces 
\[
(D_1\otimes C,\mathcal{F}_{\Xi_1})\longrightarrow (D_1\otimes C,\mathcal{F}_{\Xi, 1})
\]
which implies that 
\[
\deg\mathcal{F}_{\Xi_1}\leq \deg \mathcal{F}_{\Xi, 1},
\]
and hence
\[\deg (D_1, \Xi_1)=\deg(D_1,\mathcal{F}_{\Xi_1})\leq \deg(D_1, \mathcal{F}_{\Xi, 1}),
\]
where the first equality holds since the $\mathrm{BB}$ map preserves degree. On the other hand, \[\mathrm{rank} (D_1, \Xi_1)=\mathrm{rank}_{\breve F}D_1=\mathrm{rank}(D_1, \mathcal{F}_{\Xi, 1}),\] it follows that  $\mathrm{HN}(b, x)\leq \mathrm{HN}(b, \mathrm{BB}(x))$.
\end{proof}

\begin{remark}\label{remark_classical} (1) Shen mentioned in \cite[\S~2.4]{Sh} that the BB map 
\[
\mathrm{BB}: \mathbf{BunIsoc}_{\breve{F}|F}^{\BdR} \longrightarrow \mathbf{FilIsoc}_{\breve{F}|F}^C
\]
does not preserve strict sub-objects. Viehmann gave an example in \cite[Example 4.10]{Vi} to illustrate that the semi-stable locus of the two sides does not correspond to each other via the $\mathrm{BB}$ map.

(2) We keep the notations as in the proof of the previous proposition, the lattice $\Xi$ could be considered as an element in $\mathrm{Gr}_{G, \mu}(C)$ for some $\mu$, where $G=\mathrm{GL}(D)$. 
Let $M_1=\mathrm{GL}(D_1)$ and  $P=\mathrm{stab}(D_1)\subset G$ a parabolic subgroup with Levi component $\mathrm{GL}(D_1)\times \mathrm{GL}(D/D_1)$. Then is a natural projection $\mathrm{pr}_{M_1}: P\rightarrow M_1$.

The fact that the $\mathrm{BB}$ map is not compatible with subobjects could be reformulated by the fact that the following diagram is not commutative:  

\begin{center}
	\begin{tikzpicture}
	\matrix(a)[matrix of math nodes, row sep={1cm,between origins}, column sep={1.2cm,between origins},
	text height=1.5ex, text depth=0.25ex]
	{\mathrm{Gr}_{G,\mu}(C) & & & & \mathcal{F}(G, \mu)(C)\\
	 & \Xi & & \mathcal{F}_{\Xi} & \\ 
	 & & & & \\ 
     & \Xi_1 & & \mathcal{F}_{\Xi_1}, \mathcal{F}_{\Xi, 1} \ \ \quad & \\
	 \mathrm{Gr}_{M_1}(C) & & & &\coprod_{\mu_1}\mathcal{F} (M_1, \mu_1)(C)\\};
	\path[->] (a-1-1) edge node[above]{{\scriptsize $\mathrm{BB}$}} (a-1-5); 
	\path[->] (a-1-1) edge node[left]{\scriptsize $\mathrm{pr}_{\mathrm{Gr}}$} (a-5-1); 
	\path[->] (a-5-1) edge node[below]{\scriptsize $\mathrm{BB}$} (a-5-5); 
	\path[->] (a-1-5) edge node[right]{\scriptsize $\mathrm{pr}_{\mathcal{F}}$} (a-5-5); 
	\path[|->] (a-2-2) edge node[above]{} (a-2-4); 
	\path[|->] (a-2-2) edge node[left]{} (a-4-2); 
	\path[|->] (a-4-2) edge node[below]{} (a-4-4); 
	\path[|->] (a-2-4) edge node[right]{} (a-4-4); 
\end{tikzpicture}
\end{center}

where \begin{itemize}
\item $\mathrm{pr}_{\mathrm{Gr}}(gG(\BpdR)):=\mathrm{pr}_{M_1}(p)M_1(\BpdR)$, if $gG(\BpdR)=pG(\BpdR)$ for some $p\in P(\BdR)$ by Iwasawa decomposition. 
\item $\mathrm{pr}_{\mathcal{F}}(gP_{\mu}(C)):=\mathrm{pr}_{M_1}(\tilde{p})\in \mathcal{F}(M_1, \mu_1)(C)$, where $\mu_1$ is the projection of $w\mu$ to the $M_1$-component,  if $g\in \tilde{p} wP_{\mu}(C)$ for some $\tilde{p}\in P(C)$ by Bruhuat decomposition.  
\item $\mathrm{pr}_{\mathrm{Gr}}(\Xi)=\Xi_1$ and $\mathrm{pr}_{\mathcal{F}}(\mathcal{F}_{\Xi})=\mathcal{F}_{\Xi, 1}$.
\end{itemize}
The non-commutativity of the diagram follows from the fact that the Iwasawa decomposition and Bruhat decomposition on $G(\BdR)$ are in general NOT compatible. 

On the other hand, for $\Xi\in \mathrm{Gr}_{G, \mu}$, if
\begin{eqnarray}\label{eqn_commutative}\Xi\in P(\BpdR)w\mu^{-1}(\xi)G(\BpdR)/G(\BpdR)(C)\end{eqnarray}
for some $w\in W_G$, then we can read off simultaneously the Iwasawa decomposition and the Bruhat decomposition for $\Xi$, and it follows by direct verification that the previous diagram is commutative for such that $\Xi$. For example, when $\mu$ is minuscule or $\Xi\in G(C)\mu^{-1}(\xi)G(\BpdR)/G(\BpdR)$ for some section $C\rightarrow \BpdR$, then $\Xi$ satisfies (\ref{eqn_commutative}) (compare also \cite[Theorem 5.1]{NV}).
\end{remark}


\subsection{Classical points}
Recall that $\{\mu\}$ is a geometric conjugacy class of cocharacters of the $F$-reductive group $G$, with $E$ its reflex field inside a fixed algebraic closure $\bar F$ of $F$. Write $\breve E=E\cdot \breve F$, the maximal unramified extension of $E$. Then the $\BpdR$-affine Grassmannian $\mathrm{Gr}_{G,\mu}$ is defined over $E$, giving in particular a functor on $\mathbf{Perf}/\mathrm{Spd}(\breve E)$. By abusing the notation, for a finite field extension $K/\breve E$, we write 
\[
\mathrm{Gr}_{G,\mu}(K):=\mathrm{Hom}(\mathrm{Spd}(K),\mathrm{Gr}_{G,\mu})
\]
and call it the set of $K$-valued classical points of $\mathrm{Gr}_{G,\mu}$. Same notation applies if we replace $\mathrm{Gr}_{G,\mu}$ by $\mathcal F(G,\mu)^{\diamond}$ or more generally by any v-sheaf on $\mathbf{Perf}/\mathrm{Spd}(\breve E)$.

\begin{lemma}\label{lem:description-classical-pts} Let $K/\breve E$ be a finite field extension and denote by $C$ the p-adic completion of an algebraic closure of $K$.
\begin{enumerate}
\item The natural maps 
\[
\mathcal F(G,\mu)^{\diamond}(K)\longrightarrow \mathcal F(G,\mu)^{\diamond}(C)=\mathcal F(G,\mu)(C) \quad \textrm{and}\quad \mathrm{Gr}_{G,\mu}(K)\longrightarrow \mathrm{Gr}_{G,\mu}(C)
\]
are injective, and induce bijections 
\[
\mathcal F(G,\mu)^{\diamond}(K)\stackrel{\sim}{\longrightarrow} \mathcal F(G,\mu)(C)^{\mathrm{Gal}(C/K)} \quad \textrm{and}\quad \mathrm{Gr}_{G,\mu}(K)\stackrel{\sim}{\longrightarrow} \mathrm{Gr}_{G,\mu}(C)^{\mathrm{Gal}(C/K)}. 
\]
\item (\cite[Proposition 10.4.4]{FF}, \cite[Proposition 5.1]{Vi}) The Bialynicki-Birula map induces a bijection $
\mathrm{Gr}_{G, \mu}(K)\stackrel{\sim}{\rightarrow} \mathcal{F}(G, \mu)^{\diamond}(K)$.
\end{enumerate} 
\end{lemma}

\begin{proof} The first part of this lemma is well-known and the second part of the lemma is proved by \cite{FF} and \cite{Vi}. But we still give the proof for the convenience of the readers. 

The morphism $\mathrm{Spd}(C)\ra \mathrm{Spd}(K)$ of pro-\'etale sheaves  on $\mathbf{Perf}/\mathrm{Spd}(\breve E)$ is surjective. It follows that the natural map 
\[
\mathcal{F}(G,\mu)^{\diamond}(K)\longrightarrow \mathcal F(G,\mu)^{\diamond}(C)
\]
is injective, whose image is obviously contained in $\mathcal F(G,\mu)(C)^{\mathrm{Gal}(C/K)}$. On the other hand, as $C^{\mathrm{Gal}(C/K)}=K$, every element $x\in \mathcal F(G,\mu)(C)^{\mathrm{Gal}(C/K)}$ is a $K$-point of $\mathcal F(G,\mu)$. The latter defines a morphism $\mathrm{Spd}(K)\ra \mathcal F(G,\mu)^{\diamond}$ by mapping any affinoid perfectoid space $S=\mathrm{Spa}(R,R^+ )$ over $K$ to the $R$-point of $\mathcal F(G,\mu)$ induced by $x$. Therefore $\mathcal F(G,\mu)^{\diamond}(K)\stackrel{\sim}{\ra}\mathcal F(G,\mu)^{\diamond}(C)^{\mathrm{Gal}(C/K)}$.  

The remaining part of the proof is adapted from the proof of \cite[Proposition 5.1]{Vi}. Recall that in the $\mathrm{GL}_n$-case, $\mathrm{Gr}_{\mathrm{GL}_n}(C)^{\mathrm{Gal}(C/K)}$ classifies the $\BpdR(C)$-lattices in $\BdR(C)^n$ that are invariant under the action of $\mathrm{Gal}(C/K)$. Therefore, by \cite[Proposition 10.4.4]{FF}, the Bialynicki-Birula map induces a bijection 
\[
\mathrm{Gr}_{\mathrm{GL}_n,\mu}(C)^{\mathrm{Gal}(C/K)}\stackrel{\sim}{\longrightarrow } \mathcal F(\mathrm{GL}_n,\mu)(K). 
\]
Now for a general reductive group $G$, we take a closed embedding $G\hookrightarrow \mathrm{GL}_n$ over $F$. The latter induces a commutative diagram as below 
\[
\xymatrix{\mathrm{Gr}_{G,\mu}(K)\ar[r]\ar[d] & \mathrm{Gr}_{G,\mu}(C)^{\mathrm{Gal}(C/K)}\ar[r]^{\rm BB}\ar[d] & \mathcal F(G,\mu)^{\diamond}(C)^{\mathrm{Gal}(C/K)}\ar[d]\ar@{=}[r] & \mathcal F(G,\mu)(K)\ar[d]\\ \mathrm{Gr}_{\mathrm{GL}_n,\mu}(K)\ar[r] & \mathrm{Gr}_{\mathrm{GL}_n,\mu}(C)^{\mathrm{Gal}(C/K)}\ar[r]^{\rm BB} & \mathcal F(\mathrm{GL}_n,\mu)^{\diamond}(C)^{\mathrm{Gal}(C/K)}\ar@{=}[r] & \mathcal F(\mathrm{GL}_n,\mu)(K) }.
\]
Here the vertical maps are all induced from the closed immersion $G\hookrightarrow \mathrm{GL}_n$ thus are injective: for example the  injectivity of the first two maps follows from \cite[Lemma 19.1.5]{SW}. As the Bialynicki-Birula map for $\mathrm{GL}_n$ in the diagram above is bijective, together with the fact that the second vertical map (from left) is injective, we deduce that the Bialynicki-Birula map for general $G$ induces an injective map 
\begin{equation}\label{eq:BB-for-K}
\mathrm{Gr}_{G,\mu}(K)\longrightarrow \mathcal F(G,\mu)^{\diamond}(K)=\mathcal F(G,\mu)(K).
\end{equation}

To complete the proof of our lemma, it remains to show that the map \eqref{eq:BB-for-K} is also surjective (which implies also $\mathrm{Gr}_{G,\mu}(K)\stackrel{\sim}{\ra}\mathrm{Gr}_{G,\mu}(C)^{\mathrm{Gal}(C/K)}$). Take $x\in \mathcal F(G,\mu)(K)$. Since the field $K$ is strictly henselian with algebraically closed residue field, $x$ lifts to an element in $G(K)$ still denoted by $x$. On the other hand, the $\xi$-adic complete ring $\BpdR(C)$ is naturally an $\breve F$-algebra whose residue field $\BpdR(C)/\xi\BpdR(C)\simeq C$ contains the finite separable extension $K/\breve F$, it follows by Hensel's Lemma that $\BpdR(C)$ is also naturally an algebra over $K$. Hence the element $x\in G(K)$ can be viewed as an element of $G(\BpdR(C))$, yielding an element 
\[
x\mu(\xi)^{-1}G(\BpdR(C))\in \mathrm{Gr}_{G,\mu}(C)
\]
whose image via the Bialynicki-Birula map is $x\in \mathcal F(G,\mu)(C)$. Furthermore, $x$ also defines a morphism 
\[
\alpha:\mathrm{Spd}(K)\longrightarrow \mathrm{Gr}_{G,\mu}
\] by sending every affinoid perfecoid $\mathrm{Spa}(R,R^ +)$ over $K$ to $x\mu(\xi)^ {-1}\BpdR(R)\in \mathrm{Gr}_{G,\mu}(R,R^+)$. Here as $R$ is an $K$-algebra, the period ring $\BpdR(R)$ contains naturally $K$ as a subring by the similar argument using Hensel's Lemma mentioned above. Finally we check that the image by \eqref{eq:BB-for-K} of $\alpha$ is exactly $x$, showing the desired surjectivity.  
\end{proof}

\begin{remark}[cf.  {\cite[Remark 5.4]{NV}}]\label{remark_classical definition} We keep the notation of Lemma \ref{lem:description-classical-pts}. By the proof of the lemma above, a $K$-valued classical point, viewed as an element in $\mathrm{Gr}_{G, \mu}(C)$, always has the form $x\mu(\xi)^{-1}G(\BpdR)$ for some $x\in G(K)$, where $K$ is identified with $\BdR^{\mathrm{Gal}(C|K)}$.  
\end{remark}

The main result for classical points is the following theorem. 

\begin{theorem}\label{thm_classical}Let $x$ be a classical point of $\mathrm{Gr}_{G, \mu}$ with value in some finite field extension of $\breve E$, and let $[b']\in B(G)$, then the following statements are equivalent:
\begin{enumerate}
\item $x\in \mathrm{Gr}_{G,\mu, b}^{\mathrm{New}=[b']}$,
\item $x\in \mathrm{Gr}_{G,\mu, b}^{\mathrm{HN}=[b']}$,
\item $\mathrm{BB}(x)\in \mathcal{F}(G,\mu, b)^{\mathrm{HN}=[b']}$.
\end{enumerate}
\end{theorem}

\begin{remark}When $b$ is basic, this theorem is proved in \cite[Theorem 5.2]{Vi} for $[b']=1$ and in \cite[Proposition 5.5]{NV} for general $b'$.
\end{remark}

The rest of this section is devoted to the proof of Theorem \ref{thm_classical}. We first prove some weaker results. 

\begin{proposition}\label{prop_classical_HN=flag}Let $K$ be a finite extension of $\breve E$ and $[b']\in B(G)$, then
\[
\mathrm{BB}\left(\mathrm{Gr}_{G,\mu, b}^{\mathrm{HN}=[b']}(K)\right)=\mathcal{F}(G,\mu, b)^{\mathrm{HN}=[b']}(K).\]
\end{proposition}
\begin{proof}It suffices to show $\mathrm{HN}(b, x)=\mathrm{HN}(b, \mathrm{BB}(x))$ for any $x\in\mathrm{Gr}_{G, \mu}(K)$. When $b$ is basic, this is done in \cite[Theorem 5.1]{NV}. And their proof could be adapted to this setup. For the convenience of the readers, we still give the proof here.

Using Tannakian formalism, we may first reduce to the $G=\mathrm{GL}_n$ case. Recall that in general, the Bialynicki-Birula map \[\begin{split}\mathrm{BB}: \mathrm{BunIsoc}_{\breve{F}|F}^{\BdR} &\longrightarrow \mathrm{FilIsoc}_{\breve{F}|F}^C\\
(D, \Xi)&\longmapsto (D, \mathcal{F}_{\Xi}),\end{split}\] 
is NOT compatible with subobjects. Here, we want to show that $\mathrm{BB}$ is compatible with subojects for the particular object $(D_b, \Xi_x)\in \mathrm{BunIsoc}_{\breve{F}|F}^{\BdR}$, where $D_b$ is the isocrystal corresponding to $[b]$ and $\Xi_x$ is the lattice corresponding to $x$. For simplicity, we will write $D=D_b$ and $\Xi=\Xi_x$.
Let $D_1$ be a sub-isocrystal of $D$. Let $(D_1, \Xi_1)$ (resp. $(D_1, \mathcal{F}_{\Xi, 1})$) be the corresponding sub-object of $(D, \Xi)$ (resp. $(D, \mathcal{F}_{\Xi})$) in $\mathrm{BunIsoc}_{\breve{F}|F}^{\BdR}$ (resp. $\mathrm{FilIsoc}_{\breve{F}|F}^C$). We want to show $\deg(D_1, \Xi)=\deg(D_1, \mathcal{F}_{\Xi, 1})$. 
It suffices to show $\mathcal{F}_{\Xi_1}= \mathcal{F}_{\Xi,1}$.  By Remark \ref{remark_classical definition} and Remark \ref{remark_classical} (2), $x\in \mathrm{Gr}_{G, \mu}(K)$ satisfies (\ref{eqn_commutative}). It follows that the diagram in Remark \ref{remark_classical} (2) is commutative for the particular element $\Xi$ and hence $\mathrm{BB}$ map is compatible with subobjects for the particular object $(D, \Xi)$. In particular
\[\mathrm{BB}(D_1, \Xi_1)=(D_1, \mathcal{F}_{\Xi, 1}),\] and hence $\mathcal{F}_{\Xi_1}= \mathcal{F}_{\Xi,1}$.
\end{proof}


The following result shows that weakly admissible locus can always  be considered as an algebraic approximation of the admissible locus.
\begin{proposition}\label{Prop_a=wa} Let $x$ be a classical point of $\mathrm{Gr}_{G, \mu}$, then 
\[x\in \mathrm{Gr}_{G,\mu, b}^{a}\Longleftrightarrow x\in \mathrm{Gr}_{G,\mu, b}^{wa}\]
\end{proposition} 
\begin{remark}When $\mu$ is minuscule and $G=\mathrm{GL}_n$, this theorem is proved by Colmez-Fontaine \cite[Theorem A]{CF}. When $b$ is basic, this theorem is proved by Viehmann \cite[Theorem 5.2]{Vi}.
\end{remark}

\begin{proof}The proof is the same as \cite[Theorem 5.2]{Vi}. So we only outline a sketch. 

By Proposition \ref{prop_classical_HN=flag},  
\[x\in \mathrm{Gr}_{G,\mu, b}^{wa}\Longleftrightarrow \mathrm{BB}(x)\in \mathcal{F}(G,\mu, b)^{ss}.\] Therefore it suffices to show 
\[x\in \mathrm{Gr}_{G,\mu, b}^{a}\Longleftrightarrow \mathrm{BB}(x)\in \mathcal{F}(G,\mu, b)^{ss}.\]
When $G=\mathrm{GL}_n$, this is proved in \cite[Proposition 10.5.6]{FF}. The general case is reduced to the $\mathrm{GL}_n$ by choosing some suitable faithful representation of $G$ as done in the proof of \cite[Theorem 5.2]{Vi}.
\end{proof}

Now we are ready to prove the main result about classical points. 
\begin{proof}[Proof of Theorem \ref{thm_classical}] The strategy is the same as \cite[Proposition 5.5]{NV}. By Proposition \ref{prop_classical_HN=flag}, we have (2) is equivalent to (3). It remains to show that (1) is equivalent to (2). Suppose $[b']=\mathrm{HN}(b, x)$, it remains to show $\mathcal{E}_{b, x}\simeq \mathcal{E}_{b'}$. Let $M$ be the centralizer of $w_0\nu_{b'}$ and $P$ the corresponding parabolic subgroup of $G$. Let $b'_M$ be the reduction of $b'$ to $M$ such that $\nu_{b'_M}=w_0\nu_{b'}$. By Proposition \ref{prop_can reduction}, there exists a canonical reduction $(\mathcal{E}_{b,x})_P$ of $\mathcal{E}_{b,x}$ such that 
$\mathrm{pr}_M (x)$ is semi-stable and $\mathrm{HN}(b_M, \mathrm{pr}_M(x))=[b'_M]$. As $\mathrm{pr}_M(x)\in \mathrm{Gr}_M$ is again a classical point, by Proposition \ref{Prop_a=wa} it follows that
$(\mathcal{E}_{b, x})\times^P M\simeq \mathcal{E}_{b'_M}$. As the slope vector of $\mathcal{E}_{b'_M}$ is $G$-dominant, by \cite[Theorem 2.7, Corollary 2.9]{Ch}, $\mathcal{E}_{b, x}\simeq \mathcal{E}_{b'}$.
\end{proof}

\section{Basic properties of HN-strata}
In this section, we will discuss some basic properties of a HN-stratum, such as non-emptiness, dimension formula and classical points. By Section \ref{subsection_comp inner twists HN strat}, a HN-stratum associated to $(G, \mu, b)$ is always isomorphic to a HN-stratum associated to some $(H, \mu^H, b^H)$ with $H$ an inner twist of $G$. Therefore, without loss of generality, in this section we will always assume that $G$ is quasi-split. 

For a given Harder-Narasimhan stratum $\mathrm{Gr}_{G,\mu,b}^{\mathrm{HN}=[b']}$, let $M$ be the centralizer of $w_0\nu_{b'}$ and $P$ be the associated standard parabolic subgroup of $G$. Let $b'_M$ be the reduction of $b'$ to $M$ such that $\nu_{b'_M}=w_0\nu_{b'}$. By passing to the graded quotient of the canonical filtration, for any algebraically closed perfectoid field $C$ containing $\bar F$ with a bounded valuation subring $C^+$, we have the following inclusion
\[
\mathrm{Gr}_{G,\mu,b}^{\mathrm{HN}=[b']}(C)\subset \bigcup_{\textrm{reduction }(b_M,g) \atop \textrm{of }b \textrm{ to }M} \bigcup_{\lambda\in S_{M}(\mu) \textrm{ with }\atop \lambda=\kappa_M(b_M)-\kappa_M(b_M') \text{ in } \pi_1(M)_\Gamma} g\cdot \left(\mathrm{Gr}_{G,\mu}\cap \mathrm{Gr}_{P,\lambda}\right)(C),
\]
where $\mathrm{Gr}_{P, \lambda}$ and $S_M(\mu)$ are defined in \S~\ref{sec_semi-infinite orbits} and the conditions on $\lambda$ follows from Lemma \ref{lemma_reduction type} and the fact that 
\[\mathrm{Gr}_{M, b_M, \lambda}^{\mathrm{HN}=[b'_M]}\neq \emptyset\Rightarrow \kappa_M(b_M)-\lambda=\kappa_M(b'_M) \text{ in }\pi_1(M)_{\Gamma}.
\]
On the other hand, we have a projection map 
\[
\mathrm{Gr}_{P,\lambda}\longrightarrow \mathrm{Gr}_{M,\lambda}
\]
and we denote by $\mathrm{Gr}_{P,\lambda,b_M}^{ss}$ the preimage of the HN-semi-stable locus $\mathrm{Gr}_{M,\lambda,b_M}^{wa}\subset \mathrm{Gr}_{M,\lambda}$ through the above projection. 

\begin{proposition}\label{prop_fibration}  Keep the above notations. 
\begin{enumerate}
\item Let $C$ be an  algebraically closed perfectoid field containing $\bar F$ with a bounded valuation subring $C^+$. We have 
\[
\mathrm{Gr}_{G,\mu,b}^{\mathrm{HN}=[b']}(C)= \bigcup_{\textrm{reduction }(b_M,g) \atop \textrm{of }b \textrm{ to }M} \bigcup_{\lambda\in S_{M}(\mu) \textrm{ with }\atop \lambda=\kappa_M(b_M)-\kappa_M(b_M') \text{ in } \pi_1(M)_\Gamma}g\cdot \left(\mathrm{Gr}_{G,\mu}\cap \mathrm{Gr}_{P,\lambda,b_M}^{ss}\right)(C).
\]
\item Assume $\mu$ minuscule. Then the map 
\[
\mathrm{Gr}_{G,\mu}\cap \mathrm{Gr}_{P,\lambda,b_M}^{ss}\longrightarrow \mathrm{Gr}_{M,\lambda,b_M}^{wa}
\]
is (the diamond attached to) an affine fibration with fiber $N/N\cap P_{\lambda}$, where $N$ is the unipotent radical of $P$. 
\end{enumerate}
\end{proposition}

\begin{proof} (1) follows from Proposition \ref{prop_HN vector maximum of slope vector}. For (2), we need to study the map 
\[
\mathrm{Gr}_{G,\mu}\cap \mathrm{Gr}_{P,\lambda}\longrightarrow \mathrm{Gr}_{M,\lambda}
\]
for any $\lambda\in S_M(\mu)$, provided that $\mu$ is minuscule. Recall that, since $\mu$ is minuscule the Bialynicki-Birula map yields an isomorphism of v-sheaves
\begin{equation}\label{eq:BB}
\mathrm{BB}: \mathrm{Gr}_{G,\mu}\stackrel{\sim}{\longrightarrow}\mathcal{F}(G,\mu)^{\diamond}.
\end{equation}
In the level of geometric points, the above BB map  is given by 
\[
\mathrm{Gr}_{G,\mu}(C,C^ +)\ni g[\mu^{-1}] \longmapsto [\bar g]\in \mathcal F(G,\mu)^{\diamond}(C^{\flat},C^{\flat, +})=G(C)/P_{\mu}(C),
\] 
with $[\mu^ {-1}]:=\mu(\xi)^{-1}G(\BpdR)/G(\BpdR)\in \mathrm{Gr}_{G}(C,C^+)$.
We will relate $\mathrm{Gr}_{G,\mu}\cap \mathrm{Gr}_{P,\lambda}$ to some subspace of the flag variety $\mathcal F(G,\mu)$. Assume $\lambda=w\mu$ for some $w\in G(\bar F)$ normalizing $T$ and write $\mathcal F(G,\mu)^{w}$ the $P$-orbit of the element 
\[
[w]\in \mathcal F(G,\mu)=G/P_{\mu}.
\]
In particular, we have 
\[
\mathcal F(G,\mu)^{w}=P/P\cap P_{w\mu}\hookrightarrow \mathcal F(G,\mu)=G/P_{\mu}.
\]
We claim that that under the BB map \eqref{eq:BB} above, $\mathrm{Gr}_{G,\mu}\cap \mathrm{Gr}_{P,\lambda}\hookrightarrow \mathrm{Gr}_{G,\mu}$ corresponds to the sub-diamond 
\[
\mathcal F(G,\mu)^{w,\diamond}\hookrightarrow \mathcal F(G,\mu)^{\diamond}.
\]
Indeed, in the level of geometric points  
\[
\begin{array}{cl}
& \mathrm{Gr}_{G,\mu}(C,C^+)\cap \mathrm{Gr}_{P,\lambda}(C,C^+) \\ =&  G(\BpdR)\cdot [\mu^{-1}] \cap (N(\BdR)M(\BpdR)\cdot [\lambda^{-1}]) \\ =& P(\BpdR)w\cdot [\mu^{-1}] 
\end{array}
\] where the last equality holds by \cite[Proposition 13.1]{HKM}. 
Therefore under the identification \eqref{eq:BB}, $\mathrm{Gr}_{G,\mu}\cap \mathrm{Gr}_{P,w\mu}$ and $\mathcal F(G,\mu)^{w,\diamond}$ have the same geometric points. On the other hand, for a general affinoid perfectoid space $S$ over $\mathrm{Spd}(\bar F)$, a morphism $S\ra \mathrm{Gr}_{G,\mu}$ factors through $\mathrm{Gr}_{G,\mu}\cap \mathrm{Gr}_{P,\lambda}$ if and only if this is the case after restricting to every geometric points of $S$. Therefore, to conclude the proof of our claim,  it suffices to check that a morphism 
\[
\alpha:S^{\sharp}\longrightarrow \mathcal F(G,\mu)
\]
factors through the subspace 
\[
\mathcal F(G,\mu)^{w}\subset \mathcal F(G,\mu)
\]
if $\alpha(\bar x)\in \mathcal F(G,\mu)^w$ for every geometric points $\bar x$ of $S$. Let $U\subset \mathcal F(G,\mu)$ be an open subspace, which contains $\mathcal F(G,\mu)^w$ as a Zariski closed subspace. Because for every geometric point $\bar x$ of $S$, $\alpha(\bar x)\in \mathcal F(G,\mu)^w$, it follows that $\mathrm{Im}(\alpha)\subset U$. By \cite[Lemma II.2.2]{Sch1}, the pullback of $\mathcal F(G,\mu)^w\hookrightarrow U$ through the morphism $\alpha:S^{\sharp}\ra U$ defines a Zariski closed immersion $T\hookrightarrow S^{\sharp}$ of affinoid  perfectoid spaces, together with a commutative diagram 
\[
\xymatrix{T\ar[r]\ar[d]& \mathcal F(G,\mu)^w \ar[d]\\ S^ {\sharp}\ar[r]^<<<<<<{\alpha}& \mathcal F(G,\mu).} 
\]
Moreover, for a morphism $f:T'\ra S^{\sharp}$ of perfectoid spaces, the composed map $\alpha \circ f$ factors through $\mathcal F(G,\mu)^w$ if and only if $f$ factors through $T$ (cf. \cite[Remark II.2.3]{Sch1}).
Our assumption on $\alpha$ implies that the induced map $T(C,C^+)\ra S^{\sharp}(C,C^+)$ is bijective for all algebraically closed perfectoid field $C$ together with an open bounded valuation subring $C^+\subset C$. Hence by \cite[Lemma 5.4]{Sch2}, $T\stackrel{\sim}{\ra}S^{\sharp}$ and the morphism $\alpha$ factors through $\mathcal F(G,\mu)^{w}$, as wanted. This finishes the proof of our claim.      

On the other hand, the map $\mathrm{Gr}_{G,\mu}\cap \mathrm{Gr}_{P,\lambda,b_M}^{ss}\ra \mathrm{Gr}_{M,\lambda,b_M}^{wa}$ is the pullback of the open subdiamond $\mathrm{Gr}_{M,\lambda,b_M}^{wa}\subset \mathrm{Gr}_{M,\lambda}$ through the morphism $\mathrm{Gr}_{G,\mu}\cap \mathrm{Gr}_{P,\lambda}\ra \mathrm{Gr}_{M,\lambda}$, which by the claim above can be identified with the natural map
\[
\mathcal F(G,\mu)^{w,\diamond}=(P/P\cap P_{\lambda})^{\diamond}\longrightarrow \mathcal F(M,\lambda)^{\diamond}.
\]
So our proposition follows from the fact that the map $P/P\cap P_{\lambda}\ra M/M\cap P_{\lambda}$ is an affine fibration with fiber isomorphic to $N/N\cap P_{\lambda}$.  
\end{proof}

\subsection{Non-emptiness}

\begin{theorem}\label{thm_nonempty}Suppose $G$ is quasi-split. A HN-stratum $\mathrm{Gr}_{G, \mu,b}^{\mathrm{HN}=[b']}$ is non-empty if and only if there exists a $M$-dominant cocharacter $\lambda\in S_M(\mu)$ such that the generalized Kottwitz set (cf. ~\ref{section_Kottwitz} for the definition)
\[B(M, \lambda+\kappa_M(b'_M), \lambda^{\diamond}\nu_{b'_M})\]
contains a reduction $[b_M]$ of $b$ to $M$.


\end{theorem}

\begin{remark}(1) In the previous proposition, we may replace $S_M(\mu)$ by $\Sigma(\mu)_{M-max}$. Indeed, for any $\lambda\in S_M(\mu)$ we can always find $\lambda_{max}\in \Sigma(\mu)_{M-max}$ such that $\lambda\leq_M \lambda_{max}$. The result follows from the fact that
\[B(M, \lambda+\kappa_M(b'_M), \lambda^{\diamond}\nu_{b'_M})\subseteq B(M, \lambda_{max}+\kappa_M(b'_M), \lambda_{max}^{\diamond}\nu_{b'_M}).\]

(2) When $G=\mathrm{GL}_n$, an analogous result for the non-emptiness of Harder-Narasimhan strata of the flag varieties is obtained by Orlik \cite[Theorem 1]{Orl}. See Proposition \ref{prop_non_empty_HN_flag} for the precise version. 

(3) When $b$ is basic, the previous conditions are equivalent to the following: there exists a reduction $b_M$ of $b$ to $M$, and there exists a $M$-dominant cocharacter $\lambda\in S_M(\mu)$ such that
\[\kappa_M(b_M)-\lambda=\kappa_M(b'_M) \text{ in } \pi_1(M)_{\Gamma}.\] In particular, $[b']\in B(G, \kappa_G(b)-\mu, \nu_b\mu^{-1})$. This is compatible with the description of the non-emptiness of HN-strata for $b=1$ in \cite[Prop. 3.13]{NV}.
\end{remark}

\begin{lemma}\label{lemma_ss_nonempty}
$\mathrm{Gr}_{G, \mu, b}^{wa}$ is non-empty if and only if $\mu^{ad,\diamond}\geq \nu_b^{ad}$ in $G^{ad}$. In particular, when $[b]\in B(G, \mu)$,  $\mathrm{Gr}_{G, \mu, b}^{wa}$ is non-empty if and only if $\mu^{\diamond}\geq \nu_b$.
\end{lemma}
\begin{proof} By  \S\ref{subsection_adjoint}, via the morphism $\pi: \mathrm{Gr}_G\rightarrow \mathrm{Gr}_{G^{ad}}$ induced from the quotient $G\ra G^{ad}$, we have 
\[
\pi^{-1}(\mathrm{Gr}_{G^{ad}, \mu^{ad}, b^{ad}}^{wa})=\mathrm{Gr}_{G, \mu, b}^{wa}.\] It follows that $\mathrm{Gr}_{G, \mu, b}^{wa}$ is non-empty if and only if so is $\mathrm{Gr}_{G^{ad}, \mu^{ad}, b^{ad}}^{wa}$. So  we may assume that $G$ is adjoint.

Necessity. Suppose $\mathrm{Gr}_{G, \mu, b}^{wa}\neq \emptyset$. Choose $x\in \mathrm{Gr}_{G, \mu, b}^{wa}(C)$ for some algebraically closed perfectoid field $C$ containing $\bar F$. Let $b_{M_1}$ be a reduction of $b$ to a standard Levi subgroup $M_1$ such that $b_{M_1}$ is basic in $M$ and $\nu_{b_{M_1}}$ is $G$-anti-dominant. Let $P_1$ be the standard parabolic subgroup associated to $M_1$. By weak admissibility, for any  $\chi\in X^*(P_1/Z_G)^+$,
\begin{eqnarray}\label{eqn_wa}
\deg \chi_*((\E_{b, x})_P)=\deg \chi_*(\E_{b_{M_1}, \mathrm{pr}_{M_1}(x)})=\langle \chi, \lambda-\nu_{b_{M_1}}\rangle\leq 0,
\end{eqnarray}
where $\mathrm{pr}_{M_1}(x)\in \mathrm{Gr}_{M_1, \lambda}(C)$ for some $\lambda\in S_{M_1}(\mu)$ by Lemma \ref{lemma_reduction type}.

\emph{Claim: $\lambda_{G-dom}\geq \nu_b$}. 

The Claim implies that  $\mu\geq \lambda_{G-dom}\geq \nu_b$. Now it remains to prove the Claim. The inequality (\ref{eqn_wa}) implies that 
\[\mathrm{Av}_{M_1}(\lambda)-\nu_{b_{M_1}}=\mathrm{Av}_{M_1}(\lambda-\nu_{b_{M_1}})\leq 0, \] where $\mathrm{Av}_{M_1}: X_*(T)_{\Q}\rightarrow X_*(T)_{\Q}$ denotes the $W_{M_1}$-average, i.e., 
\[
\mathrm{Av}_{M_1}(\lambda)=\frac{1}{|W_{M_1}|}\sum_{w\in W_{M_1}} w\lambda.
\]
It follows that 
\[\nu_b=w_0\nu_{b_{M_1}}\leq w_0\mathrm{Av}_{M_1}(\lambda)=\mathrm{Av}_{w_0M_1}((w_0\lambda)_{w_0M_1-dom})\leq (w_0\lambda)_{w_0M_1-dom}\leq \lambda_{G-dom},
\]
as claimed.

Sufficiency. \cite[Theorem 3]{FR} shows that $\mathcal{F}(G, \mu, b)^{wa}=\mathcal{F}(G, \mu, b)^{ss}$ is non-empty if and only if $\nu_b\leq\mu^{\diamond}$. Then the result follows from Proposition \ref{prop_comparison via BB}.

\end{proof}

\begin{proof}[Proof of Theorem \ref{thm_nonempty}]Necessity. Suppose $x\in \mathrm{Gr}_{G, \mu, b}^{\mathrm{HN}=[b']}(C)$. By Proposition \ref{prop_can reduction}, there exists a reduction $(b_M, h)$ of $b$ to $M$ such that $\mathrm{pr}_M(h^{-1}x)\in \mathrm{Gr}_{M, \lambda, b}^{wa}(C)=\mathrm{Gr}_{M, \lambda, b}^{\mathrm{HN}=[b'_M]}(C)$ for some $\lambda\in S_M(\mu)$. Then $\kappa_M(b'_M)=\kappa_M(b_M)-\lambda^\sharp$ and $\nu_{b_M}^{ad}\leq_M \lambda^{ad,\diamond}$ by Lemma \ref{lemma_ss_nonempty}. 
It follows that $\nu_{b_{M}}\leq_M \lambda^{\diamond}\nu_{b'_M}$ and the result follows.

Sufficiency. Suppose $(b_M, h)$ and $\lambda$ are as in the conditions. Then again by Lemma \ref{lemma_ss_nonempty}, $\mathrm{Gr}_{M, \lambda, b}^{wa}$ is non-empty. Note that
$h\cdot \mathrm{pr}_M^{-1}(\mathrm{Gr}_{M, \lambda, b}^{wa})\subseteq \mathrm{Gr}_{G, \mu, b}^{\mathrm{HN}=[b']}$, it follows that the HN stratum is non-empty.
\end{proof}


We have analogous result for the non-emptiness of Harder-Narasimhan strata in the flag varieties. Let 
\[
S_M(\mu)_{cl}:=\{\lambda\in W\mu \mid \lambda \text{ is } M \text{-dominant} \}.
\] Obviously, $S_M(\mu)_{cl}\subset \Sigma(\mu)_{M-max}$.

\begin{proposition}\label{prop_non_empty_HN_flag}Suppose $G$ is quasi-split. A HN-stratum $\mathcal{F}(G, \mu,b)^{\mathrm{HN}=[b']}$ is non-empty if and only if there exists a $M$-dominant cocharacter $\lambda\in S_M(\mu)_{cl}$ such that the generalized Kottwitz set 
\[
B(M, \lambda+\kappa_M(b'_M), \lambda\nu_{b'_M})\]
contains a reduction $[b_M]$ of $b$ to $M$. Moreover, each non-empty HN-stratum has classical points.
\end{proposition}

\begin{proof}When $G=\mathrm{GL}_n$, this is proved by Orlik \cite{Orl}. For general $G$, the strategy is the same which is parallel to the proof of Theorem \ref{thm_nonempty}. In order to avoid the repetition, we only outline the strategy here.  Fontaine and Rapoport determines in \cite[Theorem 3]{FR} the condition when a semi-stable locus is non-empty (and contains classical points) . The result then follows from the fact that the HN-stratum is unions of some affine fibrations over the semi-stable locus of the HN-stratification of flag variety for some Levi subgroup of $G$ (cf. Proposition \ref{prop_fibration} and \cite[Proposition 7]{Orl}).  
\end{proof}

As an application of this result, we can determine which HN-strata contain classical points.

\begin{proposition}\label{prop_classical}Suppose $G$ is quasi-split. A HN-stratum $\mathrm{Gr}_{\mu,b}^{\mathrm{HN}=[b']}$ contains classical points if and only if there exists a $M$-dominant cocharacter $\lambda\in S_M(\mu)_{cl}$ such that the set 
\[
B(M, \lambda+\kappa_M(b'_M), \lambda^{\diamond}\nu_{b'_M})
\]
contains a reduction $[b_M]$ of $b$ to $M$.
\end{proposition}
\begin{proof}This follows from Proposition \ref{prop_non_empty_HN_flag} combined with Theorem \ref{thm_classical}. 
\end{proof}

\begin{remark}\begin{enumerate} \item 
When $b$ is basic, this proposition is proved by Viehmann \cite[Theorem 5.5]{Vi}. Indeed, Viehmann determines which Newton-strata contain classical points. It's the same as which HN-strata contain classical points by Theorem \ref{thm_classical}.

\item When $\mu$ is minuscule, compared with Theorem \ref{thm_nonempty}, we see that every non-empty HN-stratum contains classical points.
\end{enumerate}
\end{remark}

\begin{example}Let $G=\mathrm{GL}_3$, $\nu_b=(\frac{5}{2}^{(2)}, 0)$ and $\mu=(3, 1, 1)$. There are 4 non-empty HN-strata with HN-vector $\nu_{b'}$ of the following form
\[
\left(\frac{1}{2}, \frac{1}{2}, -1\right),\  (1, 1, -2),\  \left(\frac{3}{2}, \frac{3}{2}, -3\right),\ (0, 0, 0).\] Except the HN-stratum corresponding to $(1, 1, -2)$, the other 3 non-empty HN-strata all contain classical points.
\end{example}

\subsection{Dimension formula}
Now we consider the dimension of a non-empty HN-stratum $\mathrm{Gr}_{\mu,b}^{\mathrm{HN}=[b']}$. Let $S_{\mu, b, b'}$ be the set of $\lambda\in S_M(\mu)$ which satisfies
\[
[b_M]\in B(M, \lambda^{\sharp}+\kappa_M(b'_M), \lambda^{\diamond}\nu_{b'_M}),
\] for some reduction $b_M$ of $b$ to $M$. The set $S_{\mu, b, b'}$ is non-empty by Theorem \ref{thm_nonempty}. We give the dimension formula of this HN-stratum. Here for $X$ a locally spatial diamond or an adic space, its \emph{dimension}, written 
\[
\dim X,
\]
is defined to be the dimension $\dim(|X|)$ of the associated locally spectral space $|X|$, that is, the supremum of all integers $n$ for which there exist a chain $x_0,x_1,\ldots, x_n\in |X|$ of distinct points in $|X|$ such  that $x_i$ is a a specialization of $x_{i+1}$ for $i=0,\ldots, n-1$ (cf. the paragraph after Remark 21.8 of \cite{Sch2} for the case of locally spatial diamonds and \cite[\S~1.8]{Hub} for the case of adic spaces). We remark that, for $T$ a locally spectral space, if $T=\cup_i T_i$ is a union by closed subsets of $T$, we have
\[
\dim T=\mathrm{max}_i\dim T_i.
\]
Furthermore, if $X=Z^{\diamond}$ is the diamond associated with an analytic adic space $Z$ over $\mathbb Z_p$ we have $|X|=|Z|$ (\cite[Lemma 15.6]{Sch2}) and thus $\dim X=\dim Z$.

\begin{theorem}\label{thm_dimension} Notations as above. Assume that $\mu$ is minuscule, then
\[\mathrm{dim} \mathrm{Gr}_{G, \mu,b}^{\mathrm{HN}=[b']}=\langle\mu, \rho\rangle+ \mathrm{max}_{\lambda\in S_{\mu, b, b'}} \langle \lambda, 2\rho_M-\rho\rangle.\]
\end{theorem}

\begin{remark}\begin{enumerate}
\item 
We expect that the dimension formula still holds for a non-minuscule cocharacter $\mu$ although we can't prove it for the moment.

\item Let $S_{\mu, b, b'}^{op}:=\{-w_{M, 0}\lambda\mid \lambda\in S_{\mu, b, b'}\}$. Replacing $S_{\mu, b, b'}$ by $S_{\mu, b, b'}^{op}$, we can write the dimension formula in Theorem \ref{thm_dimension} in a simpler form:

\[\mathrm{dim} \mathrm{Gr}_{G, \mu,b}^{\mathrm{HN}=[b']}= \mathrm{max}_{\lambda\in S_{\mu, b, b'}^{op}} \langle\mu+ \lambda, \rho\rangle.\]

\item When $b$ is basic, the set $S_{\mu, b, b'}$ consists of the elements $\lambda\in S_M(\mu)$ such that
\[\lambda^{\sharp}=\kappa_M(b_M)-\kappa_M(b'_M) \text{ in } \pi_1(M)_{\Gamma}.\] In particular, $\mathrm{Av}_M(\lambda)=\nu_b-w_0\nu_{b'}$. Hence, the dimension formula could also be written in the following form:
\[\mathrm dim \mathrm{Gr}_{G, \mu,b}^{\mathrm{HN}=[b']}=\langle \mu-\nu_{b'}, \rho\rangle+ \mathrm{max}_{\lambda\in S_{\mu, b, b'}} \langle \lambda, \rho_M\rangle.
\]
Recall that Nguyen-Viehmann defined in \cite[3.17(iii)]{NV} the set $\Theta(\mu, b')$ of Harder-Narasimhan types consisting of $G(F)$-conjugacy classes of Harder-Narasimhan pairs of the form $(Q, \{\lambda\})$ where $Q$ is a parabolic subgroup which is conjugate to the centralizer of $\nu_{b'}$ and $\{\lambda\}$ is a $Q(\bar F)$-conjugacy class of cocharacters of $Q_{\bar F}$ such that $[b']$ has a reduction $[b'_{M_Q}]$ to some Levi factor $M_Q$ of $Q$ that is basic in $M_Q$, and such that 
\[\quad
-\lambda^{\sharp_{M_Q}}=\kappa_{M_Q}(b')\in \pi_1(M_Q)_\Gamma\quad \textrm{and}\quad (-\lambda)_{G-dom}\leq \mu_{G-dom}.
\] 
We may omit $Q$ in the pairing $(Q, \{\lambda\})$ for $\Theta(\mu, b')$. When $\mu$ is minuscule, it's easy to check that there is an identification 
$S_{\mu, 1, b'}^{op}=\Theta(\mu, b'^*)$,
where $[b'^*]$ denotes the element in $B(G)$ such that $\nu_{b'^*}=-w_0\nu_{b'}$ and $\kappa(b'^*)=-\kappa(b')$ (cf. \cite[Corollary 2.9]{CT}). Therefore when $b=1$, the dimension formula could also be written in the way that 
\[\mathrm{dim} \mathrm{Gr}_{G,\mu,1}^{\mathrm{HN}=[b']}=\mathrm{max}_{\lambda\in \Theta(\mu, b')} \langle\mu+ \lambda, \rho\rangle\]from which we could see  that the upper bound of the dimension of a  HN-stratum given by Nguyen-Viehmann (\cite[Proposition 7.1]{NV}) is actually reached. This also answers the question in \cite[Remark 7.2(2)]{NV}.

\item When $G$ is split and $b$ is basic, the set $S_{\mu,b, b'}$ consists of a unique element. It follows that in this situation, we don't need to take maximum in the dimension formula.

\item When $b$ is basic, the dimension of a HN-stratum is already studied by many people. For $G=\mathrm{GL}_n$, the dimension formula is proved by Fargues in \cite[Proposition 23]{Far}.  In \cite[Theorem 3.9]{Sh}, Shen shows that any non-basic HN-strata are parabolic induction by introducing a finer decomposition on a given HN-stratum. As a byproduct, he also gives a description of the dimension as the maximum of the dimension of each stratum \cite[Remark 3.10]{Sh}. However, we want to clarify that our dimension formula is not the same as Shen's. We have mentioned that when $G$ is split, the maximum disappears in our dimension formula while it always exists in Shen's even for $G=\mathrm{GL}_n$.  In \cite[Proposition 7.1, 7.3]{NV}, Nguyen-Viehmann gives some upper bound for the dimension formula.   
\end{enumerate}
\end{remark}

\begin{proof}[Proof of Theorem \ref{thm_dimension}]
By Proposition \ref{prop_fibration} combined with Lemma \ref{lemma_ss_nonempty}, we have that
\[
\left|\mathrm{Gr}_{G,\mu,b}^{\mathrm{HN}=[b']}\right|= \bigcup_{\textrm{reduction }(b_M,g) \atop \textrm{of }b \textrm{ to }M} \bigcup_{\lambda\in S_{\mu, b, b'} }g\cdot \left|\left(\mathrm{Gr}_{G,\mu}\cap \mathrm{Gr}_{P,\lambda,b_M}^{ss}\right)\right|.
\] 
Note that each $|\mathrm{Gr}_{G,\mu}\cap \mathrm{Gr}_{P,\lambda,b_M}^{ss}|$ in the union above is a dense open subset of $|\mathrm{Gr}_{G,\mu}\cap \mathrm{Gr}_{P,\lambda}|$, and the latter is itself a locally closed subset of $|\mathrm{Gr}_{G,\mu}|\simeq |\mathcal F(G,\mu)|$, it follows that the dimension of $|\mathrm{Gr}_{G,\mu}\cap \mathrm{Gr}_{P,\lambda,b_M}^{ss}|$ is the same as the dimension of its closure $\overline{|\mathrm{Gr}_{G,\mu}\cap \mathrm{Gr}_{P,\lambda,b_M}^{ss}|}$ in $|\mathrm{Gr}_{G,\mu}|$. Consequently, 
\[
\left|\mathrm{Gr}_{G,\mu,b}^{\mathrm{HN}=[b']}\right|= \bigcup_{\textrm{reduction }(b_M,g) \atop \textrm{of }b \textrm{ to }M} \bigcup_{\lambda\in S_{\mu, b, b'} }g\cdot \left(\overline{ \left|\left(\mathrm{Gr}_{G,\mu}\cap \mathrm{Gr}_{P,\lambda,b_M}^{ss}\right)\right|}\cap \left|\mathrm{Gr}_{G,\mu,b}^{\mathrm{HN}=[b']}\right|\right).
\] 
and hence we have 
\[\begin{split}\dim\mathrm{Gr}_{G,\mu, b}^{\mathrm{HN}=[b']}
&=\mathrm{max}_{(b_M, g)\atop \lambda\in S_{\mu, b, b'}} \dim \left(\overline{ \left|\left(\mathrm{Gr}_{G,\mu}\cap \mathrm{Gr}_{P,\lambda,b_M}^{ss}\right)\right|}\cap \left|\mathrm{Gr}_{G,\mu,b}^{\mathrm{HN}=[b']}\right|\right)\\  &=\mathrm{max}_{(b_M, g)\atop \lambda\in S_{\mu, b, b'}} \dim \mathrm{Gr}_{G,\mu}\cap \mathrm{Gr}_{P,\lambda,b_M}^{ss} \\
&=\mathrm{max}_{(b_M, g)\atop \lambda\in S_{\mu, b, b'}} \dim N/N\cap P_{\lambda} + \dim (\mathrm{Gr}_{M, \lambda, b_M}^{wa})\\
&=\mathrm{max}_{\lambda\in S_{\mu, b, b'}} \dim N/N\cap P_{\lambda} + \langle \lambda, 2\rho_M\rangle,
\end{split}
\]
where the third equality follows from Proposition \ref{prop_fibration} (2) and the last equality follows from the fact that $\mathrm{Gr}_{M, \lambda, b_M}^{wa}$ is the semi-stable HN-stratum and hence open in $\mathrm{Gr}_{M, \lambda}$. So
\[
\dim\mathrm{Gr}_{M, \lambda, b_M }^{wa}=\dim \mathrm{Gr}_{M, \lambda}=\langle \lambda, 2\rho_M\rangle.
\] is independent of the choice of $(b_M, g)$.
Therefore it remains to show that 
\[
\dim N/ N\cap P_{\lambda}=\langle \mu-\lambda, \rho\rangle.
\]
Now we compute this dimension:  
\[\begin{split}\dim N/N\cap P_{\lambda}
=&\sum_{\alpha\in\Phi_G^+-\Phi_M^+\atop s.t. \langle \lambda, \alpha\rangle<0} -\langle\lambda, \alpha\rangle\\
=&\sum_{\alpha\in\Phi_G^+-\Phi_M^+\atop s.t. \langle \lambda, \alpha\rangle<0} -\frac{1}{2}\langle\lambda, \alpha\rangle + \sum_{\alpha\in\Phi_G^+-\Phi_M^+\atop s.t. \langle \lambda, \alpha\rangle\geq 0} \frac{1}{2}\langle\lambda, \alpha\rangle-\langle \lambda, \rho_N\rangle \\
=&\langle \mu, \rho\rangle -\langle \lambda, \rho_M\rangle
-\langle \lambda, \rho_N\rangle\\
=&\langle \mu-\lambda, \rho\rangle.
\end{split}\]

\end{proof}
\bibliographystyle{amsalpha}

\begin{thebibliography}{alpha}


\bibitem[BL]{BL} A. Beauville, Y. Laszlo, \textsl{Un lemme de descent}, Comptes Rendus de l’Acad\'emie des Sciences-S\'erie I-Math\'ematique 320.3 (1995), pp. 335–340.


\bibitem[BH99]{BH} M. R. Bridson, A. Haefliger, \textsl{Metric spaces of non-positive curvature}, volume 319 of Grundlehren der Mathematischen Wissenschaften, Springer-Verlag, Berlin, 1999


\bibitem[BT84]{BT} F. Bruhat, J. Tits, \textsl{Sch\'emas en groupes et immeubles des groupes classiques sur u corps local}, Bulletin de la S. M. F., tome 112 (1984), 259-301



\bibitem[Che23]{Ch}M. Chen, \textsl{Fargues-Rapoport conjecture for $p$-adic period domains in the non-basic case},  Journal of European Mathematical Society,  25 (2023), no. 7, 2879–2918.

\bibitem[CFS21]{CFS}M. Chen, L. Fargues, X. Shen, \textsl{On the structure of some $p$-adic period domains}, Cambridge Journal of Mathematics, 9 (2021), 213-267.

\bibitem[CT]{CT}M. Chen, J. Tong, \textsl{Weakly admissible locus and Newton stratification in $p$-adic Hodge theory}, to appear in American Journal of Mathematics. 

\bibitem[CF00]{CF} P. Colmez, J.-M. Fontaine, \textsl{Construction des repr\'esentations $p$-adiques semi-stables}, Invent. Math. 140 (2000), 1-43.

\bibitem[Cor18]{Cor} C. Cornut, \textsl{On Harder-Narasimhan filtration and their compatibility with tensor products}, Confluentes Mathematici, Tome 10, n. 2 (2018), 3-49.

\bibitem[Cor20]{Cor20} C. Cornut, \textsl{Filtrations and Buildings}, Memoirs of the American Mathematical Society, Volume 266, Number 1296 (2020).

\bibitem[CPI19]{CPI} C. Cornut, M. Peche Irissarry, \textsl{Harder-Narasimhan filtrations for
Breuil-Kisin-Fargues modules}, Ann. H. Lebesgue, 2 (2019), 415–480.

\bibitem[CS17]{CS}A. Caraiani, P. Scholze, \textsl{On the generic part of the cohomology of compact unitary Shimura varieties},  Ann. Math.  186 (2017), no. 3, 649-766.

\bibitem[DOR10]{DOR}J.-F. Dat, S. Orlik, M. Rapoport, \textsl{Period domains over finite and $p$-adic fields}, volume 183 of Cambridge Tracts in Mathematics, Cambridge University Press, Cambridge, 2010.

Vol. 900. Lecture Notes in Mathematics, Springer, 1982, 101–228.

\bibitem[Fal94]{Fa} G. Faltings, \textsl{Mumford-Stabilit\"at in der algebraischen Geometrie}, in
the 1994 ICM proceedings.

\bibitem[Fal10]{Fal1} G. Faltings, \textsl{Coverings of $p$-adic period domains}, J. Reine Angew. Math. 643 (2010), 111-139.

\bibitem[FF18]{FF} L. Fargues, J.-M. Fontaine, \textsl{Courbes et fibr\'es vectoriels en th\'eorie de Hodge $p$-adique}, Ast\'erisque 406, Soc. Math. France, 2018.

\bibitem[FS]{FS} L. Fargues, P. Scholze, \textsl{Geometrization of the local Langlands correspondence}, preprint.

\bibitem[Far1]{Far1}L. Fargues, \textsl{Geometrization of the local Langlands correspondence:an overview}, preprint.

\bibitem[Far]{Far} L. Fargues, \textsl{Th\'eorie de la r\'eduction pour les groupes $p$-divisibles}, preprint.

\bibitem[Far20]{Far20} L. Fargues, \textsl{$G$-torseurs en th\'eorie de Hodge $p$-adique}, Compositio Math. 156 (2020), no. 10, 2076-2110.

\bibitem[FR05]{FR}J.-M. Fontaine, M. Rapoport, \textsl{Existence de filtrations admissibles sur des isocristaux}, Bull. Soc. Math. France 133 (2005), no. 1, 73–86.

\bibitem[GHKR06]{GHKR} U. G\"{o}rtz, Th. J. Haines, R. E. Kottwitz, D. C. Reuman, \textsl{Dimensions of some affine Deligne-Lusztig varieties}, Ann. Sci. École Norm. Sup. 39 (2006), no. 3, 467–511.

\bibitem[Han]{Han} D. Hansen, \textsl{On the supercuspidal cohomology of basic local Shimura varieties}, to appear in J. reine angew. Math.

\bibitem[Har08]{Har1}U. Hartl, \textsl{On period spaces for $p$-divisible groups}, C. R. Math. Acad. Sci. Paris
346 (2008), no. 21-22, 1123-1128.

\bibitem[Har13]{Har} U. Hartl, \textsl{On a conjecture of Rapoport and Zink}, Invent. Math. (2013) 193, 627-696.

\bibitem[HKM12]{HKM} T. Haines,  M. Kapovich, J. Millson, \textsl{Ideal triangles in Euclidean buildings and branching to Levi subgroups}, J. Algebra,  361(2012), 41-78.


\bibitem[Hub96]{Hub} R. Huber, \textsl{\'Etale cohomology of rigid analytic varieties and adic spaces}, Aspects of Mathematics, E30, Friedr. Vieweg \& Sohn, Braunschweig, 1996.


\bibitem[KL15]{KL} K. S. Kedlaya, R. Liu, \textsl{Relative $p$-adic Hodge theory: Foundations}, Ast\'erisque 371, Soc. Math. France, 2015.

\bibitem[Kot85]{Kot1}R.E. Kottwitz, \textsl{Isocrystals with additional structure}, Compositio Math., 56 (1985), 201-220.


\bibitem[Kot97]{Kot} R. Kottwitz, \textsl{Isocrystals with additional structure II}, Compositio Math., 109 (1997), 255-339.

\bibitem[MV07]{MV} I. Mirkovi\'c  and K. Vilonen, \textsl{Geometric Langlands duality and representations of algebraic groups over commutative rings}, Ann. of Math. (2) 166 (2007), no. 1, 95–143.

\bibitem[NV23]{NV} K. H. Nguyen, E. Viehmann, \textsl{A Harder-Narasimhan stratification of the
$B^+_{\rm dR}$-affine Grassmannian}, Compositio Math. 159 (2023), 711-745.

\bibitem[Ngu]{Ngu}K.H. Nguyen, \textsl{On categorical local Langlands program for $\mathrm{GL}_n$}, preprint, arxiv:2309.16505v2.

\bibitem[Orl06]{Orl} S. Orlik. \textsl{On Harder-Narasimhan strata in flag manifolds}, Math. Z., 252 (2006), 209-222.

\bibitem[Rap18]{Ra} M. Rapoport, \textsl{Accessible and weakly accessible period domains}, Appendix of \textsl{On the $p$-adic cohomology of the Lubin-Tate tower} by Scholze, Ann. Sci. École Norm. Sup. 51 (2018), 856-863.

\bibitem[RR96]{RR}M. Rapoport, M. Richartz, \textsl{On the classification and specialization of $F$-isocrystals with additional structure}, Compositio Math. 103 (1996), no. 2, 153-181.

\bibitem[RZ96]{RZ}M. Rapoport, T. Zink, \textsl{Period spaces for $p$-divisible groups}, Ann. of Math. Stud. 141,
Princetion Univ. Press, 1996.

\bibitem[Schi15]{Sch} S. Schieder, \textsl{The Harder-Narasimhan stratification of the moduli stack of $G$-bundles via Drinfeld's compatifications}, Sel. Math. New Ser. 21 (2015), 763-831.


\bibitem[Scho15]{Sch1} P. Scholze, \textsl{On torsion in the cohomology of locally symmetric varieties}, Ann. of Math. (2) \textbf{182} (2015), no. 3, 945-1066.

\bibitem[Scho18]{Sch2} P. Scholze, \textsl{\'Etale cohomology of diamonds}, preprint 2018. 

\bibitem[SW20]{SW}P. Scholze, J. Weinstein, \textsl{Berkeley lectures on $p$-adic geometry}, Annals of Mathematics Studies, 207,  Princeton University Press, 2020.

\bibitem[Sh23]{Sh} X. Shen, \textsl{Harder-Narasimhan strata and $p$-adic period domains}, Trans. Amer. Math. Soc. 376 (2023), no. 5, 3319–3376.

\bibitem[Spr10]{Spr} T. A. Springer, \textsl{Linear algebraic groups}, Second edition, Modern Birkh\"auser Classics, Birkhäuser Boston, MA, 2010.

\bibitem[Tot94]{To1} B. Totaro, \textsl{Tensor products of semi-stables are semi-stable}, in Geometry and
Analysis on Complex Manifolds, World Scientific (1994), 242-250.

\bibitem[Tot96]{To2} B. Totaro, \textsl{Tensor products in p-adic Hodge theory}, Duke Math. J., 83(1996):79-104.


\bibitem[Vie]{Vi} E. Viehmann, \textsl{On Newton strata in the $\BpdR$-Grassmannian},  to appear in Duke Mathematical Journal.


\bibitem[Zhu17]{Zhu} X. Zhu, \textsl{An introduction to affine Grassmannians and the geometric Satake equivalence}, Geometry of moduli spaces and representation theory, IAS/Park City Math. Ser., vol. 24, Amer. Math. Soc., Providence, RI, 2017, pp. 59–154.


\end{thebibliography}

\end{document}